\documentclass[12pt, reqno]{amsart}
\usepackage[table]{xcolor}
\usepackage{amsmath,amssymb, amsfonts,amscd,psfrag,graphicx,stmaryrd}
\usepackage{epsfig}
\usepackage{amsthm}
\usepackage{fullpage}
\usepackage{verbatim}
\usepackage{listings}
\usepackage[table]{xcolor} %!color
\usepackage{mathtools}
\DeclarePairedDelimiter{\ceil}{\lceil}{\rceil}
\DeclarePairedDelimiter\floor{\lfloor}{\rfloor}
\newcommand{\lebn}

\usepackage[misc]{ifsym}
\usepackage{caption}
\usepackage{float}
\usepackage{amsmath}
\usepackage{subfigure}

\theoremstyle{plain}
\newtheorem{proposition}[equation]{Proposition}

\newtheorem{theorem}[equation]{Theorem}
\newtheorem{conjecture}[equation]{Conjecture}

\newtheorem{corollary}[equation]{Corollary}
\newtheorem{lemma}[equation]{Lemma}

\theoremstyle{definition}

\newtheorem{remark}[equation]{Remark}

\numberwithin{equation}{section}

\newcommand{\D}{\Delta}

\usepackage[a4paper,left=1.5cm,right=1.5cm,top=2cm,bottom=2cm]{geometry}

\allowdisplaybreaks

\usepackage{tikz}
\usepackage{verbatim}
\usetikzlibrary{shapes,shadows,calc}
\usepgflibrary{arrows}
\usetikzlibrary{arrows, decorations.markings, calc, fadings, decorations.pathreplacing, patterns, decorations.pathmorphing, positioning}
\tikzset{nodc/.style={circle,draw=blue!50,fill=pink!80,inner sep=1.6pt}}
\tikzset{nodr/.style={circle,draw=black,fill=green!50!black,inner sep=2.5pt}}
\tikzset{nodel/.style={circle,draw=black,inner sep=2.2pt}}
\tikzset{nodinvisible/.style={circle,draw=white,inner sep=2pt}}
\tikzset{nodpale/.style={circle,draw=gray,fill=gray,inner sep=1.6pt}}

\tikzset{nod1/.style={circle,draw=black,fill=black,inner sep=1pt}}
\tikzset{nod2/.style={circle,draw=black,fill=blue!75!black,inner sep=1.6pt}}
\tikzset{nodgs/.style={circle,draw=black,dotted,fill=gray,inner sep=1.6pt}}
\tikzset{nod3/.style={circle,draw=black,fill=black,inner sep=1.8pt}}
\tikzset{noddiam/.style={diamond,draw=black,inner sep=2pt}}
\tikzset{nodw/.style={circle,draw=black,inner sep=1.8pt}}

\usepackage[colorlinks, urlcolor=blue, linkcolor=blue, citecolor=blue]{hyperref}

\makeatletter
 \def\@textbottom{\vskip \z@ \@plus 10pt}
 \let\@texttop\relax
\makeatother

\usepackage{enumerate}  

\usepackage{enumitem}
\setitemize{itemsep=2.5pt}

\begin{document}

\bibliographystyle{plain}

\title[2-distance coloring of planar graphs]{A $2$-distance $(2\D+7)$-coloring of planar graphs}

\author{Zakir Deniz}
%\thanks{The author is supported by T\" UB\. ITAK, grant no:122F250}
\address{Department of Mathematics, D\"uzce University, D\"uzce, 81620, T{\"{u}}rk{\.{\.i}}ye.} 
\email{zakirdeniz@duzce.edu.tr}

\date{\today}
\thanks{}
\subjclass[2010]{}

\begin{abstract}

A vertex coloring of a graph $G$ is called a $2$-distance coloring if any two vertices at a distance at most $2$ from each other receive different colors.  
Recently, Bousquet et al. (Discrete Mathematics, 346(4), 113288, 2023)   proved that $2\D+7$ colors are sufficient for the $2$-distance coloring of planar graphs with maximum degree $\D\geq 9$. In this paper, we strengthen their result by removing the maximum degree constraint and show that all planar graphs admit a 2-distance $(2\D+7)$-coloring.
This particularly improves the result of  Van den Heuvel  and McGuinness (Journal of Graph Theory, 42(2), 110-124, 2003).
\end{abstract}
\maketitle

%%%%%%%%%%%%%%%%%%%%%%%%%%%%%%%%%%%%%%%%%%%%%%%%%%%%%%%%%%%%
\section{Introduction}

Among the numerous problems and concepts associated with graphs, coloring stands out as a fundamental and extensively studied topic.
A vertex coloring of a graph assigns colors to its vertices so that adjacent vertices receive distinct colors. One particularly interesting variant of vertex coloring is 2-distance coloring, where any two vertices at a distance of $2$ or less have distinct colors. Its motivation arises from the necessity in some real-life problems to assign different colors not only to adjacent vertices but also to those in close proximity \cite{krumke,cormic}.
The concept of 2-distance coloring was first introduced in \cite{kramer-1, kramer-2}, and it has received considerable attention due to the fact that certain problems, such as the Total Coloring Conjecture (see \cite{daniel} for details), can be formulated as a 2-distance coloring of specific graphs.
A comprehensive survey on 2-distance coloring and its related coloring concepts has been presented by Cranston \cite{daniel}.

The smallest number of colors for which graph $G$ admits a 2-distance coloring is known as the 2-distance chromatic number $\chi_2(G)$ of $G$. 
In 1977, Wegner \cite{wegner} proposed the following conjecture.

\begin{conjecture}\label{conj:main}
For every planar graph $G$ with maximum degree $\D$, $\chi_2(G) \leq 7$ if $\Delta=3$, \linebreak $\chi_2(G) \leq \Delta+5$ if $4\leq \Delta\leq 7$, and $\chi_2(G) \leq  \floor[\big]{\frac{3\Delta}{2}}+1$ if $\Delta\geq 8$.
\end{conjecture}

Wegner's Conjecture is one of the most popular problems in graph coloring, and it has remained open for many years, except for the case when $\D=3$, which was solved by Thomassen \cite{thomassen} (independently by Hartke et al. \cite{hartke}). The conjecture is also known to be asymptotically true for large values of $\D$, as shown in \cite{amini,havet}.  
For general planar graphs,  Van den Heuvel and McGuinness \cite{van-den} showed that $\chi_2(G) \leq 2\Delta + 25$, while the bound $\chi_2(G) \leq \ceil[\big]{ \frac{5\D}{3}}+78$ was proved by Molloy and Salavatipour \cite{molloy}.
On the other hand, some improved results are presented in \cite{bosquet-max4,bu-zhu-g6,deniz-max5, deniz-g5, deniz-g6} with a certain degree or girth restrictions. In particular, Bousquet et al. \cite{bosquet}   recently showed that $\chi_2(G) \leq 2\D+7$ if $\D\geq 9$,  which significantly improves upon the work in \cite{kry}.

In this paper, we strengthen the result of Bousquet et al. \cite{bosquet} by removing the maximum degree constraint and show that all planar graphs admit a 2-distance $(2\D+7)$-coloring, which particularly improves the best-known bound of $2\Delta+25$ provided by  Van den Heuvel  and McGuinness \cite{van-den}.

\begin{theorem}\label{thm:main}
For every planar graph $G$, we have $\chi_2(G) \leq 2\D+7$. 
\end{theorem}

All graphs in this paper are assumed to be simple, and we refer to \cite{west} for terminology and notation not defined here. When $G$ is a graph, we use $V(G),E(G),F(G),$ and $\D(G)$ to denote the vertex, edge and face set, and the maximum degree of $G$, respectively. If there is no confusion in the context, we abbreviate $\D(G)$ to $\D$. For a positive integer $t$, we denote by  $[t]$ the  set of integers between $1$ and $t$. Given a planar  graph $G$, %Euler's formula states $|V|-|E|-|F|=2$ where $|V|$, $|E|$, $|F|$  are the number of vertices, edges, faces, respectively. 
we denote by $\ell(f)$ the \emph{length} of a face $f$ and by $d(v)$ the \emph{degree} of a vertex $v$. 
A \emph{$k$-vertex} is a vertex of degree $k$. A \emph{$k^{-}$-vertex} is a vertex of degree at most $k$ while a \emph{$k^{+}$-vertex} is a vertex of degree at least $k$. A \emph{$k$ ($k^-$ or $k^+$)-face} is defined analogously as for the vertices, where the \emph{degree} of a face is the number of edges incident to it. A vertex $u\in N(v)$ is called \emph{$k$-neighbour} (resp. \emph{$k^-$-neighbour}, \emph{$k^+$-neighbour}) of $v$ if $d(u)=k$ (resp. $d(u)\leq k$, $d(u)\geq k$).  %When $r\leq d(v)\leq s$, the vertex $v$ is said to be $(r\text{-}s)$-vertex.

For a vertex $v\in V(G)$, we use $n_i(v)$ to denote the number of $i$-vertices adjacent to $v$. %Let $v\in V(G)$, we define $D(v)=\Sigma_{v_i\in N(v)}d(v_i)$. 
We denote by $d(u,v)$ the \emph{distance} between $u$ and $v$ for a pair $u,v\in V(G)$. Also, we set $N_i(v):=\{u\in V(G) \ | \  1 \leq d(u,v) \leq i \}$ for $i\geq 1$, so $N_1(v)=N(v)$ and let $d_2(v)=|N_2(v)|$.
For $v \in V(G)$, we use $m_k(v)$ to denote the number of $k$-faces incident with $v$.
%An $(x_1,x_2,\ldots,x_k)$-face is a $k$-face with vertices of degrees $x_1,x_2,\ldots,x_k$.
A $k$-vertex $v$ with $m_3(v)=d$ is called \emph{$k(d)$-vertex}. In particular, a $k$-vertex $v$ is called  \emph{$k(d^-)$-vertex}   (resp. \emph{$k(d^+)$-vertex}) if  $m_3(v)\leq d$ (resp. $m_3(v)\geq d$). Two faces $f_1 $ and $f_2$ are said to be \emph{adjacent} if they share a common edge.

% Let $f$ be a $3$-face. If all vertices on $f$ are $5$-vertices, then we call $f$ as a \emph{good $3$-face}. If one vertex on $f$ is $5$-vertex (resp. $4$-vertex) and the other two vertices are $4$-vertices (resp. $5$-vertices), then we call $f$ as a \emph{bad $3$-face} (resp. \emph{semi-bad $3$-face}).

\section{The Proof of Theorem \ref{thm:main}} 

Let $G$ be a minimal counterexample to Theorem~\ref{thm:main} such that 
$|V(G)| + |E(G)|$ is minimum.  
Thus, $G$ does not admit any $2$-distance $(2\Delta + 7)$-coloring, 
whereas every planar graph $G'$ obtained from $G$ with 
$|V(G')| + |E(G')| < |V(G)| + |E(G)|$ admits a $2$-distance $(2\Delta + 7)$-coloring.  
By the results of \cite{bosquet, deniz-max5}, it suffices to consider the case 
$6 \le \Delta \le 8$.  
Clearly, $G$ is connected.

%Moreover, we have the following.
%
%\begin{lemma}[\cite{bosquet}]\label{lem:2-connected}
%The graph $G$ is $2$-connected.
%\end{lemma}

\smallskip

We call a graph $H$ \emph{proper} with respect to $G$ if $H$ is obtained from $G$ 
by deleting some edges or vertices and adding some edges, in such a way that 
every pair of vertices $x_1, x_2 \in V(G) \cap V(H)$ that have distance at most $2$ in $G$ 
also have distance at most $2$ in $H$.  
If $f$ is a $2$-distance coloring of such a graph $H$, then $f$ can be extended to a 
$2$-distance coloring of $G$, provided that each remaining uncolored vertex has an 
available color.  
By the minimality of $G$, we obtain the following.

\begin{remark}\label{rem:proper}
If $H$ is proper with respect to $G$ such that $ V(G)\setminus V(H) =\{v\}$, then  $d_2(v) \geq  2\Delta + 7$.
\end{remark}

\begin{proof}
Assume, for a contradiction, that there exists a vertex \(v \in V(G) \setminus V(H)\) 
such that \(d_2(v) \le 2\Delta + 6\).  
Since \(H\) is proper with respect to \(G\), by the minimality of \(G\), the graph
\(H\) admits a \(2\)-distance coloring using \(2\Delta + 7\) colors.  
As \(d_2(v) \le 2\Delta + 6\), there is at least one available color for \(v\), 
and thus the coloring of \(H\) extends to a \(2\)-distance coloring of \(G\), 
a contradiction to the minimality of \(G\).
\end{proof}

For a vertex $v$ with $d(v)=k$, let $v_1, v_2, \ldots, v_k$ denote the neighbors of $v$
in clockwise order, and let $f_1, f_2, \ldots, f_k$ be the faces incident to $v$.
Define $E(v)=\{ v_i v_{i+1} \in E(G) : i \in [k] \,\}$, where indices are taken modulo $k$.
Let $t(v)$ denote the number of edges in $E(v)$ that are contained in two $3$-faces.
We begin with a useful lemma that provides an upper bound on $d_2(v)$.

\begin{lemma}\label{lem:d2v}
For any vertex $v$, we have $d_2(v)\leq \left(\sum_{u\in N(v)}d(u)\right)-2m_3(v)-m_4(v)-t(v)$. 
\end{lemma}

\begin{proof}
Recall that $d_2(v)$ counts the number of vertices at distance at most $2$ from $v$, 
that is, $d_2(v) = |N_2(v)|$.  
Let $p = \sum_{u \in N(v)} d(u)$.  
From $p$, we subtract $2$ for each $3$-face incident to $v$ and $1$ for each $4$-face 
incident to $v$.  
Furthermore, for each edge $v_i v_{i+1}$, if $v_i$ and $v_{i+1}$ have a common neighbor 
other than $v$, we subtract an additional $1$ from $p$.  
Consequently, we obtain
$d_2(v) \le p - 2m_3(v) - m_4(v) - t(v).$
\end{proof}

In the remainder of the paper, we divide the proof into three cases according to 
\( \Delta \in \{6,7,8\} \). In each case, we employ the same overall strategy: 
we first derive structural properties from the minimality of \(G\), and then 
introduce an appropriate set of discharging rules and apply the discharging 
method to reach a contradiction, thereby showing that \(G\) cannot exist.

\subsection{The case \texorpdfstring{$\D=6$}{D6} } \label{sub:6}~~\medskip

Recall that $G$  does not admit a $2$-distance $19$-coloring, whereas 
every planar graph $G'$ obtained from $G$ with a smaller value of $|V(G')|+|E(G')|$ admits one.

\begin{lemma}\label{6lem:min-deg-4}
$\delta(G)\geq 3$. 
\end{lemma}
\begin{proof}
If $v$ is a vertex of degree at most $2$, then $d_2(v)\leq 12$. Let $G'$ be the graph obtained from $G-v$ by adding an edge between the vertices in $N(v)$.  Notice that $G'$ is proper with respect to $G$. By Remark \ref{rem:proper}, this yields a contradiction.
\end{proof}

\begin{lemma}\label{6lem:3-vertex}
Let $v$ be a $3$-vertex. Then $m_3(v)=0$, $m_4(v)\le 1$, and every neighbor of $v$ is a $6(4^-)$-vertex.
\end{lemma}

\begin{proof}
Assume first, for a contradiction, that $v$ is incident to a $3$-face, say $f_1 = v_1 v v_2$.  
By Lemma~\ref{lem:d2v}, we have $d_2(v) \le 16$.  
If we set $G' = G - v + \{v_1 v_3\}$, then $G'$ is proper with respect to $G$, 
yielding a contradiction to Remark~\ref{rem:proper}.

Suppose next that $m_4(v) \ge 2$, that is, $v$ is incident to two $4$-faces, 
say $f_1 = v_1 v v_2 x$ and $f_2 = v_2 v v_3 y$.  
Again, $d_2(v) \le 16$.  
If we set $G' = G - v + \{v_1 v_3\}$, then $G'$ is proper with respect to $G$, 
giving the same contradiction as above.

We now show that every neighbor of $v$ is a $6(4^-)$-vertex.  
Since $m_3(v)=0$, the vertex $v$ cannot have a $6(5^+)$-neighbor.  
Moreover, $v$ cannot have a $5^-$-neighbor.  
Indeed, suppose $v_1$ is such a neighbor.  
Let $G' = G - v + \{v_1 v_2, v_1 v_3\}$.  
In this case $d_2(v) \le 17$, and $G'$ is proper with respect to $G$, 
again contradicting Remark~\ref{rem:proper}.  
Hence every neighbor of $v$ must be a $6(4^-)$-vertex.
\end{proof}

We now introduce some terminology for vertices of special types.
A $4(1)$- or $4(2)$-vertex is called a \emph{bad $4$-vertex}.  
Similarly, a $5(4)$- or $5(5)$-vertex is called a \emph{bad $5$-vertex}.

\begin{lemma}\label{6lem:4-vertex-m3v}
If $v$ is a $4$-vertex, then $m_3(v) \le 2$, and $v$ has no $k(k)$-neighbour for any $k\leq 6$.
\end{lemma}

\begin{proof}
Let $v$ be a $4$-vertex.  
Assume, for a contradiction, that $m_3(v) \ge 3$.  
Let $f_i = v_i v v_{i+1}$ for $i \in [3]$.  
By Lemma~\ref{lem:d2v}, we have $d_2(v) \le 18$.  
If we define $G' = G - v + \{v_1 v_4\}$ (assuming $v_1 v_4 \notin E(G)$), then $G'$ is proper with respect to $G$.  
This contradicts Remark~\ref{rem:proper}.

Now suppose that $v$ has a $k(k)$-neighbour for some $k\leq 6$.  
Without loss of generality, assume that $v_2$ is a $k(k)$-vertex.  
Then $v$ must be incident with two consecutive $3$-faces $f_1$ and $f_2$, say  
$f_1 = v_1 v v_2$ and $f_2 = v_2 v v_3$.  
Moreover, each of the edges $v_1v_2$ and $v_2v_3$ must be contained in two $3$-faces.  
Hence $d_2(v) \le 18$ by Lemma~\ref{lem:d2v}.  
If we set $G' = G - v + \{v_2 v_4\}$, then $G'$ would again be proper with respect to $G$, giving the same contradiction.
\end{proof}

\begin{lemma}\label{6lem:4-vertex-has-no-5(4)-neigh}
Let $v$ be a $4(1)$-vertex. If $v$ has a $4$-neighbour, then the other neighbours of $v$ are neither $4$-vertices nor bad $5$-vertices.
\end{lemma}

\begin{proof}
Suppose that $v$ has a $4$-neighbour, say $v_1$, and assume for a contradiction that
$v$ has another neighbour that is either a $4$-vertex or a $5(4^+)$-vertex.
In this case, Lemma~\ref{lem:d2v} implies $d_2(v) \le 18$.
If we form
$  G' = G - v + \{ v_1 v_2,\, v_1 v_3,\, v_1 v_4 \}$,
then $G'$ is proper with respect to $G$.
By Remark~\ref{rem:proper}, this yields a contradiction.
\end{proof}

\begin{lemma}\label{6lem:4-1-vertex}
Let $v$ be a $4(1)$-vertex.  
If $m_4(v)=r$ for $0 \le r \le 3$, then $n_6(v) \ge r+1$.
\end{lemma}

\begin{proof}
Let $f_1 = v_1 v v_2$ be the unique $3$-face incident to $v$.
Assume $m_4(v)=r$ for some $0 \le r \le 3$.
Suppose, for a contradiction, that $n_6(v) \le r$.
Then $v$ has $4-r$ neighbours of degree at most $5$, and is incident to $r$ $4$-faces.
By Lemma~\ref{lem:d2v}, this implies $d_2(v) \le 18$.
In particular, either $v$ has two $5^-$-neighbours, or $v$ is incident to three $4$-faces. We show that in each case, we can construct a graph $G'$ proper with respect to $G$, which gives a contradiction.

If $v$ has two $5^-$-neighbours $v_i,v_j$ such that  $\{i,j\}\cap \{1,2\}\neq \emptyset$, then we set $G'=G-v+\{v_kv_3,v_kv_4\}$ for  $k\in\{i,j\}\cap \{1,2\}$.
If $v$ has two $5^-$-neighbours  $v_i,v_j$ with $i<j$ such that  $\{i,j\}\cap \{1,2\}= \emptyset$, then we set $G'=G-v+\{v_iv_j,v_iv_2,v_jv_1\}$. 
If $v$ is incident to exactly three $4$-faces, then we set $G'=G-v+\{v_2v_3,v_1v_4\}$. 
In each case, $G'$ is proper with respect to $G$.
By Remark~\ref{rem:proper}, this yields a contradiction.
Hence, $n_6(v) \ge r+1$, as desired.
\end{proof}

\begin{lemma}\label{6lem:4-2-vertex}
Let $v$ be a $4(2)$-vertex. Then $m_4(v) \le 1$. In particular, if $m_4(v)=r$ for $0 \le r \le 1$, then $n_6(v) \ge r+3$.
\end{lemma}

\begin{proof}
Suppose that $f_1$ and $f_2$ are $3$-faces incident to $v$, where
$f_1 = v_i v v_{i+1}$ and $f_2 = v_j v v_{j+1}$ with $i < j$, and $i,j \in [4]$ taken cyclically.
Assume, for a contradiction, that $v$ is incident to two $4$-faces.
Then, by Lemma~\ref{lem:d2v}, we have $d_2(v) \le 18$.
If $f_1$ and $f_2$ are adjacent, choose  
$v_k \in N(v) \setminus \{v_i, v_{i+1}, v_j, v_{j+1}\}$ and set  
$G' = G - v + \{ v_j v_k \}$.
Otherwise, set  
$G' = G - v + \{ v_i v_{j+1},\, v_{i+1} v_j \}$.
In both cases, the resulting graph $G'$ is proper with respect to $G$.
By Remark~\ref{rem:proper}, this yields a contradiction.
Thus $m_4(v) \le 1$.

Now assume $m_4(v)=0$ and $n_6(v) \le 2$.
Then $v$ has two $5^-$-neighbours, and so Lemma~\ref{lem:d2v} gives $d_2(v) \le 18$.
Similarly, if $m_4(v)=1$ and $n_6(v) \le 3$, then $v$ has a $5^-$-neighbour, and again $d_2(v) \le 18$.
In each case, applying the same modification of $G$ as above yields a proper graph $G'$, contradicting Remark~\ref{rem:proper}.  
Thus $n_6(v) \ge r+3$ for $r = m_4(v)$.
\end{proof}

\begin{lemma}\label{6lem:4-vertex-has-no-bad5-66}
Let $v$ be a $4(2)$-vertex.  Then, $v$ has neither a $4$-neighbour nor a bad $5$-neighbour.
\end{lemma}
\begin{proof}
Suppose first that one of $v_i$'s is a $4$-vertex, say $v_1$. Obviously, we have $d_2(v)\leq 18$, and let $G'=G-v+\{v_1v_2,v_1v_3,v_1v_4\}$, assuming these edges are not already present.  Observe that $G'$ is proper with respect to $G$. This contradicts  Remark \ref{rem:proper}.

Suppose now that $v$ has a bad $5$-neighbour, say $v_i$. Then, there exists an edge $v_iv_j$ for $v_{j}\in N(v)$ such that $v_iv_{j}$ is contained in two $3$-faces. By Lemma \ref{lem:d2v}, this implies $d_2(v)\leq 18$.  If we set $G'=G-v+\{v_iv_p,v_iv_t\}$ for $v_p,v_t\in N(v)\setminus \{v_i,v_j\}$, then $G'$ would be proper with respect to $G$.  Similarly as above, we get a contradiction.
\end{proof}

\begin{lemma}\label{6lem:5(4)-vertex-has-no-two-4 or 5(4)-neigh}
Let $v$ be a $5(4)$-vertex.  
If $v$ has a $4$-neighbour, then the other neighbours of $v$ are neither $4$-vertices nor $5(4)$-vertices.
\end{lemma}

\begin{proof}
Let $f_i = v_i v v_{i+1}$ for $i \in [4]$.
Suppose that $v$ has a $4$-neighbour, and assume for a contradiction that $v$ has another neighbour that is either a $4$-vertex or a $5(4)$-vertex.  
Then, by Lemma~\ref{lem:d2v}, we obtain $d_2(v) \le 18$.
If we define $G' = G - v + \{ v_1 v_5 \}$, then $G'$ is proper with respect to $G$.
By Remark~\ref{rem:proper}, this yields a contradiction.
\end{proof}

The following is an immediate consequence of Lemma~\ref{6lem:5(4)-vertex-has-no-two-4 or 5(4)-neigh},  
since a $5(4)$-vertex cannot simultaneously have a $4$-neighbour and a $5(4)$-neighbour.

\begin{corollary}\label{6cor:4-1-vertex-has-no-two-5-4}
If $v$ is a $4(1)$-vertex, then $v$ has no two $5(4)$-neighbours.
\end{corollary}

\begin{lemma}\label{6lem:5-4-vertex-n4vleq2}
Let $v$ be a $5(4)$-vertex. 
\begin{itemize}
\item[$(a)$] If $m_4(v)=0$, then $v$ has at least two $6(5^-)$-neighbours.
\item[$(b)$] If $m_4(v)=1$, then $v$ has at least three $6(5^-)$-neighbours.
\end{itemize}
\end{lemma}

\begin{proof}
Let $f_i = v_i v v_{i+1}$ for $i \in [4]$.

$(a)$ Let $m_4(v)=0$ and assume, for a contradiction, that $v$ has at most one $6(5^-)$-neighbour. 
Then $v$ has four neighbours that are all $6(6)$- or $5^-$-vertices, and hence $d_2(v) \le 18$ by Lemma~\ref{lem:d2v}.  
If we define $G' = G - v + \{ v_1 v_5 \}$, then $G'$ is proper with respect to $G$.  
By Remark~\ref{rem:proper}, this yields a contradiction.

\smallskip
$(b)$ Let $m_4(v)=1$ and assume, again for a contradiction, that $v$ has at most two $6(5^-)$-neighbours.
Then $v$ has three neighbours that are $6(6)$- or $5^-$-vertices, implying $d_2(v) \le 18$ by Lemma~\ref{lem:d2v}.  
As in the previous case, this leads to a contradiction.
\end{proof}

\begin{proposition}\label{6prop:5(4)-has-no-two-non-adj-5(4)}
A $5(4)$-vertex cannot have two non-adjacent $5(4)$-neighbours.
\end{proposition}

\begin{proof}
Let $v$ be a $5(4)$-vertex, and let $f_i = v_i v v_{i+1}$ for $i \in [4]$.
Assume that $v$ has two $5(4)$-neighbours $v_i$ and $v_j$ such that $|i-j| \ge 2$.
Then either $v_i v_{i-1}$ or $v_i v_{i+1}$ (in cyclic order) is contained in two $3$-faces.
Similarly, either $v_j v_{j-1}$ or $v_j v_{j+1}$ is contained in two $3$-faces.
Thus $d_2(v) \le 18$ by Lemma~\ref{lem:d2v}.  
Setting $G' = G - v + \{ v_1 v_5 \}$ yields a graph $G'$ that is proper with respect to $G$, giving a contradiction by Remark~\ref{rem:proper}.
\end{proof}

\begin{lemma}\label{6lem:5-5-vertex-has-no-bad5}
Let $v$ be a $5(5)$-vertex.  
Then $v$ has neither a $4$-neighbour nor a bad $5$-neighbour nor a $6(6)$-neighbour.  
In particular, $n_5(v) \le 1$.
\end{lemma}

\begin{proof}
Let $f_i = v_i v v_{i+1}$ for $i \in [5]$ in cyclic order.
By Lemma~\ref{6lem:4-vertex-m3v}, $v$ has no $4$-neighbour.
Now suppose that $v$ has either a bad $5$-neighbour, or a $6(6)$-neighbour, or two $5$-neighbours.  
In each of these cases, Lemma~\ref{lem:d2v} gives $d_2(v) \le 18$.  
If we set $G' = G - v$, then $G'$ remains proper with respect to $G$.  
By Remark~\ref{rem:proper}, this contradicts the minimality of $G$.  
\end{proof}

The following is an immediate consequence of Lemma~\ref{6lem:5-5-vertex-has-no-bad5}.

\begin{corollary}\label{6cor:5-5-vertex-has-four-light}
If $v$ is a $5(5)$-vertex, then $v$ has four $6(5^-)$-neighbours.
\end{corollary}

An edge $uv$ is said to be \emph{special} if $v$ is a $5(5)$-vertex and $uv$ is contained
in two $3$-faces $f_1,f_2$, each of which is adjacent to a $4^+$-face
(see Figure~\ref{fig:special edge}).  
Note that $u$ is a $5^+$-vertex by Lemmas~\ref{6lem:3-vertex} and \ref{6lem:4-vertex-m3v},

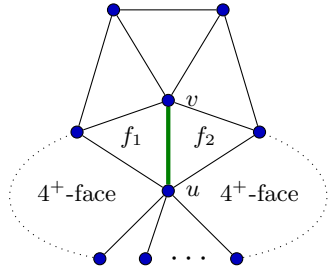
\begin{figure}[htb]
\centering   
\begin{tikzpicture}[scale=1.2]
\node [nod2] at (0,0) (v) [label=right: {\scriptsize $u$}] {};
\node [nod2] at (0,1) (u) [label=right: {\scriptsize $v$}] {}
	edge[green!50!black, ultra thick]  (v);
\node [nod2] at (-1,.65) (u1)  {}
	edge  (u)
	edge  (v);
\node [nod2] at (1,.65) (u2)  {}
	edge  (u)
	edge  (v);
\node [nod2] at (-.6,2) (u3)  {}
	edge  (u)
	edge  (u1);
\node [nod2] at (.6,2) (u4)  {}
	edge  (u)
	edge  (u2)
	edge  (u3);
\node [nod2] at (-.75,-.75) (v1)  {}
	edge  (v);
	%edge [bend left,dotted]  (u1);
\node [nod2] at (-.25,-.75) (v2)  {}
	edge  (v);
\node [nod2] at (.75,-.75) (v3)  {}
	edge  (v);
\node at (0.25,-.75) (asd)     {$\cdots$} ;

\node at (-.4,.6) (f1)     {\scriptsize $f_1$} ;
\node at (.4,.6) (f2)     {\scriptsize $f_2$} ;

\draw[dotted] (v1) .. controls (-2,-.5) and (-2,.25) .. (u1);
\node at (-1,0) (asd)     {\scriptsize $4^+$-face} ;
	
\draw[dotted] (v3) .. controls (2,-.5)  and (2,.25) .. (u2);
\node at (1,0) (asd)     {\scriptsize $4^+$-face} ;

\end{tikzpicture}  
\caption{A special edge $uv$ with a $5(5)$-vertex $v$ and a $5^+$-vertex $u$.}
\label{fig:special edge}
\end{figure}

\begin{proposition}\label{6prop:5-5-vertex-edge-contained-3-faces}
Let $v$ be a $5(5)$-vertex, and let $v_iv_{i+1} \in E(v)$.
\begin{itemize}
\item[$(a)$] If $v_iv_{i+1}$ is contained in two $3$-faces, then all neighbours of $v$ are $6(5^-)$-vertices.  
In particular, for every vertex $v_j \in N(v) \setminus \{v_i,v_{i+1}\}$, the edge $vv_j$ is a special edge.

\item[$(b)$] If no edge in $E(v)$ is contained in two $3$-faces, then  each edge $vv_i$ is a special edge for every $v_i \in N(v)$, and $v$ has four $6(4^-)$-neighbours.
\end{itemize}
\end{proposition}

\begin{proof}
Let $f_i = v_i v v_{i+1}$ for $i \in [5]$ in cyclic order.

$(a)$ Suppose that the edge $v_i v_{i+1}$ is contained in two $3$-faces.  
If $v$ has either a $5^-$-neighbour or a $6(6)$-neighbour, then Lemma~\ref{lem:d2v} implies $d_2(v) \le 18$.  
Removing $v$ (i.e., taking $G' = G - v$) yields a graph proper with respect to $G$, contradicting Remark~\ref{rem:proper}.  
Hence every neighbour of $v$ must be a $6(5^-)$-vertex.

Now suppose that there exists another neighbour $v_j \in N(v) \setminus \{v_i, v_{i+1}\}$ such that either $v_j v_{j-1}$ or $v_j v_{j+1}$ is contained in two $3$-faces.  
Then Lemma~\ref{lem:d2v} again gives $d_2(v) \le 18$, producing the same contradiction.  
Therefore, for every $v_j \in N(v) \setminus \{v_i, v_{i+1}\}$, the edge $vv_j$ is a special edge.

\smallskip
$(b)$ Suppose no edge in $E(v)$ is contained in two $3$-faces.  
Then, by definition of a special edge, each $vv_i$ is a special edge for all $v_i \in N(v)$.  
Since a $5(5)$-vertex has no bad $5$-neighbour and no $6(6)$-neighbour by Lemma~\ref{6lem:5-5-vertex-has-no-bad5}, it follows from Corollary~\ref{6cor:5-5-vertex-has-four-light} that $v$ has four $6(5^-)$-neighbours.  
\end{proof}

\begin{lemma}\label{6lem:6-5-vertex}
Let $v$ be a $6(5)$-vertex. Then $n_4(v) \le 3$. In particular:
\begin{itemize}
\item[$(a)$] If $n_4(v)=3$, then $v$ has no bad $5$-neighbour.
\item[$(b)$] If $n_4(v)=2$, then $v$ has at most two bad $5$-neighbours.
\item[$(c)$] If $n_4(v)=2$ and $v$ has two bad $5$-neighbours, then $v$ has two $6(4^-)$-neighbours.
\item[$(d)$] If $n_4(v)=1$, then $v$ has at most three bad $5$-neighbours.
\item[$(e)$] If $n_4(v)=1$ and $v$ has three bad $5$-neighbours, then $v$ has a $6(4^-)$-neighbour.
\item[$(f)$] If $n_4(v)=0$, then $v$ has at most four bad $5$-neighbours.
\item[$(g)$] If $n_4(v)=0$, $m_4(v)=1$, and $v$ has four bad $5$-neighbours, then $v$ has two $6(4^-)$-neighbours.
\end{itemize}
\end{lemma}

\begin{proof}
Let $f_i = v_i v v_{i+1}$ for $i \in [5]$.  
Recall that $v$ has no $3$-neighbour by Lemma~\ref{6lem:3-vertex}.

\smallskip

We claim that $n_4(v)\leq 3$. Indeed, if $v$ had four $4$-neighbours, then Lemma~\ref{lem:d2v} would imply $d_2(v) \le 18$.  
Among $v_2,v_3,v_4,v_5$, at least one must be a $4$-vertex; assume $v_2$ is such.  
Then setting  
$G' = G - v + \{ v_2 v_4, v_2 v_5, v_2 v_6 \}$  
yields a graph that is proper with respect to $G$, contradicting Remark~\ref{rem:proper}.  

\smallskip
$(a)$ Let $n_4(v)=3$.
If $v$ had a bad $5$-neighbour $v_i$, then  by Lemma~\ref{lem:d2v} we would have $d_2(v) \le 18$.  
Similarly as above, this leads to a contradiction.

\smallskip
$(b)$ Let $n_4(v)=2$.
If $v$ had three bad $5$-neighbours, then  by Lemma~\ref{lem:d2v} we would have $d_2(v) \le 17$.  
If one of $v_2,v_3,v_4,v_5$ is a $4$-vertex (say $v_2$), take  
$G' = G - v + \{ v_2 v_4, v_2 v_5, v_2 v_6 \}$.  
Otherwise, both $v_1$ and $v_6$ are $4$-vertices; with a $5$-neighbour $v_2$, we may take  
$G' = G - v + \{ v_1 v_6, v_2 v_4, v_2 v_6 \}$.  
In both cases, $G'$ is proper with respect to $G$, a contradiction.

\smallskip
$(c)$ Let $n_4(v)=2$, and suppose that $v$ has two bad $5$-neighbours.
Each bad $5$-neighbour $v_i$ forces either $v_i v_{i-1}$ or $v_i v_{i+1}$ to lie in two $3$-faces.  
%If the two bad $5$-neighbours are non-adjacent, the same occurs for both.  
If $v$ had at most one $6(4^-)$-neighbour, then the last neighbour of $v$ would be either a $5$- or  $6(5)$- or  $6(6)$-vertex, giving $d_2(v) \le 18$.  
Applying the same reductions as in part (b) yields a contradiction.
Thus $v$ has at least two $6(4^-)$-neighbours.

\smallskip
$(d)$ Let $n_4(v)=1$.
If $v$ had four bad $5$-neighbours, then by Lemma~\ref{lem:d2v} we would have $d_2(v) \le 18$. 
If one of $v_2,v_3,v_4,v_5$ is a $4$-vertex (say $v_2$), use  
$G' = G - v + \{ v_2 v_4, v_2 v_5, v_2 v_6 \}$.  
Otherwise the $4$-vertex is $v_1$ or $v_6$; assume $v_1$.  
Since $v$ has four $5$-neighbours, one of $v_3,v_5$ is $5$; assume $v_3$.  
Then  
$G' = G - v + \{ v_1 v_6, v_3 v_1, v_3 v_5 \}$  
is proper with respect to $G$, a contradiction.

\smallskip
$(e)$ Let $n_4(v)=1$ and suppose that $v$ has three bad $5$-neighbours.
Observe that two distinct edges among $v_i v_{i+1}$ and $v_j v_{j+1}$ are contained in two $3$-faces.  
If $v$ had no $6(4^-)$-neighbour, then its remaining neighbours would be $5$-, $6(5)$-, or $6(6)$-vertices, giving $d_2(v) \le 18$ by Lemma~\ref{lem:d2v}.  
If one of $v_2,v_3,v_4,v_5$ is a $4$-vertex (say $v_2$), then define  
$G' = G - v + \{ v_2v_4, v_2v_5, v_2v_6 \}$.
Otherwise, the $4$-vertex must be $v_1$ or $v_6$; assume $v_1$.  
If $v_2,v_4,v_6$ are $5$-vertices, then we set  
$G' = G - v + \{ v_1v_6, v_2v_4, v_2v_6 \}$.
Else, one of $v_3,v_5$ is a $5$-vertex; assume $v_3$.  
Then $G' = G - v + \{ v_1v_6, v_3v_1, v_3v_5 \}$.
In each case, $G'$ is proper with respect to $G$, contradicting Remark~\ref{rem:proper}.  
Therefore, $v$ must have a $6(4^-)$-neighbour.

\smallskip
$(f)$ Let $n_4(v)=0$.
If $v$ had five bad $5$-neighbours, then there would  exist three distinct edges of $E(v)$ such that each of them is contained in two $3$-faces, giving $d_2(v) \le 18$ by Lemma~\ref{lem:d2v}.
Among $v_1,v_6$, one is $5$-vertex; assume $v_1$.  
Among $v_3,v_5$, one is $5$-vertex; assume $v_3$.  
Then  
$G' = G - v + \{ v_1 v_6, v_3 v_1, v_3 v_5 \}$  
is proper with respect to $G$, a contradiction.
Thus $v$ has at most four bad $5$-neighbours.

\smallskip
$(g)$ Let $n_4(v)=0$, $m_4(v)=1$, and suppose that $v$ has four bad $5$-neighbours.
Observe that two non-adjacent edges among $v_i v_{i+1}$ is contained  in two $3$-faces.  
If $v$ had at most one $6(4^-)$-neighbour, then its remaining neighbour would be a $5(3^-)$-vertex or a $6(5^+)$-vertex, giving $d_2(v) \le 18$ by Lemma~\ref{lem:d2v}.  
This implies either $v$ has five $5$-neighbours or three edges around $v$ lie in two $3$-faces.

If $v$ has five $5$-neighbours, the argument from part (f) applies, giving a contradiction.  
Thus $v$ has exactly four $5$-neighbours.  
If $v_1$ (resp. $v_6$) is a $5$-vertex and one of $v_3,v_5$ (resp. $v_2,v_4$) is a $5$-vertex, we again use the same reduction as in (f).  
Therefore we may assume $v_1,v_6$ are $6$-vertices and $v_2,v_3,v_4,v_5$ are bad $5$-vertices.
Since a $5(5)$-vertex cannot have any bad $5$-neighbour by Lemma~\ref{6lem:5-5-vertex-has-no-bad5}, each of $v_2, v_3, v_4, v_5$ must be a $5(4)$-vertex.  
However, by Proposition~\ref{6prop:5(4)-has-no-two-non-adj-5(4)}, a $5(4)$-vertex cannot have two non-adjacent $5(4)$-neighbours.  
This contradicts planarity of $G$, and therefore $v$ must have at least two $6(4^-)$-neighbours.
\end{proof}

We  now apply discharging method to show that $G$ cannot exist. First, we assign to each vertex $v$ a charge $\mu(v)=d(v)-4$ and to each face $f$ a charge $\mu(f)=\ell(f)-4$. By Euler's formula, we have
$$\sum_{v\in V(G)}\left(d(v)-4\right)+\sum_{f\in F(G)}(\ell(f)-4)=-8  .$$

We next present some rules and redistribute accordingly. Once the discharging finishes, we check the final charge $\mu^*(v)$ and $\mu^*(f)$. If $\mu^*(v)\geq 0$ and $\mu^*(f)\geq 0$, we get a contradiction that no such a counterexample can exist. \medskip

\noindent  
\textbf{{Discharging Rules}} \medskip

We apply the following discharging rules.

\begin{itemize}
\item[\textbf{R1:}] Every $3$-face  receives $\frac{1}{3}$  from each of its incident vertices.
\item[\textbf{R2:}] Let $f$ be a $5^+$-face. Then $f$ gives
\begin{itemize}
\item[$(a)$]  $\frac{1}{3}$  to each of its incident $3$-vertices and bad $4$-vertices,
\item[$(b)$]  $\frac{1}{5}$  to each of its incident $5(4)$-vertices,
\item[$(c)$] $\frac{1}{5}$  to each of its incident $6(5)$-vertex having no $4$-neighbour.
\end{itemize}
\item[\textbf{R3:}] Every $6(3^-)$-vertex gives $\frac{1}{6}$  to each of its  neighbours.
\item[\textbf{R4:}] Let $v$ be a $6(4)$-vertex. Then $v$ gives
\begin{itemize}
\item[$(a)$]  $\frac{1}{12}$  to each of its  $4$- and $6(5)$-neighbours,
\item[$(b)$]  $\frac{1}{9}$  to each of its $3$- and $5(4)$-neighbours,
\item[$(c)$] $\frac{1}{6}$  to each of its  $5(5)$-neighbour $u$ if $uv$ is a special edge,
\item[$(d)$]  $\frac{1}{9}$  to each of its  $5(5)$-neighbour $u$ if $uv$ is not a special edge.
\end{itemize}
\item[\textbf{R5:}] Let $v$ be a $6(5)$-vertex. Then $v$ gives
\begin{itemize}
\item[$(a)$]  $\frac{1}{12}$  to each of its bad  $4$-neighbours,
\item[$(b)$]  $\frac{1}{9}$  to each of its bad $5$-neighbours.
\end{itemize}
\end{itemize}

\vspace*{1em}

\noindent
\textbf{Checking} $\mu^*(v), \mu^*(f)\geq 0$ for $v\in V(G), f\in F(G)$.\medskip

First we show that $\mu^*(f)\geq 0$ for each $f\in F(G)$. Given a face $f\in F(G)$, 
if $f$ is a $3$-face, then it receives $\frac{1}{3}$ from each of incident vertices by R1, and so $\mu^*(f)= -1+3\times \frac{1}{3}=0$. 
If $f$ is a $4$-face, then $\mu(f)=\mu^*(f)= 0$.

Let $f$ be a $5$-face.
Suppose first that $f$ is not incident to any $3$-vertex. Note that if $f$ is not incident to any bad $4$-vertices, then  $\mu^*(v)\geq 1-5\times\frac{1}{5}= 0$ after   $f$ sends  $\frac{1}{5}$ to each of its incident vertices by R2(b),(c).
Therefore we  further suppose that $f$ is incident to at least one bad $4$-vertex, let $u$ be such a bad $4$-vertex incident to $f$. 
If $u$ is a $4(1)$-vertex, then it has at most one neighbour consisting of  $4$- or  $5(4)$-vertex by Lemma \ref{6lem:4-vertex-has-no-5(4)-neigh} and Corollary \ref{6cor:4-1-vertex-has-no-two-5-4}.
Besides, if $u$ is a $4(2)$-vertex, then the neighbours of $u$ are neither $4$-vertex nor $5(4)$-vertex by Lemma \ref{6lem:4-vertex-has-no-bad5-66}.
As a result, $f$ is incident to at most three bad $4$-vertices. On the other hand, a $5(4)$-vertex having a $4$-neighbour cannot have any other $4$- or $5(4)$-neighbour by Lemma \ref{6lem:5(4)-vertex-has-no-two-4 or 5(4)-neigh}.
Thus, we conclude the following.
\begin{itemize}
\item If $f$ is incident to three bad $4$-vertices, then the remaining vertices incident to $f$ are neither a $4$-vertex, nor a $5(4)$-vertex, nor a \(6(5)\)-vertex having no $4$-neighbour.
\item If $f$ is incident to exactly two bad $4$-vertices, then $f$ is incident to at most one vertex  that is either a \(5(4)\)-vertex or a \(6(5)\)-vertex  having no $4$-neighbour.
\item If $f$ is incident to exactly one bad $4$-vertex, say $w$, then one of the neighbors of $w$ lying on $f$ is neither a $4$-vertex, nor a $5(4)$-vertex, nor a \(6(5)\)-vertex having no $4$-neighbour. 
\end{itemize}
In each case, we deduce that $\mu^*(f)\geq  0$, since the face $f$ sends $\frac{1}{3}$ to each of its incident bad $4$-vertices by R2(a), $\frac{1}{5}$ to each of its incident $5(4)$-vertices by R2(b) and $\frac{1}{5}$ to each of its incident $6(5)$-vertices having no $4$-neighbour by R2(c).

Now we suppose that $f$ is incident to a $3$-vertex, say $v_1$. Then the neighbours of $v_1$ lying on $f$, denote by $v_2,v_5$, are different from $5^-$- and $6(5^+)$-vertex by Lemma \ref{6lem:3-vertex}. Thus,  $\mu^*(v)\geq 1-3\times\frac{1}{3}= 0$ after   $f$ sends  at most $\frac{1}{3}$ to each of its incident vertices other than $v_2,v_5$ by R2.

On the other hand, if $f$ is a $6^+$-face, then we have $\mu^*(v)\geq \ell(f)-4-\ell(f)\times\frac{1}{3}\geq 0$ after  $v$ sends at most $\frac{1}{3}$ to each of its incident vertices by R2.  \medskip

We now pick a vertex $v\in V(G)$ with $d(v)=k$. By Lemma \ref{6lem:min-deg-4}, we have $k\geq 3$. \medskip

\textbf{(1).} Let $k=3$. By Lemma \ref{6lem:3-vertex}, $m_3(v)=0$, $m_4(v)\leq 1$, and each neighbour of $v$ is a $6(4^-)$-vertex. Note that $v$ is incident to at least two $5^+$-faces, and each of them gives $\frac{1}{3}$ to $v$ by R2(a). On the other hand, $v$ is incident to three $6(4^-)$-vertices, and each of them gives at least $\frac{1}{9}$ to $v$ by R3 and R4(b). Thus, we have $\mu^*(v)\geq -1 + 2\times\frac{1}{3}+3\times \frac{1}{9}= 0$.  \medskip

\textbf{(2).} Let $k=4$. By Lemma \ref{6lem:4-vertex-m3v}, $v$ is incident to at most two $3$-faces, in particular,  $v$ has no $6(6)$-neighbour. 
Clearly, we have $\mu^*(v)\geq  0$ when $v$ is not incident to any $3$-face, since $v$ does not give a charge to any vertex or face. We may therefore assume that $1\leq m_3(v)\leq 2$; namely, $v$ is a bad $4$-vertex.

First suppose that $m_3(v)=1$. Let $m_4(v)=r$ for $0\leq r \leq 3$. By Lemma \ref{6lem:4-1-vertex}, we have $n_6(v)\geq r+1$, i.e., $v$ has at least $r+1$ many $6(5^-)$-neighbours, and each of them gives at least $\frac{1}{12}$ to $v$ by R3, R4(a) and R5(a). On the other hand,  $v$ is incident to $3-r$ many $5^+$-faces, and each of them gives $\frac{1}{3}$ to $v$ by R2(a).  Consequently, $v$ receives totally $(3-r)\times \frac{1}{3}$ from its incident $5^+$-faces, and totally $(r+1)\times \frac{1}{12}$ from its $6(5^-)$-neighbours. Thus,  $\mu^*(v)\geq (3-r)\times \frac{1}{3}+(r+1)\times \frac{1}{12}-\frac{1}{3}\geq 0$ after  $v$ sends $\frac{1}{3}$ to its incident $3$-face by R1. 

We now suppose that $m_3(v)=2$. Then, by Lemma \ref{6lem:4-2-vertex}, we have   $m_4(v)=r \leq 1$ and $n_6(v)\geq r+3$. Since $v$ is not adjacent to any $6(6)$-vertex, $v$ has at least $r+3$ many $6(5^-)$-neighbours, and each of them gives at least $\frac{1}{12}$ to $v$ by R3, R4(a) and R5(a). On the other hand, $v$ is incident to $2-r$ many $5^+$-faces, and each of them gives $\frac{1}{3}$ to $v$ by R2(a).  Consequently, $v$ receives totally $(2-r)\times \frac{1}{3}$ from its incident $5^+$-faces, and totally $(r+3)\times \frac{1}{12}$ from its $6$-neighbours. Thus,  $\mu^*(v)\geq (2-r)\times \frac{1}{3}+(r+3)\times \frac{1}{12}-2\times\frac{1}{3}\geq 0$ after  $v$ sends $\frac{1}{3}$ to each of its incident $3$-faces by R1.  \medskip

\textbf{(3).} Let $k=5$. We first note that if $m_3(v) \leq 3$, then  $\mu^*(v)\geq 1-3\times \frac{1}{3}=0$ after  $v$ sends $\frac{1}{3}$ to each of its incident $3$-faces by R1. Therefore, we may assume that $m_3(v)\geq 4$; namely, $v$ is a bad $5$-vertex. 

Let $m_3(v)=4$. If $m_4(v)=0$, then $v$ has at least two $6(5^-)$-neighbours by Lemma \ref{6lem:5-4-vertex-n4vleq2}(a), and $v$ receives at least $\frac{1}{9}$ from each of its $6(5^-)$-neighbours by R3, R4(b) and R5(b). Also, $v$ receives $\frac{1}{5}$ from its incident $5^+$-face by R2(b). Thus,  $\mu^*(v)\geq  1+ 2\times \frac{1}{9}+\frac{1}{5}-4\times\frac{1}{3}> 0$ after  $v$ sends $\frac{1}{3}$ to each of its incident $3$-faces by R1. 
Suppose now that $m_4(v)=1$. Then $v$ has at least three $6(5^-)$-neighbours by Lemma \ref{6lem:5-4-vertex-n4vleq2}(b). It follows from applying R3, R4(b) and R5(b) that $v$ receives at least $\frac{1}{9}$ from each of its $6(5^-)$-neighbours. Thus,  $\mu^*(v)\geq  1+3\times \frac{1}{9}-4\times\frac{1}{3}= 0$ after  $v$ sends $\frac{1}{3}$ to each of its incident $3$-faces by R1.

Let $m_3(v)=5$.  Recall that $v$ has no $6(6)$-neighbour by Lemma \ref{6lem:5-5-vertex-has-no-bad5}. Suppose first that there exists $v_iv_{i+1}\in E(v)$ such that $v_iv_{i+1}$ is contained in two $3$-faces. It then follows from Proposition \ref{6prop:5-5-vertex-edge-contained-3-faces}(a) that all neighbours of $v$ are $6(5^-)$-vertices, and each $vv_j$  for $v_j \in N(v)\setminus \{v_i,v_{i+1}\}$ is a special edge. This particularly implies that  each $v_j \in N(v)\setminus \{v_i,v_{i+1}\}$ is a $6(4^-)$-vertex.
Then, $v$ receives $\frac{1}{6}$ from each $v_j$ for $j\in [5]\setminus \{i,i+1\}$ by R3 and R4(c), and at least $\frac{1}{9}$ from each of $v_i,v_{i+1}$ by R3, R4(d) and R5(b). Thus, $v$ totally receives at least $\frac{2}{3}$ from its $6$-neighbours. 
Next we suppose that no edge $v_iv_{i+1}\in E(v)$ is contained in two $3$-faces. By Proposition \ref{6prop:5-5-vertex-edge-contained-3-faces}(b), there exist four $6(4^-)$-vertices $v_1,v_2,v_3,v_4 \in N(v)$ such that each $vv_i$ is a special edge. Then, $v$ receives $\frac{1}{6}$ from each $v_i$ for $i\in [4]$ by R3 and R4(c). Clearly, $v$  receives totally at least $\frac{2}{3}$ from its $6$-neighbours. 
Hence,  $\mu^*(v)\geq  1+ \frac{2}{3}-5\times\frac{1}{3}= 0$ after  $v$ sends $\frac{1}{3}$ to each of its incident $3$-faces by R1. \medskip

\textbf{(4).} Let $k=6$. Notice first that if $m_3(v) \leq 3$, then  $\mu^*(v)\geq 2- 3\times \frac{1}{3}-6\times \frac{1}{6}=0$ after  $v$ sends $\frac{1}{3}$ to each of its incident $3$-faces by R1, and at most $\frac{1}{6}$ to each of its neighbours by R3. In addition, if $m_3(v)=6$, then  $\mu^*(v)\geq 2-6\times \frac{1}{3} =0$ after  $v$ sends $\frac{1}{3}$ to each of its incident $3$-faces by R1.
Therefore we may assume that $4\leq m_3(v)\leq 5$. \medskip

\textbf{(4.1).} Let $m_3(v)=4$. Obviously, $v$ has one of the three configurations depicted in Figure \ref{6fig:6-4-config}.  Notice that $v$ only gives charge to its incident $3$-faces, and $4^-$-, $5(4^+)$-, $6(5)$-neighbours by R1, R4. 
We first note that if $v$ has no $5(5)$-neighbour $v_i$  such that $vv_i$ is a special edge, then $v$ gives at most $\frac{1}{9}$ to each of its neighbours by R4(a),(b),(d), and so   $\mu^*(v)\geq 2-6\times \frac{1}{9}-4\times\frac{1}{3}=0$ after  $v$ sends $\frac{1}{3}$ to each of its incident $3$-faces by R1.  We may further assume that $v$ has at least one $5(5)$-neighbour $v_i$ such that $vv_i$ is a special edge. In fact, we will determine the final charge of $v$ based on the number of $5(5)$-neighbors $v_i$ of $v$ for which each edge $vv_i$ is a special edge.
It can be easily observed that $v$ cannot have three $5(5)$-neighbours by Lemma \ref{6lem:5-5-vertex-has-no-bad5}.

\begin{figure}[htb]
\centering   
\subfigure[]{
\begin{tikzpicture}[scale=.7]
\node [nodr] at (0,0) (v) [label=above: {\scriptsize $v$}] {};
\node [nod2] at (1,-1.5) (v3) [label=below:{\scriptsize $v_3$}] {}
	edge  (v);		
\node [nod2] at (-1,-1.5) (v4) [label=below:{\scriptsize $v_4$}] {}
	edge [] (v);
\node [nod2] at (-2,0) (v5) [label=left:{\scriptsize $v_5$}]  {}
	edge [] (v)
	edge [] (v4);
\node [nod2] at (-1,1.5) (v6) [label=above:{\scriptsize $v_6$}] {}
	edge [] (v)
	edge [] (v5);	
\node [nod2] at (1,1.5) (v1) [label=above:{\scriptsize $v_1$}] {}
	edge [] (v);
\node [nod2] at (2,0) (v6) [label=right:{\scriptsize $v_2$}]  {}
	edge [] (v)
	edge [] (v3)
	edge [] (v1);				
\end{tikzpicture}  }\hspace*{1cm}
\subfigure[]{
\begin{tikzpicture}[scale=.7]
\node [nodr] at (0,0) (v) [label=above: {\scriptsize $v$}] {};
\node [nod2] at (1,-1.5) (v3) [label=below:{\scriptsize $v_3$}] {}
	edge  (v);		
\node [nod2] at (-1,-1.5) (v4) [label=below:{\scriptsize $v_4$}] {}
	edge [] (v)
	edge [] (v3);
\node [nod2] at (-2,0) (v5) [label=left:{\scriptsize $v_5$}]  {}
	edge [] (v);
\node [nod2] at (-1,1.5) (v6) [label=above:{\scriptsize $v_6$}] {}
	edge [] (v)
	edge [] (v5);	
\node [nod2] at (1,1.5) (v1) [label=above:{\scriptsize $v_1$}] {}
	edge [] (v);
\node [nod2] at (2,0) (v6) [label=right:{\scriptsize $v_2$}]  {}
	edge [] (v)
	edge [] (v3)
	edge [] (v1);				
\end{tikzpicture}  }\hspace*{1cm}
\subfigure[]{
\begin{tikzpicture}[scale=.7]
\node [nodr] at (0,0) (v) [label=above: {\scriptsize $v$}] {};
\node [nod2] at (1,-1.5) (v3) [label=below:{\scriptsize $v_3$}] {}
	edge  (v);		
\node [nod2] at (-1,-1.5) (v4) [label=below:{\scriptsize $v_4$}] {}
	edge [] (v)
	edge [] (v3);
\node [nod2] at (-2,0) (v5) [label=left:{\scriptsize $v_5$}]  {}
	edge [] (v)
	edge [] (v4);
\node [nod2] at (-1,1.5) (v6) [label=above:{\scriptsize $v_6$}] {}
	edge [] (v);	
\node [nod2] at (1,1.5) (v1) [label=above:{\scriptsize $v_1$}] {}
	edge [] (v);
\node [nod2] at (2,0) (v6) [label=right:{\scriptsize $v_2$}]  {}
	edge [] (v)
	edge [] (v3)
	edge [] (v1);				
\end{tikzpicture} }
\caption{Three configurations of $6(4)$-vertices.}
\label{6fig:6-4-config}
\end{figure}
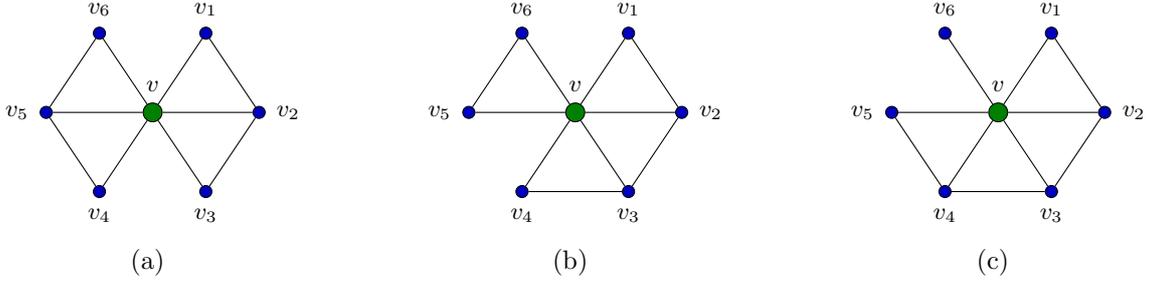

Suppose first that $v$ has two $5(5)$-neighbours $v_i $ and $ v_j$ such that both $vv_i$ and $vv_j$ are special edges. By Lemma \ref{6lem:5-5-vertex-has-no-bad5}, $v_i $ and $ v_j$  are non-adjacent, so the configuration in Figure \ref{6fig:6-4-config}(b) is not possible for $v$. Moreover, $v$ cannot have the form in Figure \ref{6fig:6-4-config}(c) as well, since $vv_i$ and $vv_j$ are special edges.
Hence, we deduce that the neighbours of $v$ can only form as depicted in Figure \ref{6fig:6-4-config}(a), where $v_i = v_2$ and $v_j = v_5$. By Lemmas \ref{6lem:3-vertex} and \ref{6lem:5-5-vertex-has-no-bad5}, $v_2$ and $v_5$ cannot have a $4^-$-neighbour, bad $5$-neighbour and $6(6)$-neighbour, so the other neighbours of $v$ are   $6(5^-)$-vertices or $5(3^-)$-vertices. 
%Recall also that none of $N(v)\setminus \{v_1,v_2\}$ is a $5(5)$-vertex. 
Thus, $v$ gives $\frac{1}{6}$ to each of $v_2,v_5$ by R4(c), and  at most $\frac{1}{12}$ to each of the other neighbours of $v$ by R4(a). Consequently, $\mu^*(v)\geq 2- 2\times\frac{1}{6}-4\times \frac{1}{12}-4\times\frac{1}{3}=0$ after  $v$ sends $\frac{1}{3}$ to each of its incident $3$-faces by R1.

Now we suppose that $v$ has exactly one $5(5)$-neighbour $v_i$ such that $vv_i$ is a special edge. Then, similarly as above, we deduce that the neighbours of $v$ can only form as depicted in Figure \ref{6fig:6-4-config}(a) where $v_i=v_2$ (or $v_i=v_5$). By Lemmas \ref{6lem:3-vertex} and \ref{6lem:5-5-vertex-has-no-bad5}, $v_2$ cannot have a $4^-$-neighbour, bad $5$-neighbour and $6(6)$-neighbour, so the neighbours of $v_2$ other than $v$ are  $6(5^-)$-vertices or $5(3^-)$-vertices. Thus, $v$ gives $\frac{1}{6}$ to $v_2$,  at most $\frac{1}{12}$ to each of $v_1,v_3$ by R4(a),  at most $\frac{1}{9}$ to  each of $v_4,v_5,v_6$ by R4(a),(b),(d).  Hence,  $\mu^*(v)\geq 2- \frac{1}{6}-2\times \frac{1}{12}-3\times \frac{1}{9}-4\times\frac{1}{3}=0$ after  $v$ sends $\frac{1}{3}$ to each of its incident $3$-faces by R1.\medskip

\textbf{(4.2).} Let $m_3(v)=5$. By Lemma \ref{6lem:6-5-vertex}, we have $n_4(v)\leq 3$.

Let $n_4(v)=3$. Note that $v$ has no bad $5$-neighbours by Lemma \ref{6lem:6-5-vertex}(a). So, $v$ only gives charge to its bad $4$-neighbours and its incident $3$-faces.  Thus, $\mu^*(v)\geq 2-5\times \frac{1}{3}-3\times\frac{1}{12}>0$ after  $v$ sends $\frac{1}{3}$ to each of its incident $3$-faces by R1, and  $\frac{1}{12}$ to each of its bad $4$-neighbours by R5(a).

Let $n_4(v)=2$. By Lemma \ref{6lem:6-5-vertex}(b), $v$ has at most two bad $5$-neighbours. If $v$ has such two neighbours, then the remaining neighbours of $v$ are $6(4^-)$-vertices by Lemma \ref{6lem:6-5-vertex}(c), i.e., $v$ has two $6(4^-)$-neighbours. In such a case, $v$ gives  $\frac{1}{12}$ to each of its bad $4$-neighbours by R5(a), and  $\frac{1}{9}$ to each of its bad $5$-neighbours by R5(b). On the other hand, $v$ receives at least $\frac{1}{12}$ from each of its $6(4^-)$-neighbours by R3 and R4(a). Thus,  $\mu^*(v)\geq 2- 2\times \frac{1}{12}-2\times \frac{1}{9}+2\times \frac{1}{12}-5\times\frac{1}{3}> 0$ after  $v$ sends $\frac{1}{3}$ to each of its incident $3$-faces by R1. Suppose now that $v$ has at most one bad $5$-neighbour. Then, $v$ gives  $\frac{1}{12}$ to each of its bad $4$-neighbours by R5(a), and  $\frac{1}{9}$ to its bad $5$-neighbour R5(b) (if exists).  Thus,  $\mu^*(v)\geq 2- 2\times \frac{1}{12}-\frac{1}{9}-5\times\frac{1}{3}>0$ after  $v$ sends $\frac{1}{3}$ to each of its incident $3$-faces by R1.

Let $n_4(v)=1$. By Lemma \ref{6lem:6-5-vertex}(d), $v$ has at most three bad $5$-neighbours. If $v$ has such three neighbours, then $v$ has a $6(4^-)$-neighbour by Lemma \ref{6lem:6-5-vertex}(e), and so  $v$ receives at least  $\frac{1}{12}$ from  its $6(4^-)$-neighbour by R3 and R4(a). On the other hand, $v$ gives  $\frac{1}{12}$ to its bad $4$-neighbour by R5(a), and  $\frac{1}{9}$ to each of its bad $5$-neighbours by R5(b).  Thus,  $\mu^*(v)\geq 2+\frac{1}{12}- \frac{1}{12}-3\times \frac{1}{9}-5\times\frac{1}{3}= 0$ after  $v$ sends $\frac{1}{3}$ to each of its incident $3$-faces by R1. Suppose now that $v$ has at most two bad $5$-vertices. Then,    $\mu^*(v)\geq 2- \frac{1}{12}-2\times\frac{1}{9}-5\times\frac{1}{3}> 0$ after  $v$ sends   $\frac{1}{12}$ to its bad $4$-neighbour R5(a), $\frac{1}{9}$ to each of its bad $5$-neighbours R5(b), and $\frac{1}{3}$ to each of its incident $3$-faces by R1.

Let $n_4(v)=0$. Notice first that $v$ has at most four bad $5$-neighbours by Lemma \ref{6lem:6-5-vertex}(f). If $v$ has at most three bad $5$-neighbours, then $v$ gives  $\frac{1}{9}$ to each of its bad $5$-neighbours by R5(b), and so  $\mu^*(v)\geq 2- 3\times\frac{1}{9}-5\times\frac{1}{3}= 0$ after  $v$ sends $\frac{1}{3}$ to each of its incident $3$-faces by R1. 
Assume further that $v$ has exactly four bad $5$-neighbours. 
If $m_4(v)=0$, i.e, $v$ is incident to a $5^+$-face, then  $v$ receives $\frac{1}{5}$ from its incident $5^+$-face by R2(c).  Thus, $\mu^*(v)\geq 2+\frac{1}{5}- 4\times\frac{1}{9}-5\times\frac{1}{3}> 0$ after  $v$ gives $\frac{1}{9}$ to each of its bad $5$-neighbours by R5(b), and $\frac{1}{3}$ to each of its incident $3$-faces by R1. 
If  $m_4(v)=1$, then $v$ has two $6(4^-)$-neighbours by Lemma \ref{6lem:6-5-vertex}(g), and so $v$ receives at least $\frac{1}{12}$ from each of its $6(4^-)$-neighbours by R3 and R4(a).   Thus,  $\mu^*(v)\geq 2+2\times \frac{1}{12}-4\times \frac{1}{9}-5\times\frac{1}{3}> 0$ after  $v$ sends  $\frac{1}{9}$ to each of its bad $5$-neighbours by R5(b), and $\frac{1}{3}$ to each of its incident $3$-faces by R1. \medskip

%%%%%%%%%%%%%%%%%%%%%%%%%%%%%%%%%%%%%%%%%%%%%%%%%%%%%%%%%%%%%%%%%%%%%%%%%%%%%%%%%%%%%%%%%%%%%%%%%%%%%

%%%%%%%%%%%%%%%%%%%%%%%%%%%%%%%%%%%%%%%%%%%%%%%%%%%%%%%%%%%%%%%%%%%%%%%%%%%%%%%%%%%%%%%%%%%%%%%%%%%%%

%%%%%%%%%%%%%%%%%%%%%%%%%%%%%%%%%%%%%%%%%%%%%%%%%%%%%%%%%%%%%%%%%%%%%%%%%%%%%%%%%%%%%%%%%%%%%%%%%%%%

%%%%%%%%%%%%%%%%%%%%%%%%%%%%%%%%%%%%%%%%%%%%%%%%%%%%%%%%%%%%%%%%%%%%%%%%%%%%%%%%%%%%%%%%%%%%%%%%%%%%%

%\subsection{The Structure of Minimum Counterexample} \label{sub:premECE}~~\medskip
\subsection{The case \texorpdfstring{$\D=7$}{D7} } \label{sub:7}~~\medskip

Recall that $G$ does not admit any $2$-distance $21$-coloring, whereas any planar graph $G'$ obtained from $G$ with a smaller value of $|V(G')| + |E(G')|$ admits a $2$-distance $21$-coloring.  
We begin by establishing several structural properties of $G$, similar to those in the case $\Delta = 6$.

The proof of Lemma~\ref{7lem:min-deg-3} is omitted, as it is similar to the proof of Lemma~\ref{6lem:min-deg-4}.

\begin{lemma}\label{7lem:min-deg-3}
$\delta(G) \ge 3$.
\end{lemma}

\begin{lemma}\label{7lem:3-vertex}
If $v$ is a $3$-vertex, then $m_3(v)=0$, $m_4(v) \le 1$, and each neighbour of $v$ is a $7(5^-)$-vertex.
\end{lemma}

\begin{proof}
Let $v$ be a $3$-vertex.
Assume first that $v$ is incident to a $3$-face, say $f_1 = v_1 v v_2$.  
By Lemma~\ref{lem:d2v}, we have $d_2(v) \le 19$.  
Setting $G' = G - v + \{v_1 v_3\}$ yields a graph that is proper with respect to $G$, contradicting Remark~\ref{rem:proper}.  
Thus $m_3(v)=0$.

Next, suppose that $m_4(v) \ge 2$, i.e., $v$ is incident to two $4$-faces, say  
$f_1 = v_1 v v_2 x$ and $f_2 = v_2 v v_3 y$.  
Again, Lemma~\ref{lem:d2v} gives $d_2(v) \le 19$.  
The graph $G' = G - v + \{v_1 v_3\}$ is proper with respect to $G$, giving the same contradiction.  
Hence $m_4(v) \le 1$.

Finally, we show that each neighbour of $v$ must be a $7(5^-)$-vertex.  
Since $m_3(v)=0$, the vertex $v$ cannot have a $7(6^+)$-neighbour.  
Moreover, $v$ cannot have any $6^-$-neighbour; otherwise,  
if $v_1$ is such a neighbour, then considering $G' = G - v + \{v_1 v_2, v_1 v_3\}$ gives $d_2(v) \le 20$ by Lemma~\ref{lem:d2v}, and $G'$ is proper with respect to $G$, again a contradiction.
Thus all neighbours of $v$ are $7(5^-)$-vertices.
\end{proof}

\begin{lemma}\label{7lem:4-vertex-m3v}
If $v$ is a $4$-vertex, then $m_3(v) \le 3$. In particular, $G$ has no $4(4)$-vertex.
\end{lemma}

\begin{proof}
Let $v$ be a $4$-vertex.  
Assume that $m_3(v)=4$.  
Then $d_2(v) \le 20$ by Lemma~\ref{lem:d2v}.  
If we take $G' = G - v$, then $G'$ is proper with respect to $G$, contradicting Remark~\ref{rem:proper}.  
Thus $m_3(v) \le 3$.
\end{proof}

We now introduce some vertices of special types:  
a $4(1)$-, $4(2)$-, or $4(3)$-vertex is called a \emph{bad $4$-vertex}, and a $5(4)$- or $5(5)$-vertex is called a \emph{bad $5$-vertex}.

\begin{lemma}\label{7lem:4-vertex-neighbours}
Let $v$ be a $4$-vertex.
\begin{itemize}
\item[$(a)$] $v$ has no $5(5)$-neighbour.
\item[$(b)$] If $m_3(v) \ge 1$, then $v$ has at most one $4$-neighbour.
\item[$(c)$] If $m_3(v) \ge 1$ and $v$ has a $4$-neighbour, then $v$ cannot have any $5(4)$-neighbour.
\end{itemize}
\end{lemma}

\begin{proof}
$(a)$ Suppose, for a contradiction, that $v$ has a $5(5)$-neighbour, say $v_2$.  
%Then $v$ must be a $4(2)$- or $4(3)$-vertex, as $G$ has no $4(4)$-vertex by Lemma \ref{7lem:4-vertex-m3v}. 
We may therefore assume that each of $v_1v_2,v_2v_3$ is contained in two $3$-faces.  
By Lemma~\ref{lem:d2v}, we have $d_2(v) \le 20$.  
Setting $G' = G - v + \{ v_2 v_4 \}$ yields a graph that is proper with respect to $G$, contradicting Remark~\ref{rem:proper}.

\smallskip
$(b)$ Let $m_3(v) \ge 1$.  
Suppose that $v$ has two $4$-neighbours, and let $v_1$ be one of them.  
Then $d_2(v) \le 20$ by Lemma~\ref{lem:d2v}.  
Taking $G' = G - v + \{ v_1 v_2, v_1 v_3, v_1 v_4 \}$ gives a graph that is proper with respect to $G$, a contradiction.

\smallskip
$(c)$ Let $m_3(v) \ge 1$.  
Suppose that $v$ has both a $4$-neighbour and a $5(4)$-neighbour.  
Then Lemma~\ref{lem:d2v} again gives $d_2(v) \le 20$, and the same reduction as in part (b) contradicts Remark~\ref{rem:proper}.
\end{proof}

\begin{lemma}\label{7lem:4-1-vertex-neighbours}
Let $v$ be a $4(1)$-vertex.
\begin{itemize}
\item[$(a)$] If $m_4(v)=2$, then $v$ has a $7$-neighbour.  
In particular, $v$ has either a $7(5^-)$-neighbour or two $7(6)$-neighbours.
\item[$(b)$] If $m_4(v)=3$, then $v$ has two $7(6^-)$-neighbours.
\end{itemize}
\end{lemma}

\begin{proof}
Let $f_1 = v_1 v v_2$.

$(a)$ Suppose $m_4(v)=2$.  
We first show that $v$ has a $7$-neighbour.  
Assume all neighbours of $v$ are $6^-$-vertices.  
Then $d_2(v) \le 20$ by Lemma~\ref{lem:d2v}.  
Setting $G' = G - v + \{ v_2 v_3, v_3 v_4, v_1 v_4 \}$ gives a graph that is proper with respect to $G$, contradicting Remark~\ref{rem:proper}.  
Thus $v$ has at least one $7$-neighbour.

Now assume $v$ has no $7(5^-)$-neighbour.  
Since $m_3(v)=1$, $v$ cannot have any $7(7)$-neighbour, so it must have a $7(6)$-neighbour.  
If the remaining neighbours of $v$ are $6^-$-vertices, then again $d_2(v) \le 20$ by Lemma~\ref{lem:d2v}, and the same replacement as above yields a  graph that is proper with respect to $G$.  
Thus $v$ must have a second $7(6)$-neighbour.

\smallskip
$(b)$ Suppose $m_4(v)=3$.  
We first show that $v$ has two $7$-neighbours.  
Assume three neighbours of $v$ are $6^-$-vertices.  
Then $d_2(v) \le 20$ by Lemma~\ref{lem:d2v}.  
Setting $G' = G - v + \{ v_2 v_3, v_1 v_4 \}$ produces a graph that is proper with respect to $G$,  
contradicting Remark~\ref{rem:proper}.   
Thus $v$ has two $7$-neighbours, and since it cannot have any $7(7)$-neighbour, these neighbours must be $7(6^-)$-vertices.
\end{proof}

\begin{lemma}\label{7lem:4-2-vertex-neighbours}
Let $v$ be a $4(2)$-vertex. 
\begin{itemize}
\item[$(a)$] If $m_4(v)=0$, then $v$ has either a $7(5^-)$-neighbour and a $6(4^-)$-neighbour, or two $7(6^-)$-neighbours.
\item[$(b)$] If $m_4(v)=1$, then $v$ has either two $7(6^-)$-neighbours and two $6(4^-)$-neighbours, or three $7(6^-)$-neighbours.
\item[$(c)$] If $m_4(v)=2$, then $v$ has either two $7(5^-)$-neighbours and two $7(6)$-neighbours, or three $7(5^-)$-neighbours.
\end{itemize}
\end{lemma}

\begin{proof}
Let $f_i = v_i v v_{i+1}$ and $f_j = v_j v v_{j+1}$ be the two $3$-faces incident to $v$, with $i<j$ (see Figure~\ref{7fig:4-2-config}).

Whenever $d_2(v) \le 20$, we reduce $G$ as follows:
If $(i,j) = (1,2)$, we set $G' = G - v + \{ v_2 v_4 \}$.
If $(i,j) = (1,3)$, we set $G' = G - v + \{ v_2 v_3, v_1 v_4 \}$.
In each case, $G'$ is proper with respect to $G$, contradicting Remark~\ref{rem:proper}.  
Therefore, throughout the proof we may assume  
$d_2(v) \ge 21$.

$(a)$ Let $m_4(v)=0$.
Suppose first that $v$ has a $7(7)$-neighbour.  
Then $v$ cannot have two $6^-$-neighbours; otherwise $d_2(v) \le 20$ by Lemma \ref{lem:d2v}, a contradiction.  
Thus $v$ has two $7(6^-)$-neighbours.

Now suppose $v$ has no $7(7)$-neighbour.  
Since $v$ cannot have all neighbours of degree at most $6$, it follows that $v$ has at least one $7$-neighbour; let it be $v_r$.  
Assume that $v$ has exactly one $7(6^-)$-neighbour (otherwise the conclusion already holds).  
Then all other neighbours of $v$ are $6^-$-vertices.  
If $v_r$ were a $7(6)$-vertex, or if $v$ had any $6(5^+)$-neighbour, then again $d_2(v)\le 20$ by Lemma \ref{lem:d2v}, a contradiction.  
Hence $v_r$ must be a $7(5^-)$-vertex, and $v$ must have a $6(4^-)$-neighbour.

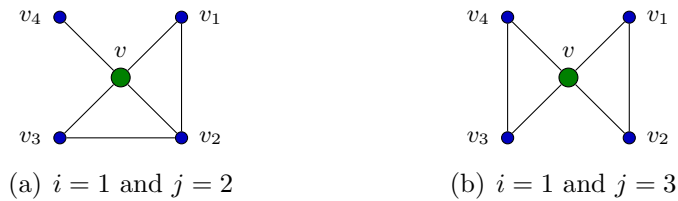
\begin{figure}[htb]
\centering   
\subfigure[$i=1$ and $j=2$]{
\begin{tikzpicture}[scale=.8]
\node [nodr] at (0,0) (v) [label=above: {\scriptsize $v$}] {};
\node [nod2] at (1,-1) (v3) [label=right: {\scriptsize $v_2$}] {}
	edge  (v);		
\node [nod2] at (-1,-1) (v4) [label=left: {\scriptsize $v_3$}] {}
	edge [] (v)
	edge [] (v3);
\node [nod2] at (-1,1) (v6) [label=left: {\scriptsize $v_4$}] {}
	edge [] (v);	
\node [nod2] at (1,1) (v1) [label=right: {\scriptsize $v_1$}] {}
	edge [] (v)
	edge [] (v3);
\node at (2,0) {};	
\node at (-2,0) {};				
\end{tikzpicture}  }\hspace*{2cm}
\subfigure[$i=1$ and $j=3$]{
\begin{tikzpicture}[scale=.8]
\node [nodr] at (0,0) (v) [label=above: {\scriptsize $v$}] {};
\node [nod2] at (1,-1) (v3) [label=right: {\scriptsize $v_2$}] {}
	edge  (v);		
\node [nod2] at (-1,-1) (v4) [label=left: {\scriptsize $v_3$}] {}
	edge [] (v);
\node [nod2] at (-1,1) (v6) [label=left: {\scriptsize $v_4$}] {}
	edge [] (v)
	edge [] (v4);	
\node [nod2] at (1,1) (v1) [label=right: {\scriptsize $v_1$}] {}
	edge [] (v)
	edge [] (v3);
\node at (2,0) {};	
\node at (-2,0) {};				
\end{tikzpicture}}
\caption{Two possible $3$-faces $f_i=v_ivv_{i+1}$ and $f_j=v_jvv_{j+1}$ incident to a $4(2)$-vertex.}
\label{7fig:4-2-config}
\end{figure}

\smallskip
$(b)$ Let $m_4(v)=1$.
We first claim that $v$ has at least two $7$-neighbours.  
Indeed, if three neighbours of $v$ were $6^-$-vertices, then $d_2(v)\le 20$ (since $m_4(v)=1$), a contradiction.
Note also that $v$ has at most one $7(7)$-neighbour, because $m_3(v)=2$.
We now distinguish cases depending on how many of the edges $v_i v_{i+1}$ and $v_j v_{j+1}$ lie in two $3$-faces.

Suppose that both $v_i v_{i+1}$ and $v_j v_{j+1}$ lie in two $3$-faces.
Then $v$ cannot have any $6^-$-neighbour (otherwise $d_2(v)\le 20$).  
Hence all its neighbours are $7$-vertices, and thus $v$ has three $7(6^-)$-neighbours.

Suppose next that exactly one of $v_i v_{i+1}$, $v_j v_{j+1}$ lies in two $3$-faces.
Then $v$ cannot have any $7(7)$- or $6(6)$-neighbour, so it has exactly two $7(6^-)$-neighbours, say $v_1$ and $v_2$.  
Consider $v_3$ and $v_4$.  
If one is a $5^-$-vertex, or if both are $6$-vertices, then $d_2(v)\le 20$, a contradiction.  
Thus one of $v_3,v_4$ is a $7(6^-)$-vertex and the other is a $6^+$-vertex.  
Hence $v$ has at least three $7(6^-)$-neighbours.

Finally suppose that neither $v_i v_{i+1}$ nor $v_j v_{j+1}$ lies in two $3$-faces.
Then $v$ has no $7(7)$-, $7(6)$-, $6(6)$-, or $6(5)$-neighbour.  
Thus $v$ has two $7(5^-)$-neighbours, say $v_1$ and $v_2$.  
For $v_3,v_4$:  
If one is a $5^-$-vertex, then the other must be a $7(5^-)$-vertex (else $d_2(v)\le 20$).  
If neither is a $5^-$-vertex, then either one is a $7(5^-)$-vertex or both are $6(4^-)$-vertices.
This completes part (b).

\smallskip
$(c)$ Let  $m_4(v)=2$.
If $v$ had two $6^-$-neighbours, or a $7(7)$-neighbour, or a $6(6)$-neighbour, then $d_2(v)\le 20$ (since $m_4(v)=2$), a contradiction.  
Thus $v$ has exactly three $7(6^-)$-neighbours.

If $v$ had three $7(6)$-neighbours, then both $v_i v_{i+1}$ and $v_j v_{j+1}$ would have to lie in two $3$-faces, forcing again $d_2(v)\le 20$, a contradiction.  
Thus $v$ has at most two $7(6)$-neighbours.

If $v$ has a $7(6)$-neighbour, then at least one of $v_i v_{i+1}$ or $v_j v_{j+1}$ lies in two $3$-faces, say $v_i v_{i+1}$.  
This forces that $v_{i+1}$ is a $7(6^-)$-vertex, and all neighbours of $v$ other than $v_i, v_{i+1}$ are  $7(5^-)$-vertices (otherwise $d_2(v)\le 20$).  
Hence $v$ has either:
 two $7(6)$-neighbours and two $7(5^-)$-neighbours, or  
 three $7(5^-)$-neighbours.
\end{proof}

\begin{lemma}\label{7lem:4-3-vertex-neighbours}
Let $v$ be a $4(3)$-vertex.
\begin{itemize}
\item[$(a)$] If $m_4(v)=0$, then $v$ has one of the following:
 four $7(5^-)$-neighbours, or
 three $7(5^-)$- and one $7(6)$-neighbours, or
 three $7(5^-)$- and one $6(4^-)$-neighbours, or
  two $7(5^-)$- and two $7(6)$-neighbours.

\item[$(b)$] If $m_4(v)=1$, then all neighbours of $v$ are $7(5^-)$-vertices.
\end{itemize}
\end{lemma}

\begin{proof}
Let $f_i=v_ivv_{i+1}$ for $i\in [3]$.
If $d_2(v) \le 20$, then $G'=G-v+\{v_1v_4\}$ yields a graph that is proper with respect to $G$, contradicting Remark~\ref{rem:proper}.
Therefore, throughout the proof we may assume  
$d_2(v) \ge 21$. Note that at most one of the edges $v_1v_2$, $v_2v_3$, $v_3v_4$ is contained in two $3$-faces, and so
$v$ has no $6(6)$- or $7(7)$-neighbour. 
Observe also that
$v$ has no $5^-$- or $6(5)$-neighbour;
$v$ has at most one $6$-neighbour; and  
 $v$ has at most two $7(6)$-neighbours.

\smallskip
$ (a) $ Let $m_4(v)=0$.  
As noted above, $v$ has at most one $6(4^-)$-neighbour.  
If $v$ has such a $6(4^-)$-neighbour, then it cannot have a $7(6)$-neighbour, since otherwise $d_2(v)\le 20$, a contradiction.  
Thus all other neighbours of $v$ are $7(5^-)$-vertices.
If $v$ has no $6(4^-)$-neighbour, then   the only possible neighbourhood types are exactly those listed in the statement, since $v$ has at most two $7(6)$-neighbours.

\smallskip
$ (b) $
Let $m_4(v)=1$.  
If $v$ has a $7(6^+)$- or a $6$-neighbour, then $d_2(v)\le 20$ by Lemma~\ref{lem:d2v}, a contradiction.  
Hence every neighbour of $v$ must be a $7(5^-)$-vertex.
\end{proof}

\begin{lemma}\label{7lem:5-4-vertex}
Let $v$ be a $5(4)$-vertex. 
\begin{itemize}
\item[$(a)$] $v$ cannot have two $5(5)$-neighbours.
\item[$(b)$] If $v$ has a $5(5)$-neighbour, then $v$ has no $4$-neighbour.
\item[$(c)$] If $v$ has a $7(6)$-neighbour, then $v$ is adjacent to at most one $4$-vertex.
\item[$(d)$] If $m_4(v)=0$, then $v$ has at least two neighbours that are neither $5^-$- nor $6(6)$-vertices.
\item[$(e)$] If $m_4(v)=1$, then $v$ has at least three $6^+$-neighbours different from $6(6)$-vertex.
\item[$(f)$] If $m_4(v)=1$ and $v$ has a $7(7)$-neighbour, then $v$ has at least four $6^+$-neighbours different from $6(6)$-vertex.
\end{itemize}
\end{lemma}

\begin{proof}
Let $f_i=v_ivv_{i+1}$ for $i\in [4]$.  
If $d_2(v)\le 20$, then setting $G'=G-v+\{v_4v_5\}$ produces a graph $G'$ that is proper with respect to $G$, contradicting Remark~\ref{rem:proper}.  
Thus throughout the proof we assume $d_2(v)\ge 21$.

\smallskip
$ (a)$ 
If $v$ had two $5(5)$-neighbours, then there would exist at least three edges in $E(v)$ such that each of them lies on two $3$-faces, which forces $d_2(v)\le 20$ by Lemma~\ref{lem:d2v}, a contradiction.

\smallskip
$ (b)$
If $v$ had both a $5(5)$-neighbour and a $4$-neighbour, then again $d_2(v)\le 20$ by Lemma~\ref{lem:d2v}, a contradiction.

\smallskip
$ (c)$ 
If $v$ had a $7(6)$-neighbour and two $4$-neighbours, then some edge of $E(v)$ would be contained in two $3$-faces, which gives $d_2(v)\le 20$ by Lemma~\ref{lem:d2v}, a contradiction.

\smallskip
$ (d)$
Let $m_4(v)=0$.  
Suppose that $v$ has at most one neighbour that is a $6(5^-)$- or $7$-vertex.  
Then all remaining neighbours of $v$ are $6(6)$- or $5^-$-vertices.  
If $v$ has a $6(6)$-neighbour, then two edges of $E(v)$ lie in two $3$-faces.  
Considering all possible distributions of $5^-$- and $6(6)$-neighbours, one always obtains $d_2(v)\le 20$ by Lemma~\ref{lem:d2v}, a contradiction.  
Thus $v$ must have at least two neighbours that are neither $5^-$- nor $6(6)$-vertices.

\smallskip
$ (e)$
Let $m_4(v)=1$.  
Assume that $v$ has three neighbours forming $5^-$- or $6(6)$-vertices.  
Note that if $v$ has a $6(6)$-neighbour, then two edges of $E(v)$ lie on two $3$-faces.  
By checking all possible distributions of three neighbours among $5^-$-vertices and $6(6)$-vertices, we always obtain $d_2(v)\le 20$ by Lemma~\ref{lem:d2v}, a contradiction.

\smallskip
$ (f)$
Let $m_4(v)=1$ and suppose $v$ has a $7(7)$-neighbour.  
Then at least two edges of $E(v)$ are contained in two $3$-faces.  
If $v$ had two neighbours forming $5^-$- or $6(6)$-vertices, then $d_2(v)\le 20$ by Lemma~\ref{lem:d2v}, a contradiction.  
\end{proof}

\begin{lemma}\label{7lem:5-5-vertex}
Let $v$ be a $5(5)$-vertex. 
\begin{itemize}
\item[$(a)$] $n_4(v)=0$ and $n_5(v)\le 2$.
\item[$(b)$] $v$ cannot have two neighbours that are $5(4^+)$-, or $6(6)$-vertices.  
In particular, $v$ has at least one $7$-neighbour.
\item[$(c)$] If $v$ has four $6$-neighbours, then $v$ also has a $7(5^-)$-neighbour.
\item[$(d)$] If $n_5(v)=0$ and $v$ has both a $6(5^-)$- and a $6(6)$-neighbour, then $v$ has three $7$-neighbours, at least one of which is $7(5^-)$-vertex.
\item[$(e)$] If $n_5(v)=1$, then $v$ has no $6(6)$-neighbours; in particular, $v$ has at most two $6$-neighbours.
\item[$(f)$] If $n_5(v)=1$ and $v$ has exactly one $6(5^-)$-neighbour, then $v$ has a $7(5^-)$-neighbour.
\item[$(g)$] If $n_5(v)=1$ and $v$ has exactly two $6(5^-)$-neighbours, then $v$ has two $7(5^-)$-neighbours.
\item[$(h)$] If $n_5(v)=2$, then $v$ has three $7(5^-)$-neighbours.
\end{itemize}
\end{lemma}

\begin{proof}
Let $f_i = v_i v v_{i+1}$ for $i\in[5]$ in cyclic order.
If $d_2(v)\le 20$, then taking $G' = G-v$ gives a graph that is proper with respect to $G$, contradicting Remark~\ref{rem:proper}.  
We therefore assume $d_2(v)\ge 21$.

\smallskip
$ (a) $
By Lemma~\ref{7lem:4-vertex-neighbours}(a), $v$ has no $4$-neighbour.  
If $v$ had three $5$-neighbours, then $d_2(v)\le 19$ by Lemma~\ref{lem:d2v}, a contradiction.  

\smallskip
$ (b) $
If $v$ had two neighbours consisting of $5(4^+)$-, or $6(6)$-vertex, then $d_2(v)\le 20$ by Lemma~\ref{lem:d2v}, a contradiction.  
Similarly, if all neighbours of $v$ were $6^-$-vertices, then again $d_2(v)\le 20$, a contradiction.  
Hence $v$ must have at least one $7$-neighbour.

\smallskip
$ (c) $
Suppose that $v$ has four $6$-neighbours.  
By (b), the fifth neighbour must be a $7$-vertex, say $v_1$.  
If $v_1$ were a $7(6^+)$-vertex, then some edge in $E(v)$ would be contained in two $3$-faces, implying $d_2(v)\le 20$ by Lemma~\ref{lem:d2v}, a contradiction.  
Hence $v_1$ is a $7(5^-)$-vertex.

\smallskip
$ (d) $ 
Let $n_5(v)=0$, and suppose that $v$ has both a $6(5^-)$- and a $6(6)$-neighbour.  
If $v$ had only two $7$-neighbours, then $d_2(v)\le 20$  by Lemma~\ref{lem:d2v}, a contradiction.  
Thus $v$ has three $7$-neighbours.  
At least one of them must be $7(5^-)$-vertex; otherwise there would exist three edges of $E(v)$ such that each of them lies in two $3$-faces, again forcing $d_2(v)\le 20$ by Lemma~\ref{lem:d2v}.

\smallskip
$ (e) $
Let $n_5(v)=1$.  
If $v$ had a $6(6)$-neighbour or three $6$-neighbours, then $d_2(v)\le 20$ by Lemma~\ref{lem:d2v}, a contradiction.

\smallskip
$ (f) $
Let $n_5(v)=1$, and suppose that $v$ has exactly one $6(5^-)$-neighbour.  
By (e), $v$ has no $6(6)$-neighbour, so the other three neighbours are $7$-vertices.  
If all were $7(6^+)$-vertices, then at least two edges of $E(v)$ would lie in two $3$-faces, giving $d_2(v)\le 20$ by Lemma~\ref{lem:d2v}, a contradiction.  
Thus $v$ has a $7(5^-)$-neighbour.

\smallskip
$ (g) $ 
Let $n_5(v)=1$, and suppose that $v$ has exactly two $6(5^-)$-neighbours.  
By (a) and (e), $v$ has two $7$-neighbours.  
If one were a $7(6^+)$-vertex, some edge of $E(v)$ would lie in two $3$-faces, giving $d_2(v)\le 20$ by Lemma~\ref{lem:d2v}, a contradiction.  
Thus both $7$-neighbours are $7(5^-)$-vertices.

\smallskip
$ (h) $
Let $n_5(v)=2$.  
If $v$ had a $6$-neighbour or a $7(6^+)$-neighbour, then $d_2(v)\le 20$ by Lemma~\ref{lem:d2v}, a contradiction.  
\end{proof}

\begin{lemma}\label{7lem:6-5-vertex}
Let $v$ be a $6(5)$-vertex.  
Then $v$ has at most four bad $5$-neighbours.  
In particular, if $m_4(v)=1$ and $v$ has four bad $5$-neighbours, then $v$ has two $7(5^-)$-neighbours.
\end{lemma}

\begin{proof}
Let $f_i=v_ivv_{i+1}$ for $i\in[5]$.
First, we claim that $v$ cannot have five bad $5$-neighbours. Indeed, if all five neighbours were bad $5$-vertices, then at least three edges in $E(v)$ would each lie in two $3$-faces, forcing $d_2(v)\le 19$ by Lemma~\ref{lem:d2v}.  
Since $v_1$ or $v_6$ must be a $5$-vertex (say $v_1$) and similarly one of $v_3,v_5$ must be a $5$-vertex (say $v_3$), we may take
$
G' = G-v+\{v_1v_6,\, v_3v_1,\, v_3v_5\},
$
and $G'$ is proper with respect to $G$, contradicting Remark~\ref{rem:proper}.  
Thus $v$ has at most four bad $5$-neighbours.

Now assume $m_4(v)=1$ and $v$ has four bad $5$-neighbours.
Then at least two edges in $E(v)$ lie in two $3$-faces.  
If $v$ had at most one $7(5^-)$-neighbour, then $d_2(v)\le 20$ by Lemma~\ref{lem:d2v}, a contradiction.  
\end{proof}

Let $v$ be a $7$-vertex, and let $x$ be a $4$-neighbour of $v$.  
If $v$ and $x$ have two common neighbours (equivalently, the edge $vx$ is contained in two $3$-faces), then  
$x$ is called a \emph{support neighbour} of $v$.

A vertex is said to be \emph{poor} if it is a $4$-vertex or a $5(5)$-vertex.  
Since $G$ has no $4(4)$-vertices, the poor vertices are exactly  
$4(0)$-, $4(1)$-, $4(2)$-, $4(3)$-, and $5(5)$-vertices.

\begin{lemma}\label{7lem:7-vertex}
Let $v$ be a $7$-vertex. Then:
\begin{itemize}
\item[$(a)$] If $v$ has a support neighbour, then $d_2(v)\ge 21$.
\item[$(b)$] If $4\le m_3(v)\le 5$, then $v$ has at most six $4$-neighbours.
\item[$(c)$] If $m_3(v)=5$ and $n_4(v)\le 2$, then $v$ has at most four poor neighbours.
\item[$(d)$] If $v$ is a $7(6)$-vertex, then $v$ has at most five poor neighbours.
\item[$(e)$] If $n_4(v)=3$, then $v$ has at most two $5(5)$-neighbours.
\item[$(f)$] If $n_4(v)=4$, then $v$ has at most one $5(5)$-neighbour.
\item[$(g)$] If $m_3(v)=5$, $n_4(v)=4$, $v$ has a support neighbour, and $v$ has both a $5(4)$- and a $5(5)$-neighbour, then $v$ must be adjacent to a $7$-vertex.
\item[$(h)$] If $m_3(v)=5$, $n_4(v)\geq 3$, $v$ has a support neighbour, and $v$ has  a $3$-neighbour, then   $v$ has at most four poor neighbours.
\item[$(i)$] If $m_3(v)=5$, $n_4(v)=5$, then   $v$ has at most one $5(4)$-neighbour. In particular, if $v$ has a $5(4)$-neighbour or two $6(5)$-neighbours, then $v$ is incident to two $5^+$-faces.

\item[$(j)$] If $m_3(v)=5$, $n_4(v)=6$, then $v$ is incident to two $5^+$-faces; in particular,  $v$ has no bad $5$-neighbour.
\end{itemize}
\end{lemma}

\begin{proof}
$(a)$
Let $v_1$ be a support neighbour of $v$.  
Then $v_1$ is a $4$-vertex adjacent to both $v_2$ and $v_7$.  
Assume for a contradiction that $d_2(v)\le 20$.  
Take
$
G' = G - v + \{v_1v_3,\, v_1v_4,\, v_1v_5,\, v_1v_6\}.
$
Then $G'$ is proper with respect to $G$, contradicting Remark~\ref{rem:proper}.  
Thus $d_2(v)\ge 21$.

\smallskip
$(b)$ 
Let $4\le m_3(v)\le 5$.  
If all neighbours of $v$ were $4$-vertices, then $d_2(v)\ge 20$; and since $m_3(v)\ge 4$, at least one neighbour must be a support neighbour.  
By part (a), this yields a contradiction.  
Hence $v$ has at most six $4$-neighbours.

\smallskip
$(c)$  
Assume $m_3(v)=5$ and $n_4(v)\le 2$.  
Suppose for a contradiction that $v$ has five poor neighbours.  
By Lemma~\ref{7lem:5-5-vertex}(b), a $5(5)$-vertex is adjacent to at most one $5(5)$-vertex.  
Since $m_3(v)=5$, $v$ can have at most three $5(5)$-neighbours.  
If $v$ has five poor neighbours, then it must have exactly three $5(5)$-neighbours and two $4$-neighbours, contradicting Lemma~\ref{7lem:5-5-vertex}(a), which states that $5(5)$-vertices have no $4$-neighbours.

\smallskip
$(d)$   
Let $v$ be a $7(6)$-vertex.  
By Lemma~\ref{7lem:5-5-vertex}(a),(b), a $5(5)$-vertex has no $4$-neighbour and is adjacent to at most one bad $5$-vertex.  
Also, by Lemma~\ref{7lem:4-vertex-neighbours}(b), a $4(1^+)$-vertex is adjacent to at most one $4$-vertex.  
These restrictions prevent $v$ from having six poor neighbours.

\smallskip
$(e)$  
Let  $n_4(v)=3$.  
Since a $5(5)$-vertex has no $4$-neighbour and is adjacent to at most one bad $5$-vertex (Lemma~\ref{7lem:5-5-vertex}(a),(b)),  
$v$ can have at most two $5(5)$-neighbours.

\smallskip
$(f)$  
Let $n_4(v)=3$. Then similar reasoning as in (d) implies that $v$ can have at most one $5(5)$-neighbour.

\smallskip
$(g)$   
Assume $m_3(v)=5$, $n_4(v)=4$, $v$ has a support neighbour, and $v$ has both a $5(4)$- and a $5(5)$-neighbour.  
If the seventh neighbour of $v$ were a $6^-$-vertex, then $d_2(v)\le 20$ by Lemma~\ref{lem:d2v}, contradicting part (a).  
Hence the remaining neighbour must be a $7$-vertex.

\smallskip
$(h)$
Assume the stated conditions.
If $v$ had five poor neighbours, then $d_2(v)\le 20$ by Lemma~\ref{lem:d2v},
contradicting~(a).

\smallskip
$(i)$
Assume $m_3(v)=5$ and $n_4(v)=5$.
Then $v$ has a support neighbour.
If $v$ had two $5(4)$-neighbours, then $d_2(v)\le 20$ by Lemma~\ref{lem:d2v},
contradicting~(a). Thus $v$ has at most one such neighbour.
Similarly, if $v$ has one $5(4)$-neighbour or two $6(5)$-neighbours,
then $v$ must be incident to two $5^+$-faces.

\smallskip
$(j)$
Assume $m_3(v)=5$ and $n_4(v)=6$.
Then $v$ has a support neighbour.
If $v$ were not incident to two $5^+$-faces, then $d_2(v)\le 20$ by Lemma~\ref{lem:d2v},
contradicting~(a). Thus $v$ is incident to two $5^+$-faces.  Similarly we deduce that  $v$ has no $5(4)$- and $5(5)$-neighbours..
\end{proof}

We now apply discharging to show that $G$ does not exist. 
We use the same initial charges as in the case $\Delta=6$, and redistribute charge according to the following rules. \medskip

\noindent  
\textbf{{Discharging Rules}} \medskip

We apply the following discharging rules.

\begin{itemize}
\setlength\itemsep{.2em}
\item[\textbf{R1:}] Every $3$-face  receives $\frac{1}{3}$  from each of its incident vertices.
\item[\textbf{R2:}] Every $5^+$-face gives $\frac{1}{5}$  to each of its incident vertices.
\item[\textbf{R3:}] Every $6(3^-)$-vertex gives $\frac{1}{6}$  to each of its neighbours.
\item[\textbf{R4:}] Every $6(4)$-vertex gives $\frac{1}{9}$  to each of its neighbours.
\item[\textbf{R5:}] Every $6(5)$-vertex gives $\frac{1}{9}$  to each of its  bad $5$-neighbours.
\item[\textbf{R6:}] Every $7(3^-)$-vertex gives $\frac{2}{7}$  to each of its neighbours.
\item[\textbf{R7:}] Let $v$ be a $7(4)$- or $7(5)$-vertex. Then $v$ gives
\begin{itemize}
\item[$(a)$]  $\frac{1}{5}$  to each of its $3$-neighbours,
\item[$(b)$]  $\frac{1}{4}$  to each of its bad $4$-neighbours,
\item[$(c)$]  $\frac{2}{9}$  to each of its $5(5)$-neighbours,
\item[$(d)$]  $\frac{1}{9}$  to each of its $5(4)$-neighbours,
\item[$(e)$]  $\frac{1}{18}$  to each of its $6(5)$-neighbours.
\end{itemize}

\item[\textbf{R8:}] Let $v$ be a $7(6)$-vertex. Then $v$ gives
\begin{itemize}
\item[$(a)$]  $\frac{1}{6}$  to each of its $5(5)$-neighbours,
\item[$(b)$]  $\frac{1}{9}$  to each of its $5(4)$-neighbours,
\item[$(c)$] $\frac{1}{6}$  to each of its  bad $4$-neighbours.
\end{itemize}

\item[\textbf{R9:}] Let $v$ be a $7(7)$-vertex. Then $v$ gives
\begin{itemize}
\item[$(a)$]  $\frac{1}{6}$  to each of its $5(5)$-neighbours,
\item[$(b)$]  $\frac{1}{12}$  to each of its $5(4)$-neighbours.
\end{itemize}

\end{itemize}

\vspace*{1em}

\noindent
\textbf{Checking} $\mu^*(v), \mu^*(f)\geq 0$ for $v\in V(G), f\in F(G)$.\medskip

First we show that $\mu^*(f)\geq 0$ for each $f\in F(G)$.  
Recall that every face $f$ has initial charge $\mu(f)=\ell(f)-4$.  
If $f$ is a $3$-face, then by R1 it receives $\frac{1}{3}$ from each of its incident vertices, and hence  
$\mu^*(f)\geq -1 + 3\times \frac{1}{3} = 0$.
If $f$ is a $4$-face, then $\mu(f)=0$ and it neither sends nor receives charge, so $\mu^*(f)=\mu(f)=0$.
Let $f$ be a $5^+$-face. By R2, the face $f$ sends $\frac{1}{5}$ to each of its incident vertices. Therefore,  
$\mu^*(f)= \ell(f)-4 - \ell(f)\times \frac{1}{5} = \frac{4\ell(f)}{5} - 4 \geq 0$.
Consequently, every face $f\in F(G)$ satisfies $\mu^*(f)\geq 0$.  
\medskip

Next, let $v\in V(G)$ be a vertex of degree $d(v)=k$.  
By Lemma~\ref{7lem:min-deg-3}, we have $k\geq 3$.  
\medskip

\textbf{(1).} Let $k=3$.  The initial charge of $v$ is $\mu(v)=d(v)-4=-1$. By Lemma \ref{7lem:3-vertex}, $m_3(v)=0$, $m_4(v)\leq 1$ and each neighbour of $v$ is a $7(5^-)$-vertex. This implies that $v$ is incident to at least two $5^+$-faces. Then $v$ receives  $\frac{1}{5}$ from each of its incident $5^+$-faces by R2, and at least $\frac{1}{5}$ from each of its $7(5^-)$-neighbours by R6, R7(a). Hence, $\mu^*(v)\geq -1+2\times \frac{1}{5}+3\times\frac{1}{5}=0$. \medskip

\textbf{(2).} Let $k=4$.  The initial charge of $v$ is $\mu(v)=d(v)-4=0$. We have $m_3(v)\leq 3$ by Lemma \ref{7lem:4-vertex-m3v}.  If $m_3(v)=0$, then $\mu^*(v)\geq 0$ since $v$ does not give a charge to its any incident faces. So we may assume that  $1 \leq m_3(v)\leq 3$. \medskip

\textbf{(2.1).}  Let $m_3(v)= 1$.  
If $m_4(v)\leq 1$, then $v$ is incident to two $5^+$-faces, and by R2, $v$ receives $\frac{1}{5}$ from each of those $5^+$-faces. Thus,  $\mu^*(v)\geq 2\times \frac{1}{5}-\frac{1}{3}>0$ after  $v$ sends $\frac{1}{3}$ to its incident $3$-face by R1. 
If $m_4(v)=2$, then $v$ has either a $7(5^-)$-neighbour or two $7(6)$-neighbours by Lemma \ref{7lem:4-1-vertex-neighbours}(a). In both cases, $v$ receives at least $\frac{1}{4}=\min\{\frac{1}{4}, 2\times\frac{1}{6} \}$ from its $7$-neighbours by R6-R8. On the other hand, $v$ receives $\frac{1}{5}$ from its incident $5^+$-face by R2.
Thus,   $\mu^*(v)\geq  \frac{1}{4}+\frac{1}{5}-\frac{1}{3}>0$ after  $v$ sends $\frac{1}{3}$ to its incident $3$-face by R1. 
If $m_4(v)=3$, then $v$ has two $7(6^-)$-neighbours by Lemma \ref{7lem:4-1-vertex-neighbours}(b). By applying R6-R8, $v$ receives totally at least $\frac{1}{3}$ from its $7$-neighbours.  Thus,   $\mu^*(v)\geq  \frac{1}{3}-\frac{1}{3}=0$ after  $v$ sends $\frac{1}{3}$ to its incident $3$-face by R1. \medskip

\textbf{(2.2).}  Let $m_3(v)=2$. Since $0\leq m_4(v) \leq 2$, we consider the following cases:

If $m_4(v)=0$, then $v$ has either  a  $7(5^-)$-neighbour and a $6(4^-)$-neighbour or  two $7(6^-)$-neighbours by Lemma \ref{7lem:4-2-vertex-neighbours}(a). It follows that $v$ receives totally at least $\frac{1}{3}=\min\{\frac{1}{4}+\frac{1}{9}, 2\times\frac{1}{6} \}$ from its $6^+$-neighbours by R3-R4 and R6-R8. In addition, $v$ is incident to two $5^+$-faces, and by R2, $v$ receives $\frac{1}{5}$ from each of those $5^+$-faces.  Thus,   $\mu^*(v)\geq  \frac{1}{3}+2\times\frac{1}{5}-2\times\frac{1}{3}>0$ after  $v$ sends $\frac{1}{3}$ to each of its incident $3$-faces by R1. 

If $m_4(v)=1$, then $v$ has either  two $7(6^-)$-neighbours and two $6(4^-)$-neighbours or three $7(6^-)$-neighbours by Lemma \ref{7lem:4-2-vertex-neighbours}(b). In both cases, $v$ receives totally at least $\frac{1}{2}=\min\{2\times\frac{1}{6}+2\times\frac{1}{9}, 3\times\frac{1}{6} \}$ from its $6^+$-neighbours by R3-R4 and R6-R8. Also, $v$ receives $\frac{1}{5}$ from its incident $5^+$-face by R2. Thus,   $\mu^*(v)\geq  \frac{1}{2}+\frac{1}{5}-2\times\frac{1}{3}>0$ after  $v$ sends $\frac{1}{3}$ to each of its incident $3$-faces by R1. 

If $m_4(v)=2$, then $v$ has either  two  $7(5^-)$-neighbours and two $7(6)$-neighbours or  three $7(5^-)$-neighbours by Lemma \ref{7lem:4-2-vertex-neighbours}(c). In both cases, $v$ receives totally at least $\frac{3}{4}=\min\{2\times\frac{1}{4}+2\times\frac{1}{6}, 3\times\frac{1}{4} \}$ from its $7^+$-neighbours by R6-R8. Thus,   $\mu^*(v)\geq  \frac{3}{4}-2\times\frac{1}{3}>0$ after  $v$ sends $\frac{1}{3}$ to each of its incident $3$-faces by R1.  \medskip

\textbf{(2.3).}  Let $m_3(v)=3$. Suppose first that $m_4(v)=0$. It then follows from Lemma \ref{7lem:4-3-vertex-neighbours}(a) that then $v$ has either four $7(5^-)$-neighbours   or three $7(5^-)$-neighbours and a $7(6)$-neighbour or three $7(5^-)$-neighbours and a $6(4^-)$-neighbour or two $7(5^-)$-neighbours and two $7(6)$-neighbours. In each case, $v$ receives at least $\frac{5}{6}=\min\{4\times \frac{1}{4}, 3\times \frac{1}{4}+\frac{1}{6},3\times \frac{1}{4}+\frac{1}{9}, 2\times\frac{1}{4}+2\times\frac{1}{6} \}$ from its $6^+$-neighbours by R3-R8. In addition, $v$ receives $\frac{1}{5}$ from its incident $5^+$-face by R2.  Thus,   $\mu^*(v)\geq  \frac{5}{6}+ \frac{1}{5} -3\times\frac{1}{3}>0$ after  $v$ sends $\frac{1}{3}$ to each of its incident $3$-faces by R1. Now, we suppose that  $m_4(v)=1$. By Lemma \ref{7lem:4-3-vertex-neighbours}(b), all neighbours of $v$ are $7(5^-)$-vertices, and so $v$ receives totally $4\times \frac{1}{4}$ from its  $7(5^-)$-neighbours by R6-R7. Thus,  $\mu^*(v)\geq  4\times\frac{1}{4} -3\times\frac{1}{3}=0$ after  $v$ sends $\frac{1}{3}$ to each of its incident $3$-faces by R1. \medskip

\textbf{(3).} Let $k=5$.  The initial charge of $v$ is $\mu(v)=d(v)-4=1$. We distinguish three cases  according to the number of $3$-faces incident to $v$ as follows. \medskip

\textbf{(3.1).}  Let $m_3(v)\leq 3$. Obviously, $\mu^*(v)\geq 1-3\times\frac{1}{3}= 0$ after  $v$ sends $\frac{1}{3}$ to each of its incident $3$-faces by R1. \medskip 

\textbf{(3.2).}  Let $m_3(v)=4$. Suppose first that $m_4(v)=0$. By Lemma \ref{7lem:5-4-vertex}(d),  $v$ has two neighbours different from $5^-$- and $6(6)$-vertices, i.e., $v$ has two neighbours consisting of  $6(5^-)$- or $7$-vertices. It follows that $v$ receives totally at least $\frac{1}{6}=\min \{2\times \frac{1}{9},\frac{1}{9}+\frac{1}{12}, 2\times \frac{1}{12} \}$ from its $6^+$-neighbours by R3-R9, and $\frac{1}{5}$ from its incident $5^+$-face by R2.  Thus,  $\mu^*(v)\geq  1+\frac{1}{6}+\frac{1}{5} -4\times\frac{1}{3}>0$ after  $v$ sends $\frac{1}{3}$ to each of its incident $3$-faces by R1. 

Next we suppose that $m_4(v)=1$. By Lemma \ref{7lem:5-4-vertex}(e), $v$ has at least three $6^+$-neighbours different from $6(6)$-vertex. If $v$ has no $7(7)$-neighbour, then  $v$ receives totally at least $3\times \frac{1}{9}$ from its $6^+$-neighbours by R3-R8. %Thus,  $\mu^*(v)\geq  1+2\times \frac{1}{12}+\frac{1}{9} -4\times\frac{1}{3}=0$ after  $v$ sends $\frac{1}{3}$ to each of its incident $3$-faces by R1. 
If $v$ has a $7(7)$-neighbours, then $v$ has four $6^+$-neighbours different from $6(6)$-vertex by Lemma \ref{7lem:5-4-vertex}(f), and so $v$ receives totally at least $4\times \frac{1}{12}$ from its $6^+$-neighbours by R3-R9. Thus,  $\mu^*(v)\geq  1+\frac{1}{3} -4\times\frac{1}{3}=0$ after  $v$ sends $\frac{1}{3}$ to each of its incident $3$-faces by R1. \medskip

\textbf{(3.3).}  Let $m_3(v)=5$. Note that $v$ has no $4^-$-neighbour, and $v$ has at most two $5$-neighbours by Lemma \ref{7lem:5-5-vertex}(a). \medskip

\textbf{(3.3.1).} Let $n_5(v)=0$. By Lemma \ref{7lem:5-5-vertex}(b), $v$ has at most one $6(6)$-neighbour, i.e., $v$ has at least four $6^+$-neighbours different from $6(6)$-vertex. First, suppose that $v$ has a $6(6)$-neighbour. If $v$ has no $6(5^-)$-neighbour, then all the neighbours of $v$ different from $6(6)$-vertex are $7$-vertices,  and by applying R6-R9,  $v$ receives totally at least $4\times\frac{1}{6}$ from its $7$-neighbours. So,   $\mu^*(v)\geq  1+4\times\frac{1}{6} -5\times\frac{1}{3}=0$ after  $v$ sends $\frac{1}{3}$ to each of its incident $3$-faces by R1.
On the other hand, if $v$ has a $6(5^-)$-neighbour, then,  by Lemma \ref{7lem:5-5-vertex}(d), $v$ has three $7$-neighbours, at least one of which is a $7(5^-)$-vertex. 
Then $v$ receives at least $\frac{1}{9}$ from its $6(5^-)$-neighbour by R3-R5, at least $\frac{2}{9}$ from each of its $7(5^-)$-neighbours by R6-R7, at least $\frac{1}{6}$ from each of its other $7$-neighbours by R8-R9. So,  $v$ receives totally at least $\frac{1}{9}+ \frac{2}{9}+2\times\frac{1}{6}=\frac{2}{3}$ from its $6^+$-neighbours.   Thus,   $\mu^*(v)\geq  1+\frac{2}{3} -5\times\frac{1}{3}=0$ after  $v$ sends $\frac{1}{3}$ to each of its incident $3$-faces by R1.

Suppose now that $v$ has no $6(6)$-neighbour. That is, all neighbours of $v$ are $6^+$-vertices different from $6(6)$-vertex. By Lemma \ref{7lem:5-5-vertex}(b), $v$ has at least one $7$-neighbour, say $x$. Note that if $v$ has exactly four $6$-neighbours, then $x$ must be a $7(5^-)$-vertex by Lemma \ref{7lem:5-5-vertex}(c). This means that  $v$ has either four $6$-neighbours and one $7(5^-)$-neighbour or at most three $6$-neighbours and two $7$-neighbours. In each case, $v$ receives totally at least $\frac{2}{3}=\min \{4\times \frac{1}{9}+\frac{2}{9}, 3\times \frac{1}{9}+2\times\frac{1}{6}\}$  from its $6^+$-neighbours by R3-R9. Thus,  $\mu^*(v)\geq  1+\frac{2}{3}-5\times\frac{1}{3}=0$ after  $v$ sends $\frac{1}{3}$ to each of its incident $3$-faces by R1.\medskip

\textbf{(3.3.2).} Let $n_5(v)=1$. Notice that $v$ has no $6(6)$-neighbours by Lemma \ref{7lem:5-5-vertex}(e), i.e., all neighbours of $v$ but one are $7$- or $6(5^-)$-vertices. Moreover, $v$ has at most two $6(5^-)$-neighbours by Lemma \ref{7lem:5-5-vertex}(e). We distinguish three cases  according to the number of $6(5^-)$-vertices adjacent to $v$ as follows. 

First suppose that $v$ has no $6(5^-)$-neighbour. So, $v$ has four $7$-neighbours,  and by applying R6-R9,  $v$ receives totally at least $4\times\frac{1}{6}$ from its $7$-neighbours. So,   $\mu^*(v)\geq  1+4\times\frac{1}{6} -5\times\frac{1}{3}=0$ after  $v$ sends $\frac{1}{3}$ to each of its incident $3$-faces by R1.

Suppose next that $v$ has exactly one $6(5^-)$-neighbour. Then $v$ has three $7$-neighbours, and   by Lemma \ref{7lem:5-5-vertex}(f), one of which is a $7(5^-)$-vertex. 
Thus, $v$ receives totally at least $\frac{1}{9}+2\times \frac{1}{6}+\frac{2}{9}$ from its $6^+$-neighbours by R3-R9.   Hence,  $\mu^*(v)\geq  1+\frac{2}{3} -5\times\frac{1}{3}=0$ after  $v$ sends $\frac{1}{3}$ to each of its incident $3$-faces by R1.

Finally, suppose that $v$ has exactly two $6(5^-)$-neighbours. Then, by Lemma \ref{7lem:5-5-vertex}(g), $v$ has two $7(5^-)$-neighbours. 
So, $v$ receives totally at least $2\times\frac{1}{9}+2\times\frac{2}{9}$ from its $6^+$-neighbours by R3-R7.   Hence,  $\mu^*(v)\geq  1+\frac{2}{3} -5\times\frac{1}{3}=0$ after  $v$ sends $\frac{1}{3}$ to each of its incident $3$-faces by R1.\medskip

\textbf{(3.3.3).} Let $n_5(v)=2$.  Then $v$ has three $7(5^-)$-neighbours by Lemma \ref{7lem:5-5-vertex}(h), and so $v$ receives $\frac{2}{9}$ from each of its $7(5^-)$-neighbours by R6-R7.  Thus,  $\mu^*(v)\geq  1+3\times\frac{2}{9} -5\times\frac{1}{3}=0$ after  $v$ sends $\frac{1}{3}$ to each of its incident $3$-faces by R1. \medskip

\textbf{(4).} Let $k=6$. The initial charge of $v$ is $\mu(v)=d(v)-4=2$. If $m_3(v)\leq 3$, then $\mu^*(v)\geq  2 -3\times \frac{1}{3}-6\times\frac{1}{6}=0$  after  $v$ sends $\frac{1}{3}$ to each of its incident $3$-faces by R1, and $\frac{1}{6}$ to each of its neighbours by R3. Similarly, if  $m_3(v)=4$, then $\mu^*(v)\geq  2 -4\times \frac{1}{3}-6\times\frac{1}{9}=0$  after  $v$ sends $\frac{1}{3}$ to each of its incident $3$-faces by R1, and $\frac{1}{9}$ to each of its neighbours by R4.
On the other hand,  if $m_3(v)=6$, then $\mu^*(v)\geq  2 -6\times\frac{1}{3}=0$ after  $v$ sends $\frac{1}{3}$ to each of its incident $3$-faces by R1.
Therefore, we further assume that $m_3(v)=5$.  By Lemma \ref{7lem:6-5-vertex}, $v$ has at most four bad $5$-neighbours. Notice that if  $v$ has at most three bad $5$-neighbours, then $\mu^*(v)\geq  2 -5\times\frac{1}{3}-3\times\frac{1}{9}=0$ after  $v$ sends $\frac{1}{3}$ to each of its incident $3$-faces by R1, and $\frac{1}{9}$ to each of its bad $5$-neighbours by R5. Next we assume that $v$ has exactly four bad $5$-neighbours. 
If $m_4(v)=0$, then $v$ receives $\frac{1}{5}$ from its incident $5^+$-face by R2, and so $\mu^*(v)\geq  2+\frac{1}{5} -5\times \frac{1}{3}-4\times\frac{1}{9}>0$ after  $v$ sends $\frac{1}{3}$ to each of its incident $3$-face by R1, and $\frac{1}{9}$ to each of its bad $5$-neighbours by R5.
If $m_4(v)=1$, then $v$ has two $7(5^-)$-neighbours by Lemma \ref{7lem:6-5-vertex}, and so $v$ receives at least $\frac{1}{18}$ from each of its $7(5^-)$-neighbours by R6 and R7(e).  
Thus, $\mu^*(v)\geq  2+2\times\frac{1}{18} -5\times \frac{1}{3}-4\times\frac{1}{9}=0$ after  $v$ sends $\frac{1}{3}$ to each of its incident $3$-face by R1, and $\frac{1}{9}$ to each of its bad $5$-neighbours by R5. \medskip

\textbf{(5).} Let $k=7$. The initial charge of $v$ is $\mu(v)=d(v)-4=3$.  Notice first that if $m_3(v)\leq 3$, then we have $\mu^*(v)\geq  3 -3\times\frac{1}{3}-7\times\frac{2}{7}=0$ after  $v$ sends $\frac{1}{3}$ to each of its incident $3$-faces by R1, and $\frac{2}{7}$ to each of its neighbours by R6. Therefore, we may  assume that $4 \leq m_3(v) \leq 7$. \medskip

\textbf{(5.1).}  Let $m_3(v)=4$. By Lemma \ref{7lem:7-vertex}(b), $v$ has at most six $4$-neighbours. If  $n_4(v)\leq 4$, then $v$ sends $\frac{1}{4}$ to each of its bad $4$-neighbours by R7(b), and at most $\frac{2}{9}$ to each of its other neighbours by R7. Consequently, $\mu^*(v)\geq  3 -4\times\frac{1}{4}-3\times\frac{2}{9}-4\times\frac{1}{3}=0$ after  $v$ sends $\frac{1}{3}$ to each of its incident $3$-faces by R1. Suppose now that $5\leq n_4(v)\leq 6$. Clearly, $v$ has no $5(5)$-neighbour, since a $4$ vertex is not adjacent to any $5(5)$-vertex by Lemma \ref{7lem:4-vertex-neighbours}(a). 
If $v$ has a $3$-neighbour $x$, then the edge $xv$ is contained in at least one $5^+$-face by Lemma \ref{7lem:3-vertex}. So, $v$ receives at least $\frac{1}{5}$ from its incident $5^+$-face by R2. Thus $\mu^*(v)\geq  3+\frac{1}{5} -6\times\frac{1}{4}-\frac{1}{5}-4\times\frac{1}{3}>0$ after  $v$ sends $\frac{1}{4}$ to each of its bad $4$-neighbours by R7(b),  at most $\frac{1}{5}$ to each of its neighbours other than $4$-vertex by R7, and $\frac{1}{3}$ to each of its incident $3$-faces by R1.
On the other hand, if $v$ has no $3$-neighbour, then   $\mu^*(v)\geq  3 -6\times\frac{1}{4}-\frac{1}{9}-4\times\frac{1}{3}>0$ after $v$ sends $\frac{1}{4}$ to each of its bad $4$-neighbours by R7(b),  at most $\frac{1}{9}$ to each of its neighbours other than $4$-vertex by R7, and $\frac{1}{3}$ to each of its incident $3$-faces by R1. \medskip

\textbf{(5.2).}  Let $m_3(v)=5$. Notice that $v$ has at most one $3$-neighbour, since a $3$-vertex does not incident to any $3$-faces by Lemma \ref{7lem:3-vertex}. First we assume that $v$ has a $3$-neighbour $x$. In such a case, all $3$-faces incident to $v$ must have a  consecutive ordering. Moreover, the edge $xv$ is contained in at least one $5^+$-face by Lemma \ref{7lem:3-vertex}. So, $v$ receives at least $\frac{1}{5}$ from its incident $5^+$-face by R2. 
Observe first that if $v$ has three $4$-neighbours, then at least one of them must be a support neighbour of $v$. In such a case, $v$ has at most four poor neighbours by Lemma \ref{7lem:7-vertex}(h), and recall that a poor vertex receives at most $\frac{1}{4}$ from $v$ by R7(b),(c).  It then follows that  $\mu^*(v)\geq  3 +\frac{1}{5}-5\times\frac{1}{3}-\frac{1}{5}-4\times\frac{1}{4}-2\times\frac{1}{9}>0$ after $v$ sends $\frac{1}{3}$ to each of its incident $3$-faces by R1, $\frac{1}{5}$ to its $3$-neighbour $x$ by  R7(a), at most $\frac{1}{4}$ to each of its poor neighbours by R7(b),(c), and  at most $\frac{1}{9}$ to each of its other neighbours by R7(d),(e).   
We now suppose that $v$ has at most two $4$-neighbours. It then follows from Lemma \ref{7lem:7-vertex}(c) that $v$ has at most four poor neighbours. Thus, similarly as above, we have $\mu^*(v)\geq  3 +\frac{1}{5}-5\times\frac{1}{3}-\frac{1}{5}-2\times\frac{1}{4}-2\times\frac{2}{9}-2\times\frac{1}{9}>0$ after $v$ sends $\frac{1}{3}$ to each of its incident $3$-faces by R1, $\frac{1}{5}$ to its $3$-neighbour $x$ by  R7(a), at most $\frac{1}{4}$ to each of its bad $4$-neighbours by R7(b), $\frac{2}{9}$ to each of its $5(5)$-neighbours by R7(c), and  at most $\frac{1}{9}$ to each of its other neighbours by R7(d),(e).

We may further assume that $v$ has no $3$-neighbours. By Lemma \ref{7lem:7-vertex}(b), $v$ has at most six $4$-neighbours, i.e., $n_4(v)\leq 6$. On the other hand, recall that $m_3(v)=5$. Therefore, if $v$ has more than four $4$-neighbours, then there exists $v_i\in N(v)$ such that $vv_i$ is contained in two $3$-faces, i.e., $v_i$ is a support neighbour. \medskip

\textbf{(5.2.1).}  Let $ n_4(v)\leq 2$. Note that $v$ has at most four poor neighbours by Lemma \ref{7lem:7-vertex}(c). 
Then $\mu^*(v)\geq  3-5\times \frac{1}{3} -2\times \frac{1}{4}-2\times \frac{2}{9} -3\times\frac{1}{9}>0$ after $v$ sends $\frac{1}{3}$ to each of its incident $3$-faces by R1, $\frac{1}{4}$ to each of its bad $4$-neighbours by R7(b), $\frac{2}{9}$ to each of its $5(5)$-neighbours by R7(c), at most $\frac{1}{9}$ to each of its other neighbours by R7(d),(e). \medskip

\textbf{(5.2.2).}  Let $n_4(v)=3$. By Lemma \ref{7lem:7-vertex}(e), $v$ has at most two $5(5)$-neighbours. Suppose first that $v$ has exactly two $5(5)$-neighbours  $x, y$. Since a $4$-vertex has no $5(5)$-neighbours by Lemma \ref{7lem:4-vertex-neighbours}(a), we deduce that $x$ and $y$ are adjacent. On the other hand, a $5(5)$-vertex has at most one bad $5$-neighbour by Lemma \ref{7lem:5-5-vertex}(b). This implies that the neighbours of $v$ other than $4$- and $5(5)$-vertices are $5(3^-)$- or $6^+$-vertices. Thus we have $\mu^*(v)\geq  3-5\times \frac{1}{3} -3\times \frac{1}{4}-2\times \frac{2}{9} -2\times\frac{1}{18}>0$ after $v$ sends $\frac{1}{3}$ to each of its incident $3$-faces by R1,  $\frac{1}{4}$ to each of its bad $4$-neighbours by R7(b), $\frac{2}{9}$ to each of its $5(5)$-neighbours by R7(c), at most $\frac{1}{18}$ to each of its other neighbours by R7(e). 
Suppose now that $v$ has at most one $5(5)$-neighbour. In such a case,  we have again $\mu^*(v)\geq  3-5\times \frac{1}{3} -3\times \frac{1}{4}- \frac{2}{9} -3\times\frac{1}{9}>0$ after $v$ sends $\frac{1}{3}$ to each of its incident $3$-faces by R1,  $\frac{1}{4}$ to each of its bad $4$-neighbours by R7(b), $\frac{2}{9}$ to each of its $5(5)$-neighbours by R7(c), at most $\frac{1}{9}$ to each of its other neighbours by R7(d),(e). \medskip

\textbf{(5.2.3).}  Let $n_4(v)=4$. By Lemma \ref{7lem:7-vertex}(f), $v$ has at most one $5(5)$-neighbour. First, suppose that $v$ has no $5(5)$-neighbours, then $\mu^*(v)\geq  3-5\times \frac{1}{3} -4\times \frac{1}{4} -3\times\frac{1}{9}=0$ after $v$ sends $\frac{1}{3}$ to each of its incident $3$-faces by R1,  $\frac{1}{4}$ to each of its bad $4$-neighbours by R7(b), at most $\frac{1}{9}$ to each of its other neighbours by R7(d)-(e). We may further assume that $v$ has exactly one $5(5)$-neighbour, say $v_1$. Recall that each of $v_2,v_7$ is a $5^+$-vertex by Lemmas \ref{7lem:3-vertex} and \ref{7lem:4-vertex-neighbours}(a). 
Moreover, for $i\in\{2,7\}$,  if $v_i$ is adjacent to a $4$-vertex, then $v_i$ is different from $5(4)$-vertex by Lemma \ref{7lem:5-4-vertex}(b).  Recall also that a $4(1^+)$-vertex is adjacent to at most one $4$-vertex by Lemma \ref{7lem:4-vertex-neighbours}(b).

Suppose first that $v$ has a support neighbour $x$. If $v$ has no $5(4)$-neighbour, then  $\mu^*(v)\geq  3-5\times \frac{1}{3} -4\times \frac{1}{4} -\frac{2}{9}-2\times\frac{1}{18}=0$ after $v$ sends $\frac{1}{3}$ to each of its incident $3$-faces by R1,  $\frac{1}{4}$ to each of its bad $4$-neighbours by R7(b), $\frac{2}{9}$ to its $5(5)$-neighbour by R7(c), at most $\frac{1}{18}$ to each of its  other neighbours by R7(e). If $v$ has a $5(4)$-neighbour, then $v$ has also a $7$-neighbour by Lemma \ref{7lem:7-vertex}(g).  Thus, $\mu^*(v)\geq  3-5\times \frac{1}{3} -4\times \frac{1}{4} -\frac{2}{9}-\frac{1}{9}=0$ after $v$ sends $\frac{1}{3}$ to each of its incident $3$-faces by R1,  $\frac{1}{4}$ to each of its bad $4$-neighbours by R7(b), $\frac{2}{9}$ to its $5(5)$-neighbour by R7(c),  $\frac{1}{9}$ to its $5(4)$-neighbour by R7(d).

Next, suppose that $v$ has no support neighbour. This means that each $4$-neighbour of $v$ is incident to a $4^+$-face containing $v$.  Obviously, $v$ has one of the two configurations depicted in Figure \ref{7fig:7-5-config}, where $v_i,v_j,v_k,v_\ell$ are $4$-vertices. By Lemma \ref{7lem:4-vertex-neighbours}(a), a $4$-vertex has no $5(5)$-neighbours, so the configuration in Figure \ref{7fig:7-5-config}(a) is not possible for $v$. So the neighbours of $v$ can only form as depicted in Figure \ref{7fig:7-5-config}(b). We then infer that $v$ cannot have any $5(4)$-neighbour by Lemma \ref{7lem:5-4-vertex}(b). Thus $\mu^*(v)\geq  3-5\times \frac{1}{3} -4\times \frac{1}{4} -\frac{2}{9}-2\times\frac{1}{18}=0$ after $v$ sends $\frac{1}{3}$ to each of its incident $3$-faces by R1,  $\frac{1}{4}$ to each of its bad $4$-neighbours by R7(b), $\frac{2}{9}$ to its $5(5)$-neighbour by R7(c), at most $\frac{1}{18}$ to each of its  other neighbours by R7(e). \medskip

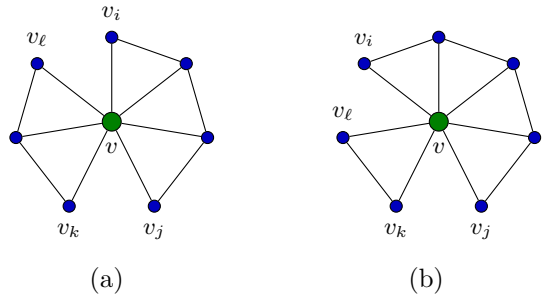
\begin{figure}[htb]
\centering   
\subfigure[]{
\begin{tikzpicture}[scale=.7]
\node [nodr] at (0,0) (v) [label=below: {\scriptsize $v$}] {};
\node [nod2] at (.8,-1.6) (v3)  [label=below:{\scriptsize $v_j$}] {}
	edge  (v);		
\node [nod2] at (-.8,-1.6) (v4)  [label=below:{\scriptsize $v_k$}] {}
	edge [] (v);
\node [nod2] at (-1.8,-.3) (v5)   {}
	edge [] (v)
	edge [] (v4);
\node [nod2] at (-1.4,1.1) (v6)  [label=above:{\scriptsize $v_\ell$}] {}
	edge [] (v)
	edge [] (v5);	
\node [nod2] at (1.4,1.1) (v1)  {}
	edge [] (v);
\node [nod2] at (1.8,-.3) (v2)   {}
	edge [] (v)
	edge [] (v3)
	edge [] (v1);
\node [nod2] at (0,1.6) (v7)   [label=above:{\scriptsize $v_i$}] {}
	edge [] (v)
	edge [] (v1);				
\end{tikzpicture}  }\hspace*{1cm}
\subfigure[]{
\begin{tikzpicture}[scale=.7]
\node [nodr] at (0,0) (v) [label=below: {\scriptsize $v$}] {};
\node [nod2] at (.8,-1.6) (v3)  [label=below:{\scriptsize $v_j$}] {}
	edge  (v);		
\node [nod2] at (-.8,-1.6) (v4) [label=below:{\scriptsize $v_k$}] {}
	edge [] (v);
\node [nod2] at (-1.8,-.3) (v5)   [label=above:{\scriptsize $v_\ell$}] {}
	edge [] (v)
	edge [] (v4);
\node [nod2] at (-1.4,1.1) (v6) [label=above:{\scriptsize $v_i$}] {}
	edge [] (v);	
\node [nod2] at (1.4,1.1) (v1)  {}
	edge [] (v);
\node [nod2] at (1.8,-.3) (v2)   {}
	edge [] (v)
	edge [] (v3)
	edge [] (v1);
\node [nod2] at (0,1.6) (v7)   {}
	edge [] (v)
	edge [] (v1)
	edge [] (v6);				
\end{tikzpicture}  }
\caption{Two possible configurations of a $7(5)$-vertex $v$.}
\label{7fig:7-5-config}
\end{figure}

\textbf{(5.2.4).}  Let $n_4(v)=5$. Then $v$ has no $5(5)$-neighbours by Lemma \ref{7lem:4-vertex-neighbours}(a). Also, $v$ has at most one $5(4)$-neighbour by Lemma \ref{7lem:7-vertex}(i). 
Suppose first that $v$ has a $5(4)$-neighbour or two $6(5)$-neighbours. Then  $v$ is incident to two $5^+$-face by Lemma \ref{7lem:7-vertex}(i). By applying R2, $v$ receives $\frac{1}{5}$ from each of its incident $5^+$-faces.
In such a case, we have $\mu^*(v)\geq  3+2\times\frac{1}{5}-5\times \frac{1}{3} -5\times \frac{1}{4} -\frac{1}{9}-\frac{1}{18}>0$ after $v$ sends $\frac{1}{3}$ to each of its incident $3$-faces by R1,  $\frac{1}{4}$ to each of its bad $4$-neighbours by R7(b), $\frac{1}{9}$ to its $5(4)$-neighbour by R7(d),   $\frac{1}{18}$ to each of its $6(5)$-neighbours by R7(e).
Now we suppose that $v$ has no $5(4)$-neighbour, and $v$ has at most one $6(5)$-neighbour.  Then $\mu^*(v)\geq  3-5\times \frac{1}{3} -5\times \frac{1}{4} -\frac{1}{18}>0$ after $v$ sends $\frac{1}{3}$ to each of its incident $3$-faces by R1,  $\frac{1}{4}$ to each of its bad $4$-neighbours by R7(b),  $\frac{1}{18}$ to its  $6(5)$-neighbour by R7(e). \medskip

\textbf{(5.2.5).} Let $n_4(v)=6$. By Lemma \ref{7lem:7-vertex}(j),  $v$ is incident to two $5^+$-faces, in particular, $v$ has no $5(4)$- or $5(5)$-neighbours. By R2, $v$ receives $\frac{1}{5}$ from each of its incident $5^+$-faces. Thus we have  $\mu^*(v)\geq  3+2\times\frac{1}{5}-5\times \frac{1}{3} -6\times \frac{1}{4} -\frac{1}{18}>0$ after $v$ sends $\frac{1}{3}$ to each of its incident $3$-faces by R1,  $\frac{1}{4}$ to each of its bad $4$-neighbours by R7(b),  $\frac{1}{18}$ to its  $6(5)$-neighbour by R7(e). \medskip

\textbf{(5.3).}  Let $m_3(v)=6$.  By Lemma \ref{7lem:7-vertex}(d), $v$ has at most five poor neighbours. 
If $v$ has at most four poor neighbours, then  $\mu^*(v)\geq  3-6\times \frac{1}{3}-4\times \frac{1}{6} -3\times\frac{1}{9}=0$ after $v$ sends $\frac{1}{3}$ to each of its incident $3$-faces by R1,  $\frac{1}{6}$ to each of its poor neighbours by R8(a),(c),  $\frac{1}{9}$ to each of its  $5(4)$-neighbours by R8(b). Suppose further that  $v$ has exactly five poor neighbours. First, observe that $v$ has no $4(0)$-neighbours as $m_3(v)=6$, i.e., all $4$-neighbours of $v$ are bad $4$-vertices. Note that, by Lemma \ref{7lem:5-4-vertex}(a),(b), a $5(4)$-neighbour of $v$ is not adjacent to two $5(5)$-vertices, and  is not adjacent to both  a $5(5)$- and $4$-vertices. Moreover, a $4$-vertex  has no $5(5)$-neighbours by Lemma \ref{7lem:4-vertex-neighbours}(a). Furthermore,  a bad $4$-vertex  has at most one $4$-neighbour by Lemma \ref{7lem:4-vertex-neighbours}(b), and   has no both $4$- and $5(4)$-neighbour by Lemma \ref{7lem:4-vertex-neighbours}(c). In addition, by Lemma \ref{7lem:5-5-vertex}(b), a $5(5)$-vertex is adjacent to at most one $5(5)$-vertex.
All those facts imply that $v$ has at most one $5(4)$-neighbour. Hence, $\mu^*(v)\geq  3-6\times \frac{1}{3}-5\times \frac{1}{6} -\frac{1}{9}>0$ after $v$ sends $\frac{1}{3}$ to each of its incident $3$-faces by R1,  $\frac{1}{6}$ to each of its poor neighbours by R8(a),(c),  $\frac{1}{9}$ to its  $5(4)$-neighbour by R8(b). \medskip

\textbf{(5.4).}  Let $m_3(v)=7$. By Lemma \ref{7lem:5-5-vertex}(b), a $5(5)$-vertex is adjacent to at most one $5(5)$-vertex. So, we deduce that $v$ has at most four $5(5)$-neighbours.  In particular, if $v$ has four such neighbours, then $v$ cannot have any $5(4)$-neighbours by Lemmas \ref{7lem:5-4-vertex}(a) and \ref{7lem:5-5-vertex}(b). Thus $\mu^*(v)\geq  3-7\times \frac{1}{3}-4\times \frac{1}{6}=0$ after  $v$ sends $\frac{1}{3}$ to each of its incident $3$-faces by R1 and $\frac{1}{6}$ to each of its $5(5)$-neighbours by R9(a). 
If $v$ has three $5(5)$-neighbours, then  $v$ has at most two  $5(4)$-neighbours by Lemmas \ref{7lem:5-4-vertex}(a) and \ref{7lem:5-5-vertex}(b), and so $\mu^*(v)\geq  3-7\times \frac{1}{3}-3\times \frac{1}{6}-2\times \frac{1}{12}=0$  after  $v$ sends $\frac{1}{3}$ to each of its incident $3$-face by R1, $\frac{1}{6}$ to each of its $5(5)$-neighbour by R9(a), and  $\frac{1}{12}$ to each of its $5(4)$-neighbours by R9(b).
Suppose now that $v$ has at most two $5(5)$-neighbours. By Lemma \ref{7lem:5-5-vertex}(b), a $5(5)$-vertex has at most one bad $5$-neighbour. We then deduce that $v$ has either seven $5(4)$-neighbours or at most  six bad $5$-neighbours. In the worst case, $\mu^*(v)\geq  3-7\times \frac{1}{3}-2\times \frac{1}{6}-4\times \frac{1}{12}=0$  after  $v$ sends $\frac{1}{3}$ to each of its incident $3$-face by R1, $\frac{1}{6}$ to each of its $5(5)$-neighbour by R9(a), and  $\frac{1}{12}$ to each of its $5(4)$-neighbours by R9(b).\medskip

%%%%%%%%%%%%%%%%%%%%%%%%%%%%%%%%%%%%%%%%%%%%%%%%%%%%%%%%%%%%%%%%%%%%%%%%%%%%%%%

%%%%%%%%%%%%%%%%%%%%%%%%%%%%%%%%%%%%%%%%%%%%%%%%%%%%%%%%%%%%%%%%%%%%%%%%%%%%%%%%%%

%%%%%%%%%%%%%%%%%%%%%%%%%%%%%%%%%%%%%%%%%%%%%%%%%%%%%%%%%%%%%%%%%%%%%%%%%%%%%%%

%%%%%%%%%%%%%%%%%%%%%%%%%%%%%%%%%%%%%%%%%%%%%%%%%%%%%%%%%%%%%%%%%%%%%%%%%%%%%%%%%%

%%%%%%%%%%%%%%%%%%%%%%%%%%%%%%%%%%%%%%%%%%%%%%%%%%%%%%%%%%%%%%%%%%%%%%%%%%%%%%%

%%%%%%%%%%%%%%%%%%%%%%%%%%%%%%%%%%%%%%%%%%%%%%%%%%%%%%%%%%%%%%%%%%%%%%%%%%%%%%%%%%

\subsection{The case \texorpdfstring{$\D=8$}{D8} } \label{sub:8}~~\medskip

Recall that $G$ does not admit any $2$-distance $23$-coloring, whereas any planar graph $G'$ obtained from $G$ with smaller value of $|V(G')|+|E(G')|$ admits a $2$-distance $23$-coloring. \medskip

We begin by establishing several structural properties of $G$, analogous to those obtained for the cases  $\Delta=6$ and $\Delta=7$.  
The proof of the following lemma is omitted, as it follows from the same arguments used in Lemma~\ref{6lem:min-deg-4}.

\begin{lemma}\label{8lem:min-deg-3}
$\delta(G)\geq 3$.
\end{lemma}

\begin{lemma}\label{8lem:3-vertex}
Let $v$ be a $3$-vertex. Then $m_3(v)=0$ and $m_4(v)\leq 1$. In particular,
\begin{itemize}
\item[$(a)$] $v$ has no $6^-$-neighbour,
\item[$(b)$] if $m_4(v)=0$, then $v$ has two $8(6^-)$-neighbours,
\item[$(c)$] if $m_4(v)=1$, then $v$ has three $8(6^-)$-neighbours.
\end{itemize}
\end{lemma}

\begin{proof}
If $v$ is incident to a $3$-face, say $f_1=v_1 v v_2$, then we set $G'=G-v+\{v_1v_3\}$.  
If $v$ is incident to two $4$-faces, say $f_1=v_1vv_2x$ and $f_2=v_2vv_3y$, then we again set $G'=G-v+\{v_1v_3\}$.  
If $v$ is adjacent to a $7^-$-vertex, say $v_1$, then we set $G'=G-v+\{v_1v_2,v_1v_3\}$.
In each case, $G'$ remains proper with respect to $G$, and moreover $d_2(v)\leq 22$ by Lemma~\ref{lem:d2v}.  
This yields a contradiction by Remark~\ref{rem:proper}.  
Thus, none of the above configurations can occur.  
\end{proof}

Since a $3$-vertex cannot have any $6^-$-neighbour by Lemma~\ref{8lem:3-vertex}(a), it follows that every vertex of degree $4$, $5$, or $6$ has no $3$-neighbour.  
This fact will be assumed in the remainder of this section.

\begin{lemma}\label{8lem:4-vertex-neighbours}
Let $v$ be a $4$-vertex with $m_3(v)\geq 1$.
\begin{itemize}
\item[$(a)$] If $m_3(v)=1$, then $v$ has no $5(5)$-neighbour. In particular, $v$ has at most one $4$-neighbour.
\item[$(b)$] If $m_3(v)=2$, then $v$ has at most one $5^-$-neighbour.
\item[$(c)$] If $m_3(v)\geq 3$, then $v$ has neither a $4$-neighbour nor a $5(4^+)$-neighbour.
\end{itemize}
\end{lemma}

\begin{proof}
$(a)$ Let $m_3(v)=1$.  
Clearly, $v$ cannot have any $5(5)$-neighbour since $m_3(v)=1$.  
Now assume that $v$ has two $4$-neighbours. Then $d_2(v)\leq 22$ by Lemma~\ref{lem:d2v}.  
If we set $G'=G-v+\{v_1v_2, v_1v_3, v_1v_4\}$ for a $4$-neighbour $v_1$ of $v$, then $G'$ is proper with respect to $G$.  
By Remark~\ref{rem:proper}, this yields a contradiction.

\smallskip
$(b)$ Let $m_3(v)=2$.  
Suppose that $v$ has two neighbours that are $4$- or $5$-vertices.  
Then $d_2(v)\leq 22$ by Lemma~\ref{lem:d2v}, and the same reduction as in $(a)$ leads to a contradiction.

\smallskip
$(c)$ Let $m_3(v)\geq 3$.  
If $v$ has a $4$-neighbour or a $5(4^+)$-neighbour, then again $d_2(v)\leq 22$ by Lemma~\ref{lem:d2v}.  
As in the previous cases, this contradicts Remark~\ref{rem:proper}.
\end{proof}

\begin{lemma}\label{8lem:4-1-vertex-neighbours}
Let $v$ be a $4(1)$-vertex.  
If $m_4(v)\geq 2$, then $v$ has two $7^+$-neighbours.
\end{lemma}

\begin{proof}
Since $m_3(v)=1$, the vertex $v$ is incident to a $3$-face $f_1=v_1 v v_2$.
Assume that $m_4(v)\geq 2$.  
Suppose, for a contradiction, that $v$ has at most one $7^+$-neighbour.  
Then at least one of $v_1,v_2$ must be a $6^-$-vertex, say $v_1$.  
In particular, $d_2(v)\leq 22$ by Lemma~\ref{lem:d2v}.  
If we set $G'=G-v+\{v_1v_3, v_1v_4\}$, then $G'$ is proper with respect to $G$, a contradiction to Remark~\ref{rem:proper}.  
\end{proof}

\begin{lemma}\label{8lem:4-2-vertex-neighbours}
Let $v$ be a $4(2)$-vertex.
\begin{itemize}
\item[$(a)$] If $m_4(v)=0$, then $v$ has two $7^+$-neighbours that are neither $7(7)$- nor $8(8)$-vertices.
\item[$(b)$] If $m_4(v)=1$, then $v$ has either two $8(7^-)$-neighbours or three $7^+$-neighbours different from $7(7)$- and $8(8)$-vertices.
\item[$(c)$] If $m_4(v)=2$, then $v$ has one of the following:
 one $8(7^-)$- and three $7(6^-)$-neighbours, or
two $8(7^-)$- and one $7(6^-)$-neighbour, or
 three $8(7^-)$-neighbours.
\end{itemize}
\end{lemma}

\begin{proof}
Let $f_i=v_i v v_{i+1}$ and $f_j=v_j v v_{j+1}$ with $i<j$ be the two $3$-faces incident to $v$.  
Notice first that  if $d_2(v)\leq 22$, then
we set $G'=G-v+\{v_2v_4\}$ when $(i,j)=(1,2)$; we set   
$G'=G-v+\{v_2v_3,\,v_1v_4\}$ when $(i,j)=(1,3)$. In each case, $G'$ is proper with respect to $G$. By Remark~\ref{rem:proper}, this yields a contradiction. Thus we further assume that $d_2(v)\geq 23$. 
We also note that if $v$ has a $7(7)$- or $8(8)$-neighbour, then at least two edges in $E(v)$ lie in two $3$-faces. Moreover,  $v$ cannot have two neighbours consisting of $7(7)$-, or $8(8)$-vertices since $m_3(v)=2$.

\smallskip
$(a)$  Let $m_4(v)=0$.  
Assume for a contradiction that $v$ has at most one $7^+$-neighbour that is not $7(7)$- or $8(8)$-vertices; equivalently, $v$ has three neighbours consisting of $6^-$-, $7(7)$-, or $8(8)$-vertices.  
Then $d_2(v)\leq 22$ by Lemma~\ref{lem:d2v}, a contradiction.  

\smallskip
$(b)$  Let $m_4(v)=1$.  
First suppose that $v$ has no  $7(7)$- or $8(8)$-neighbours.  
If $v$ has either  two $6^-$-neighbours and one $7$-neighbour, or  three $6^-$-neighbours, then $d_2(v)\leq 22$ by Lemma~\ref{lem:d2v}, giving a contradiction.  
Hence $v$ has either two $8(7^-)$-neighbours or three $7^+$-neighbours.
Now suppose $v$ has a $7(7)$- or $8(8)$-neighbour.  
If $v$ has two $6^-$-neighbours, then $d_2(v)\leq 21$, and if $v$ has one $6^-$- and one $7$-neighbour (other than $7(7)$), then $d_2(v)\leq 22$; in both cases we reach a contradiction as before.  
Thus the statement follows.

\smallskip
$(c)$  Let $m_4(v)=2$.  
We first show that $v$ has at least one $8(7^-)$-neighbour.  
Indeed, if all neighbours of $v$ were $7^-$- or $8(8)$-vertices, then $d_2(v)\leq 22$ by Lemma~\ref{lem:d2v}, a contradiction.  
Thus $v$ has at least one $8(7^-)$-neighbour.
If $v$ has exactly one $8(7^-)$-neighbour and exactly two $7(6^-)$-neighbours, then $d_2(v)\leq 22$.  
Likewise, if $v$ has exactly two $8(7^-)$-neighbours and no $7(6^-)$-neighbours, then again $d_2(v)\leq 22$.  
In each case we obtain a contradiction.  
Therefore, the remaining possibilities are those listed in the statement:  
$v$ has either one $8(7^-)$- and three $7(6^-)$-neighbours, or two $8(7^-)$- and one $7(6^-)$-neighbour, or three $8(7^-)$-neighbours.
\end{proof}

\begin{lemma}\label{8lem:4-3-vertex-neighbours}
Let $v$ be a $4(3)$-vertex.
\begin{itemize}
\item[$(a)$] If $m_4(v)=0$, then $v$ has one of the following:
one $8(7^-)$- and three $7(6^-)$-neighbours, or
 two $8(7^-)$- and one $7(6^-)$-neighbour, or
two $8(7^-)$- and two $8(8)$-neighbours, or
three $8(7^-)$-neighbours.

\item[$(b)$] If $m_4(v)=1$, then $v$ has either
two $8(7^-)$- and two $7(6^-)$-neighbours, or
 three $8(7^-)$-neighbours.
\end{itemize}
\end{lemma}

\begin{proof}
Let $f_i = v_i v v_{i+1}$ for $i\in [3]$ be $3$-faces incident to $v$.
Notice first that if $d_2(v)\le 22$, then $G' = G - v + \{v_1 v_4\}$  a graph that is proper with respect to $G$, contradicting Remark~\ref{rem:proper}.  Thus we assume that $d_2(v)\ge 23$.

$(a)$  Let $m_4(v)=0$.  
We first show that $v$ has at least one $8(7^-)$-neighbour.  
Indeed, if all neighbours of $v$ were $7^-$- or $8(8)$-vertices, then $d_2(v)\le 22$ by Lemma~\ref{lem:d2v}, a contradiction. 
Thus $v$ has at least one $8(7^-)$-neighbour.

Suppose first that $v$ has exactly one $8(7^-)$-neighbour.  
If $v$ has no $7(6^-)$-neighbours, then $d_2(v)\le 22$.  
If it has exactly one $7(6^-)$-neighbour and two neighbours consisting of $6^-$-, $7(7)$-, or $8(8)$-vertices, then again $d_2(v)\le 22$.  
Similarly, if $v$ has two $7(6^-)$-neighbours and one neighbour consisting of $6^-$-, $7(7)$-, or $8(8)$-vertices, we again obtain $d_2(v)\le 22$.  
Each case contradicts Remark~\ref{rem:proper}.  
Hence when $v$ has exactly one $8(7^-)$-neighbour, it must have three $7(6^-)$-neighbours.

Suppose next that $v$ has exactly two $8(7^-)$-neighbours. 
If $v$ has no $7(6^-)$-neighbours, then $v$ must have two $8(8)$-neighbours; otherwise $d_2(v)\le 22$.  
Thus, when $v$ has two $8(7^-)$-neighbours, it has either one $7(6^-)$-neighbour or two $8(8)$-neighbours.

\smallskip
$(b)$  Let $m_4(v)=1$.  
If $v$ had at most one $8(7^-)$-neighbour, then $d_2(v)\le 22$ by Lemma~\ref{lem:d2v}, yielding a contradiction.  
Thus $v$ has at least two $8(7^-)$-neighbours.
If $v$ had two neighbours consisting of $6^-$-, $7(7)$-, or $8(8)$-vertices, then again $d_2(v)\le 22$.  
Similarly, if $v$ had two $8(7^-)$-neighbours, one $7(6^-)$-neighbour, and one neighbour consisting of $6^-$-, $7(7)$-, or $8(8)$-vertices, then $d_2(v)\le 22$. Both situations contradict with our assumption. 
\end{proof}

\begin{lemma}\label{8lem:4-4-vertex-neighbours}
Let $v$ be a $4(4)$-vertex.  
Then $v$ has either three $8(7^-)$- and one $7(5^-)$-neighbour, or four $8(7^-)$-neighbours.
\end{lemma}

\begin{proof}
Note that if $d_2(v)\le 22$, then $G' = G - v $   is proper with respect to $G$, contradicting Remark~\ref{rem:proper}.  Thus we assume that $d_2(v)\ge 23$.
If $v$ has a neighbour consisting of $6^-$-, $7(6^+)$-, or $8(8)$-vertex, then $d_2(v)\le 22$ by Lemma~\ref{lem:d2v}, a contradiction. 
Furthermore, if $v$ has two $7^-$-neighbours, then $d_2(v)\le 22$ once again.  
Thus all neighbours of $v$ must be $8(7^-)$- or $7(5^-)$-vertices, and at most one neighbour can be of type $7(5^-)$.  
Hence $v$ has either three $8(7^-)$- and one $7(5^-)$-neighbour, or four $8(7^-)$-neighbours.
\end{proof}

\begin{lemma}\label{8lem:5-4-vertex}
Let $v$ be a $5(4)$-vertex. 
\begin{itemize}
\item[$(a)$] If $m_4(v)=0$, then $v$ has two $6^+$-neighbours that are not $6(6)$-vertices.
\item[$(b)$] If $m_4(v)=1$, then $v$ has three $6^+$-neighbours that are not $6(6)$-vertices.
\end{itemize}
\end{lemma}

\begin{proof}
Let $f_i = v_i v v_{i+1}$ for $i\in[4]$.

\smallskip
$(a)$ 
Assume $m_4(v)=0$, and suppose that $v$ has at most one $6^+$-neighbour other than a $6(6)$-vertex.  
Thus all but one neighbour of $v$ are $6(6)$- or $5^-$-vertices. 
Observe that $v$ cannot have four $6(6)$-neighbours.
We now check all possibilities:  
three $6(6)$-neighbours;  
two $6(6)$- and one $5^-$-neighbour;  
one $6(6)$- and two $5^-$-neighbours;  
or four $5^-$-neighbours.  
In every case $d_2(v)\le 22$ by Lemma~\ref{lem:d2v}.  
Setting $G'=G-v+\{v_1v_5\}$ yields a graph that is proper with respect to $G$, contradicting Remark~\ref{rem:proper}.  
Therefore $v$ must have at least two $6^+$-neighbours distinct from $6(6)$-vertices.

\smallskip
$(b)$  
Assume $m_4(v)=1$, and suppose that $v$ has three neighbours $x,y,z$ that are either $5^-$- or $6(6)$-vertices.  
Checking the possibilities  
(all of $x,y,z$ are $5^-$-vertices; two $5^-$- and one $6(6)$-vertices; one $5^-$- and two $6(6)$-vertices; or all $6(6)$-vertices)  
gives $d_2(v)\le 22$ in every case by Lemma~\ref{lem:d2v}.  
This again contradicts Remark~\ref{rem:proper}.  
Thus $v$ must have at least three $6^+$-neighbours that are not $6(6)$-vertices.
\end{proof}

\begin{lemma}\label{8lem:5-vertex}
Let $v$ be a $5(5)$-vertex. 
\begin{itemize}
\item[$(a)$] If $v$  is adjacent to another $5(5)$-vertex, then $v$ has no $4$- or $5(4)$-neighbours.
\item[$(b)$]  $v$ has at most one neighbour that is a $4$- or $5(5)$-vertex.
\end{itemize}
\end{lemma}

\begin{proof}
$(a)$ 
Suppose that $v$ is a $5(5)$-vertex adjacent to a $5(5)$-vertex.  
If $v$ also has a $4$- or $5(4)$-neighbour, then $d_2(v)\le 21$ by Lemma~\ref{lem:d2v}.  
Setting $G'=G-v$ yields a graph that is proper with respect to $G$, contradicting Remark~\ref{rem:proper}.  
Thus $v$ has no such neighbours.

\smallskip
$(b)$ 
Suppose that $v$ has two neighbours consisting of  $4$- or  $5(5)$-vertices.  
If one of them is a $5(5)$-vertex, part~(a) already gives a contradiction.  
Otherwise, $v$ has two $4$-neighbours, implying $d_2(v)\le 22$ by Lemma~\ref{lem:d2v}.  
As above, this yields a contradiction.  
Thus $v$ has at most one such neighbour.
\end{proof}

\begin{lemma}\label{8lem:5-5-vertex}
Let $v$ be a $5(5)$-vertex. 
\begin{itemize}
\item[$(a)$] If $n_4(v)=1$, then $v$ has three $7^+$-neighbours that are not $7(7)$-vertices.
\item[$(b)$] If $n_4(v)=n_5(v)=0$, then $v$ has either three $8(7^-)$-neighbours or four $6^+$-neighbours different from $6(6)$-vertices such that two of which are $7^+$-vertices different from $7(7)$-vertex.
\item[$(c)$] If $n_4(v)=0$ and $n_5(v)=1$, then $v$ has either three $7^+$-neighbours different from $7(7)$-vertex or two $6(5^-)$- and two $7^+$-neighbours different from $7(7)$-vertex.
\item[$(d)$] If $n_4(v)=0$ and $n_5(v)=2$, then $v$ has three $7^+$-neighbours different from $7(7)$-vertex.
\end{itemize}
\end{lemma}

\begin{proof}
Let $f_i = v_i v v_{i+1}$ for $i\in[5]$ in cyclic order.

$(a)$  
Assume $n_4(v)=1$ and that $v$ has at most two $7^+$-neighbours different from $7(7)$-vertices.  
Then $v$ has two neighbours consisting of  $5$-, $6$-, or $7(7)$-vertices, which forces $d_2(v)\le 22$ by Lemma~\ref{lem:d2v}.  
Removing $v$ yields a graph that is proper with respect to $G$, contradicting Remark~\ref{rem:proper}.  
Thus $v$ must have three such $7^+$-neighbours.

\smallskip
$(b)$  
Let $n_4(v)=n_5(v)=0$.  
Assume that $v$ has at most two $8(7^-)$-neighbours.  
If $v$ has at most three $6^+$-neighbours other than $6(6)$-vertices, then it must have two $6(6)$-neighbours.  
This forces  $d_2(v)\le 22$ by Lemma~\ref{lem:d2v}, a contradiction.  
Hence $v$ has four $6^+$-neighbours different from $6(6)$-vertices.
Among these four, if only one is a $7^+$-vertex different from $7(7)$-vertex, then again $d_2(v)\le 22$ by Lemma~\ref{lem:d2v}, yielding the same contradiction.  
Therefore at least two must be such $7^+$-vertices.

\smallskip
$(c)$ 
Let $n_4(v)=0$ and $n_5(v)=1$.  
If $v$ has at most one $7^+$-neighbour different from $7(7)$-vertex, then $d_2(v)\le 22$, a contradiction.  
Thus $v$ has at least two such $7^+$-neighbours.

Assume now that $v$ has exactly two $7^+$-neighbours different from $7(7)$-vertex.  
If $v$ has at most one $6(5^-)$-neighbour, then again $d_2(v)\le 22$ by Lemma~\ref{lem:d2v}.  
Thus $v$ must have two $6(5^-)$-neighbours in this configuration.

\smallskip
$(d)$  
Let $n_4(v)=0$ and $n_5(v)=2$.  
If $v$ has at most two $7^+$-neighbours different from $7(7)$-vertices, then $d_2(v)\le 22$, again contradicting Remark~\ref{rem:proper}.  
Hence $v$ has three such neighbours.
\end{proof}

\begin{lemma}\label{8lem:6-5-vertex}
Let $v$ be a $6(5)$-vertex. Then $v$ has at most four $5(4^+)$-neighbours.  
In particular, if $m_4(v)=1$ and $v$ has four $5(4^+)$-neighbours, then $v$ has two $8(6^-)$-neighbours.
\end{lemma}

\begin{proof}
Let $f_i = v_i v v_{i+1}$ for $i \in [5]$.
Assume first that $v$ has five $5(4^+)$-neighbours.  
Then at least three edges of $E(v)$ are contained in two $3$-faces, and Lemma~\ref{lem:d2v} gives $d_2(v)\le 20$.  
Since $v$ has five $5$-neighbours, one of $v_1,v_6$ must be a $5$-vertex, say $v_1$, and one of $v_3,v_5$ is also a $5$-vertex, say $v_3$.  
Let  
$G' = G - v + \{\,v_1v_6,\ v_3v_1,\ v_3v_5\,\}$.
Clearly, $G'$ is proper with respect to $G$, contradicting Remark~\ref{rem:proper}.  
Thus, $v$ has at most four $5(4^+)$-neighbours.

Now suppose that $m_4(v)=1$ and that $v$ has four $5(4^+)$-neighbours.  
Then there are at least two edges in $E(v)$ each contained in two $3$-faces.  
If $v$ has at most one $8(6^-)$-neighbour, then $d_2(v)\le 22$ by Lemma~\ref{lem:d2v}, and the same reduction as above yields a contradiction.  
Hence $v$ must have at least two $8(6^-)$-neighbours.
\end{proof}

\medskip

Let $v$ be a $7$-vertex (resp.\ an $8$-vertex), and let $x$ be a $5^-$-neighbour (resp.\ a $4^-$-neighbour) of $v$.  
If $v$ and $x$ have two common neighbours, (equivalently, if the edge $vx$ lies in two $3$-faces), then $x$ is called a \emph{support neighbour} of $v$.

\begin{lemma}\label{8lem:7-vertex}
Let $v$ be a $7^+$-vertex.  
\begin{itemize}
\item[$(a)$]  If $v$ has a support neighbour, then $d_2(v)\ge 23$.
\item[$(b)$] If $v$ is a $7(6)$-vertex, and has five $4(1^+)$-neighbours, then $v$ cannot have any $5$-neighbour. 
\item[$(c)$] If $v$ is a $7(6)$-vertex, and has a $5(5)$-neighbour, then $v$  has at most  four neighbours consisting of  $4$- or $5(4)$-vertices.
\item[$(d)$] If $v$ is a $7(6)$-vertex, and has two $5(5)$-neighbours, then $v$ has at most three neighbours consisting of  $4$- or $5(4)$-vertices.
\item[$(e)$] If $v$ is a $7(6)$-vertex, and has three $5(5)$-neighbours, then $v$ has at most one $4$-neighbour. In particular, $v$ has at most two neighbours consisting of  $4$- or $5(4)$-vertices. 
\item[$(f)$] If $v$ is a $7(7)$-vertex, then $v$ has at most five $5(4^+)$-neighbours.
\item[$(g)$] If $v$ is a $8(7)$-vertex, and has a support neighbour, then $v$ has at most six neighbours consisting of  $4(2^-)$- or $5(5)$-vertices. 
\end{itemize}
\end{lemma}

\begin{proof}
$(a)$ Let $d(v)=k$ and suppose that $v_1$ is a support neighbour of $v$.  
By definition, $v_1$ is adjacent to both $v_2$ and $v_k$.  
Assume for a contradiction that $d_2(v)\leq 22$.  
Consider
$
G' = G - v + \{v_1v_3, v_1v_4, \ldots, v_1v_{k-1} \}.
$
The graph $G'$ is proper with respect to $G$, contradicting Remark~\ref{rem:proper}.  
Thus, $d_2(v)\ge 23$.

\smallskip
$(b)$ 
Suppose that $v$ is a $7(6)$-vertex, and $v$ has five $4(1^+)$-neighbours. Clearly, $v$ has a support neighbour. If $v$ had a  $5$-neighbour, then it would be $d_2(v)\le 21$, contradicting~(a).  

\smallskip
$(c)$ 
Suppose that $v$ is a $7(6)$-vertex, and $v$ has a $5(5)$-neighbour. Clearly, $v$ has a support neighbour. If $v$ had five neighbours consisting of  $4$- or $5(4)$-vertices, then it would be $d_2(v)\le 22$, contradicting~(a).  

\smallskip
$(d)$ 
Suppose that $v$ is a $7(6)$-vertex with two $5(5)$-neighbours. Clearly, $v$ has a support neighbour. If $v$ had four neighbours consisting of  $4$- or $5(4)$-vertices, then  $d_2(v)\le 22$, contradicting part~(a).  

\smallskip
$(e)$ 
Suppose that $v$ is a $7(6)$-vertex with three $5(5)$-neighbours. Clearly, $v$ has a support neighbour. 
Since a $5(5)$-vertex is adjacent to at most one $5(5)$-vertex by Lemma~\ref{8lem:5-vertex}(b),  
$v$ can have at most one $4$-neighbour; otherwise $d_2(v)\le 22$, contradicting~(a).  
Similarly, $v$ cannot have three neighbours from $\{4\text{-vertices},\,5(4)\text{-vertices}\}$,  
as this would again imply $d_2(v)\le 22$, contradicting~(a). 

\smallskip
$(f)$ 
Assume,  for a contradiction, that $v$ has six $5(4^+)$-neighbours.
Clearly, at least one of such $5(4^+)$-neighbour is a support neighbour of $v$, 
and Lemma~\ref{lem:d2v} then gives $d_2(v)\le 21$,  
contradicting~(a).

\smallskip
$(g)$ 
Suppose that $v$ has seven such neighbours.
Then Lemma~\ref{lem:d2v} yields $d_2(v)\le 22$, again contradicting~(a).  
\end{proof}

\medskip

We now apply discharging to show that $G$ does not exist.  
We use the same initial charges as before, together with the following discharging rules. \medskip

\noindent  
\textbf{{Discharging Rules}} \medskip

We apply the following discharging rules.

\begin{itemize}
\setlength\itemsep{.2em}
\item[\textbf{R1:}] Every $3$-face  receives $\frac{1}{3}$  from each of its incident vertices.
\item[\textbf{R2:}] Every $5^+$-face gives $\frac{1}{5}$  to each of its incident vertices.
\item[\textbf{R3:}] Every $3$-vertex receives $\frac{1}{5}$  from each of its $8(6^-)$-neighbour.
\item[\textbf{R4:}] Every $4(4)$-vertex receives $\frac{1}{3}$  from each of its $7(5^-)$-neighbours.
\item[\textbf{R5:}] Every $6(5)$-vertex receives $\frac{1}{18}$ from each of its $8(6^-)$-neighbours.
\item[\textbf{R6:}] Every $6(5^-)$-vertex gives $\frac{1}{9}$  to each of its bad $5$-neighbours.
\item[\textbf{R7:}] Let $v$ be a $7(6^-)$-vertex. Then, $v$ gives
\begin{itemize}
\item[$(a)$]  $\frac{1}{6}$  to each of its $4(1)$-, $4(2)$- and $4(3)$-neighbours,
\item[$(b)$] $\frac{1}{9}$  to each of its $5(4)$-neighbours,
\item[$(c)$]  $\frac{2}{9}$  to each of its $5(5)$-neighbours,
\end{itemize}

\item[\textbf{R8:}] Every $7(7)$-vertex gives $\frac{1}{9}$  to each of its bad $5$-neighbours,

\item[\textbf{R9:}] Let $v$ be a $8(7^-)$-vertex. Then, $v$ gives
\begin{itemize}
\item[$(a)$]  $\frac{1}{6}$  to each of its $4(1)$-neighbours,
\item[$(b)$]  $\frac{1}{4}$  to each of its $4(2)$-neighbours,
\item[$(c)$]  $\frac{1}{3}$  to each of its $4(3^+)$-neighbours,
\item[$(d)$] $\frac{1}{9}$  to each of its $5(4)$-neighbours,
\item[$(e)$] $\frac{2}{9}$  to each of its $5(5)$-neighbours,
\end{itemize}

\item[\textbf{R10:}] Let $v$ be an $8(8)$-vertex. Then, $v$ gives
\begin{itemize}
\item[$(a)$]  $\frac{1}{9}$  to each of its $4(3)$-, and $5(4)$--neighbours,
\item[$(b)$] $\frac{2}{9}$  to each of its $5(5)$-neighbours,
\end{itemize}

\end{itemize}

\vspace*{1em}

\noindent
\textbf{Checking} $\mu^*(v), \mu^*(f)\geq 0$ for $v\in V(G), f\in F(G)$.\medskip

First we show that $\mu^*(f)\ge 0$ for each $f\in F(G)$. Given a face $f\in F(G)$,
if $f$ is a $3$-face, then it receives $\frac{1}{3}$ from each of its incident vertices by R1, and so $\mu^*(f)=\ell(f)-4+3\times \frac{1}{3}=0$. 
If $f$ is a $4$-face, then $\mu(f)=\mu^*(f)=0$.
Let $f$ be a $5^+$-face. By applying R2, $f$ sends $\frac{1}{5}$ to each of its incident vertices. 
It then follows that $\mu^*(f)\ge \frac{4\ell(f)}{5}-4\ge 0$.
Consequently, $\mu^*(f)\ge 0$ for each $f\in F(G)$.  \medskip

We now pick a vertex $v\in V(G)$ with $d(v)=k$. By Lemma~\ref{8lem:min-deg-3}, we have $k\ge 3$. \medskip

\textbf{(1).} Let $k=3$. The initial charge of $v$ is $\mu(v)=d(v)-4=-1$. By Lemma \ref{8lem:3-vertex}, $m_3(v)=0$ and $m_4(v)\leq 1$. If $m_4(v)=0$, then $v$ is adjacent to two $8(6^-)$-vertices by Lemma \ref{8lem:3-vertex}(b), and if $m_4(v)=1$, then $v$ is adjacent to three $8(6^-)$-neighbours by Lemma \ref{8lem:3-vertex}(c). Thus $\mu^*(v)\geq -1+3\times \frac{1}{5}+2\times\frac{1}{5}=0$ after $v$ receives  $\frac{1}{5}$ from each of its incident $5^+$-faces by R2, and $\frac{1}{5}$ from each of its $8(6^-)$-neighbours by R3. \medskip

\textbf{(2).} Let $k=4$. The initial charge of $v$ is $\mu(v)=d(v)-4=0$.  If $m_3(v)=0$, then $\mu^*(v)\geq 0$, since $v$ does not give any charge to its incident faces. Therefore, we may assume that  $1 \leq m_3(v) \leq 4 $. \medskip

\textbf{(2.1).}  Let $m_3(v)= 1$. If $m_4(v)\leq 1$, then $v$ is incident to two $5^+$-faces, and by R2, $v$ receives $\frac{1}{5}$ from each of those $5^+$-faces. Thus,  $\mu^*(v)\geq 2\times \frac{1}{5}-\frac{1}{3}>0$ after  $v$ sends $\frac{1}{3}$ to its incident $3$-face by R1.
If $m_4(v)\geq 2$, then  $v$ has two $7^+$-neighbours by Lemma \ref{8lem:4-1-vertex-neighbours}, which are clearly different from $7(7)$- and $8(8)$-vertices as $m_3(v)= 1$. It follows that $v$ receives $\frac{1}{6}$ from each of its $7(6^-)$- and $8(7^-)$-neighbours by R7(a), R9(a).  Thus,   $\mu^*(v)\geq  2\times\frac{1}{6}-\frac{1}{3}=0$ after $v$ sends $\frac{1}{3}$ to its incident $3$-face by R1.  \medskip

\textbf{(2.2).}  Let $m_3(v)=2$. Since $0\leq m_4(v) \leq 2$, we consider the following cases:

If $m_4(v)=0$, then  $v$ has  two $7^+$-neighbours different from $7(7)$- and $8(8)$-vertices by Lemma \ref{8lem:4-2-vertex-neighbours}(a). It follows that $v$ receives totally at least $2\times \frac{1}{6}$ from its $7(6^-)$- and $8(7^-)$-neighbours by R7(a), R9(b). In addition, $v$ is incident to two $5^+$-faces, and by R2, $v$ receives $\frac{1}{5}$ from each of those $5^+$-faces.  Thus,   $\mu^*(v)\geq  2\times\frac{1}{6}+2\times\frac{1}{5}-2\times\frac{1}{3}>0$ after  $v$ sends $\frac{1}{3}$ to each of its incident $3$-faces by R1. 

If $m_4(v)=1$, then  $v$ has either  two $8(7^-)$-neighbours or three $7^+$-neighbours different from $7(7)$- and $8(8)$-vertices by Lemma \ref{8lem:4-2-vertex-neighbours}(b). In each case, $v$ receives totally at least $\frac{1}{2}$ from its $7(6^-)$- and $8(7^-)$-neighbours by R7(a), R9(b). Also, $v$ receives $\frac{1}{5}$ from its incident $5^+$-face by R2. Thus,   $\mu^*(v)\geq  \frac{1}{2}+\frac{1}{5}-2\times\frac{1}{3}>0$ after  $v$ sends $\frac{1}{3}$ to each of its incident $3$-faces by R1. 

If $m_4(v)=2$, then $v$ has either  one  $8(7^-)$- and three $7(6^-)$-neighbours or  two  $8(7^-)$- and one $7(6^-)$-neighbours or three  $8(7^-)$-neighbours by Lemma \ref{8lem:4-2-vertex-neighbours}(c). In each case, $v$ receives totally at least $\frac{2}{3}=\min\{\frac{1}{4}+3\times\frac{1}{6}, 2\times\frac{1}{4}+\frac{1}{6}, 3\times\frac{1}{4} \}$ from its $7(6^-)$- and $8(7^-)$-neighbours by R7(a), R9(b). Thus,   $\mu^*(v)\geq  \frac{2}{3}-2\times\frac{1}{3}=0$ after  $v$ sends $\frac{1}{3}$ to each of its incident $3$-faces by R1.  \medskip

\textbf{(2.3).}  Let $m_3(v)=3$. Suppose first that $m_4(v)=0$. It then follows from Lemma \ref{8lem:4-3-vertex-neighbours}(a) that $v$ has either one $8(7^-)$- and three $7(6^-)$-neighbours  or two $8(7^-)$- and one $7(6^-)$-neighbours or two $8(7^-)$- and two $8(8)$-neighbours or three $8(7^-)$-neighbours. In each case, $v$ receives totally at least $\frac{5}{6}=\min\{\frac{1}{3}+3\times \frac{1}{6}, 2\times\frac{1}{3}+ \frac{1}{6},2\times\frac{1}{3}+2\times \frac{1}{9},3\times\frac{1}{3} \}$ from  its $7(6^-)$- and $8$-neighbours by R7(a), R9(c), R10(a). In addition, $v$ receives $\frac{1}{5}$ from its incident $5^+$-face by R2.  Thus,   $\mu^*(v)\geq  \frac{5}{6}+ \frac{1}{5} -3\times\frac{1}{3}>0$ after  $v$ sends $\frac{1}{3}$ to each of its incident $3$-faces by R1. Now, we suppose that  $m_4(v)=1$. By Lemma \ref{8lem:4-3-vertex-neighbours}(b), $v$ has either  two $8(7^-)$- and two $7(6^-)$-neighbours or three $8(7^-)$-neighbours. In both cases, $v$ receives totally at least $1=\min\{2\times\frac{1}{3}+2\times \frac{1}{6},3\times\frac{1}{3} \}$ from its $7(6^-)$- and $8(7^-)$-neighbours by R7(a), R9(c). Thus,  $\mu^*(v)\geq  1 -3\times\frac{1}{3}=0$ after  $v$ sends $\frac{1}{3}$ to each of its incident $3$-faces by R1. \medskip

\textbf{(2.4).}  Let $m_3(v)=4$. By Lemma \ref{8lem:4-4-vertex-neighbours}, 
$v$ has either   three $8(7^-)$- and one $7(5^-)$-neighbour or four $8(7^-)$-neighbours.  By applying R4 and R9(c), $v$ receives $\frac{1}{3}$ from each of its $7(5^-)$- and $8(7^-)$-neighbours.  Thus,   $\mu^*(v)\geq  4\times\frac{1}{3} -4\times\frac{1}{3}=0$ after  $v$ sends $\frac{1}{3}$ to each of its incident $3$-faces by R1.\medskip

\textbf{(3).} Let $k=5$. The initial charge of $v$ is $\mu(v)=d(v)-4=1$.  If $m_3(v)\leq 3$, then $\mu^*(v)\geq 1-3\times\frac{1}{3}= 0$ after  $v$ sends $\frac{1}{3}$ to each of its incident $3$-faces by R1. Therefore, we may assume that $4\leq m_3(v)\leq 5$. \medskip

\textbf{(3.1).}  Let $m_3(v)=4$. Suppose first that $m_4(v)=0$. By Lemma \ref{8lem:5-4-vertex}(a),  $v$ has two $6^+$-neighbours different from $6(6)$-vertex. 
It follows that $v$ receives  $\frac{1}{9}$ from each of its $6^+$-neighbours different from $6(6)$-vertex by R6-R10, and $\frac{1}{5}$ from its incident $5^+$-face by R2.  Thus,  $\mu^*(v)\geq  1+2\times\frac{1}{9}+\frac{1}{5} -4\times\frac{1}{3}>0$ after  $v$ sends $\frac{1}{3}$ to each of its incident $3$-faces by R1. 
Next we suppose that $m_4(v)=1$. By Lemma \ref{8lem:5-4-vertex}(b), $v$ has  three $6^+$-neighbours different from $6(6)$-vertex. So, $v$ receives  $\frac{1}{9}$ from each of its $6^+$-neighbours different from $6(6)$-vertex by R6-R10.  Thus,  $\mu^*(v)\geq  1+3\times\frac{1}{9} -4\times\frac{1}{3}=0$ after  $v$ sends $\frac{1}{3}$ to each of its incident $3$-faces by R1. \medskip

\textbf{(3.2).}  Let $m_3(v)=5$. We distinguish two subcases  according to the number of $4$-vertices adjacent to $v$ as follows. Recall that $0\leq n_4(v) \leq 1$ by Lemma \ref{8lem:5-vertex}(b). \medskip 

\textbf{(3.2.1).} Let $n_4(v)=0$.  If $n_5(v)=0$, then, by Lemma \ref{8lem:5-5-vertex}(b), $v$ has either three $8(7^-)$-neighbours or four $6^+$-neighbours different from $6(6)$-vertex, two of which are $7^+$-vertices different from $7(7)$-vertex. It follows from applying R6-R10 that $v$ receives totally at least $\frac{2}{3}=\min\{3\times \frac{2}{9},2\times \frac{1}{9}+2\times\frac{2}{9}\}$ from  its $6^+$-neighbours. Thus,   $\mu^*(v)\geq  1+\frac{2}{3} -5\times\frac{1}{3}=0$ after  $v$ sends $\frac{1}{3}$ to each of its incident $3$-faces by R1.

If $n_5(v)=1$, then, by Lemma \ref{8lem:5-5-vertex}(c), $v$ has either three $7^+$-neighbours different from $7(7)$-vertex or two $6(5^-)$- and  two $7^+$-neighbours different from $7(7)$-vertex. So, $v$ receives totally at least $\frac{2}{3}=\min\{3\times\frac{2}{9}, 2\times\frac{1}{9}+2\times\frac{2}{9} \}$ from  its $6^+$-neighbours by R6-R10. Thus,   $\mu^*(v)\geq  1+\frac{2}{3} -5\times\frac{1}{3}=0$ after  $v$ sends $\frac{1}{3}$ to each of its incident $3$-faces by R1.

If $n_5(v)=2$, then $v$ has three $7^+$-neighbours different from $7(7)$-vertex by Lemma \ref{8lem:5-5-vertex}(d). It follows that  $v$ receives $\frac{2}{9}$ from each of its $7^+$-neighbours different from $7(7)$-vertex by R7(c), R9(e), R10(b). Thus,   $\mu^*(v)\geq  1+3\times\frac{2}{9} -5\times\frac{1}{3}=0$ after  $v$ sends $\frac{1}{3}$ to each of its incident $3$-faces by R1.\medskip

\textbf{(3.2.2).} Let $n_4(v)=1$. By Lemma \ref{8lem:5-5-vertex}(a), $v$ has three $7^+$-neighbours different from $7(7)$-vertex. It follows that  $v$ receives $\frac{2}{9}$ from each of its $7^+$-neighbours different from $7(7)$-vertex by R7(c), R9(e), R10(b). So, $\mu^*(v)\geq  1+3\times\frac{2}{9} -5\times\frac{1}{3}=0$ after  $v$ sends $\frac{1}{3}$ to each of its incident $3$-faces by R1. \medskip

\textbf{(4).} Let $k=6$. The initial charge of $v$ is $\mu(v)=d(v)-4=2$. If $m_3(v)\leq 4$, 
then $\mu^*(v)\geq  2 -4\times \frac{1}{3}-6\times\frac{1}{9}=0$  after  $v$ sends $\frac{1}{3}$ to each of its incident $3$-faces by R1, and $\frac{1}{9}$ to each of its  $5(4^+)$-neighbours by R6.
On the other hand,  if $m_3(v)=6$, then $\mu^*(v)\geq  2 -6\times\frac{1}{3}=0$ after  $v$ sends $\frac{1}{3}$ to each of its incident $3$-faces by R1.
Therefore, we may assume that $m_3(v)=5$.  By Lemma \ref{8lem:6-5-vertex}, $v$ has at most four $5(4^+)$-neighbours. If  $v$ has at most three $5(4^+)$-neighbours, then $\mu^*(v)\geq  2 -5\times\frac{1}{3}-3\times\frac{1}{9}=0$ after  $v$ sends $\frac{1}{3}$ to each of its incident $3$-faces by R1, and $\frac{1}{9}$ to each of its $5(4^+)$-neighbours by R6. Next we assume that $v$ has exactly four $5(4^+)$-neighbours. 
If $m_4(v)=0$, then $v$ receives $\frac{1}{5}$ from its incident $5^+$-face by R2, and so $\mu^*(v)\geq  2+\frac{1}{5} -5\times \frac{1}{3}-4\times\frac{1}{9}>0$ after  $v$ sends $\frac{1}{3}$ to each of its incident $3$-faces by R1, and $\frac{1}{9}$ to each of its $5(4^+)$-neighbours by R6.
If $m_4(v)=1$, then $v$ has two $8(6^-)$-neighbours by Lemma \ref{8lem:6-5-vertex}, and so $v$ receives $\frac{1}{18}$ from each of its $8(6^-)$-neighbours by R5.  
Thus, $\mu^*(v)\geq  2+2\times\frac{1}{18} -5\times \frac{1}{3}-4\times\frac{1}{9}=0$ after  $v$ sends $\frac{1}{3}$ to each of its incident $3$-faces by R1, and $\frac{1}{9}$ to each of its $5(4^+)$-neighbours by R6. \medskip

\textbf{(5).} Let $k=7$. The initial charge of $v$ is $\mu(v)=d(v)-4=3$.  Notice first that if $m_3(v)\leq 3$, then we have $\mu^*(v)\geq  3 -3\times\frac{1}{3}-2\times\frac{1}{3}-5\times\frac{2}{9}>0$ after  $v$ sends $\frac{1}{3}$ to each of its incident $3$-faces by R1,  $\frac{1}{3}$ to each of its $4(4)$-neighbours by R4, and at most $\frac{2}{9}$ to each of its other neighbours by R7, where we note that $v$ can have at most two $4(4)$-neighbours due to $m_3(v)\leq 3$. Therefore, we may  assume that $4 \leq m_3(v) \leq 7$. \medskip

\textbf{(5.1).}  Let $m_3(v)=4$. Observe  that $v$ can have at most two neighbours consisting of  $4(4)$- or $5(5)$-vertices by Lemmas \ref{8lem:4-vertex-neighbours}(c) and \ref{8lem:5-vertex}(b). 
If $v$ has at most one $4(4)$-neighbour, then $\mu^*(v)\geq  3 -4\times\frac{1}{3}-\frac{1}{3}-6\times\frac{2}{9}=0$ after  $v$ sends $\frac{1}{3}$ to each of its incident $3$-faces by R1, $\frac{1}{3}$ to  its $4(4)$-neighbour by R4, and at most $\frac{2}{9}$ to each of its other neighbours by R7.
If $v$ has exactly two $4(4)$-neighbours, then $v$ is not adjacent to any $5(5)$-vertex as stated earlier, and so $\mu^*(v)\geq  3 -4\times\frac{1}{3}-2\times\frac{1}{3}-5\times\frac{1}{6}>0$ after  $v$ sends $\frac{1}{3}$ to each of its incident $3$-faces by R1, $\frac{1}{3}$ to each of its $4(4)$-neighbours by R4, and at most $\frac{1}{6}$ to each of its other neighbours by R7(a)-(b).\medskip

\textbf{(5.2).}  Let $m_3(v)=5$. By Lemma \ref{8lem:4-vertex-neighbours}(c), $v$ has at most two $4(4)$-neighbours. 
If $v$ has no $4(4)$-neighbour, then  $v$ is adjacent to at most three $5(5)$-neighbour by Lemma \ref{8lem:5-vertex}(b), and so $\mu^*(v)\geq  3 -5\times\frac{1}{3}-3\times\frac{2}{9}-4\times \frac{1}{6}=0$ after  $v$ sends $\frac{1}{3}$ to each of its incident $3$-faces by R1,  $\frac{2}{9}$ to each of its $5(5)$-neighbours by R7(c), and at most $\frac{1}{6}$ to each of its other neighbours by R7(a)-(b).
If $v$ has exactly one $4(4)$-neighbour, say $v_i$, then $v$ is adjacent to at most two $5(5)$-neighbours  by  Lemmas \ref{8lem:4-vertex-neighbours}(c) and \ref{8lem:5-vertex}(b). Moreover, the neighbours of $v_i$ belonging to $G[N(v)]$, say $v_j $ and $v_k$,  are neither $4$-vertices nor $5(4^+)$-vertices by Lemma \ref{8lem:4-vertex-neighbours}(c). Clearly, neither $v_j$ nor $v_k$ receives any charge from $v$.  Therefore,  $\mu^*(v)\geq  3 -5\times\frac{1}{3}-\frac{1}{3}-2\times\frac{2}{9}-2\times\frac{1}{6}>0$ after  $v$ sends $\frac{1}{3}$ to each of its incident $3$-faces by R1, $\frac{1}{3}$ to its $4(4)$-neighbour by R4, $\frac{2}{9}$ to each of its $5(5)$-neighbours by R7(c), and at most $\frac{1}{6}$ to each of its $4(1^+)$- and $5(4)$-neighbours different from $4(4)$-vertex by R7(a)-(b).
If $v$ has exactly two $4(4)$-neighbours, then $v$ is not adjacent to any $5(5)$-vertices by Lemma \ref{8lem:4-vertex-neighbours}(c). 
Moreover, similarly as above, $ v $ has at least three neighbours $ v_i $, $v_j $, and $ v_k $ that are neither $ 4 $-vertices nor $ 5(4^+) $-vertices by Lemma \ref{8lem:4-vertex-neighbours}(c), since $ v $ has two $ 4(4) $-neighbors. Clearly, none of $ v_i $, $ v_j $, or $ v_k $ receives any charge from \( v \).
Thus, $\mu^*(v)\geq  3 -5\times\frac{1}{3}-2\times\frac{1}{3}-2\times\frac{1}{6}>0$ after  $v$ sends $\frac{1}{3}$ to each of its incident $3$-faces by R1, $\frac{1}{3}$ to each of its $4(4)$-neighbours by R4, and at most $\frac{1}{6}$ to each of its $4(1^+)$- and $5(4)$-neighbours different from $4(4)$-vertex by R7(a)-(b).\medskip

\textbf{(5.3).}  Let $m_3(v)=6$. Notice first that if $v$ has at most four neighbours forming $4$- or $5$-vertices, then we have $\mu^*(v)\geq  3 -6\times\frac{1}{3}-4\times\frac{2}{9}>0$ after  $v$ sends $\frac{1}{3}$ to each of its incident $3$-faces by R1, and  at most $\frac{2}{9}$ to each of its $4(1)$-, $4(2)$-, $4(3)$-, and $5(4^+)$-neighbours by R7. Therefore we may further assume that $v$ has at least five neighbours forming $4$- or $5$-vertices.
% which clearly infers that $v$ has at least one support neighbour. Thus we have $d_2(v)\geq 23$ by Lemma \ref{8lem:7-vertex}. On the other hand, 
By Lemma \ref{8lem:4-vertex-neighbours}, we conclude that a $4(1^+)$-vertex  has at most one $4$-neighbour. In addition, a  $5(5)$-vertex has at most one neighbour consisting of  $4$- or $5(5)$-vertices by Lemma \ref{8lem:5-vertex}(b).  This implies that $v$ has at most four $5(5)$-neighbours; in particular, $v$ has at most five $4(1^+)$-neighbours.

Suppose first that $v$ has no $5(5)$-neighbours. 
If $v$ has at most four $4(1^+)$-neighbours, then $\mu^*(v)\geq  3 -6\times\frac{1}{3}-4\times\frac{1}{6}-3\times \frac{1}{9}=0$ after  $v$ sends $\frac{1}{3}$ to each of its incident $3$-faces by R1,  $\frac{1}{6}$ to each of its $4(1^+)$-neighbours by R7(a), and $\frac{1}{9}$ to each of its $5(4)$-neighbours by R7(b). If $v$ has five $4(1^+)$-neighbours, then $v$ cannot have any $5(4)$-neighbours by Lemma \ref{8lem:7-vertex}(b). So,  $\mu^*(v)\geq  3 -6\times\frac{1}{3}-5\times\frac{1}{6}>0$ after  $v$ sends $\frac{1}{3}$ to each of its incident $3$-faces by R1, and $\frac{1}{6}$ to each of its $4(1^+)$-neighbours by R7(a).

Next we suppose that $v$ has exactly one $5(5)$-neighbour. By Lemma \ref{8lem:7-vertex}(c), $v$ has at most four $4$-neighbours.
If $v$ has at most two $4$-neighbours, then $\mu^*(v)\geq  3 -6\times\frac{1}{3}-\frac{2}{9}-2\times\frac{1}{6}-4\times\frac{1}{9}=0$ after  $v$ sends $\frac{1}{3}$ to each of its incident $3$-faces by R1,   $\frac{2}{9}$ to its $5(5)$-neighbour by R7(c),  $\frac{1}{6}$ to each of its $4(1^+)$-neighbours by R7(a), and  $\frac{1}{9}$ to each of its $5(4)$-neighbours by R7(b).
If $v$ has at least three $4$-neighbours, then $v$ is adjacent to at most one $5(4)$-vertices by Lemma \ref{8lem:7-vertex}(c), and so $\mu^*(v)\geq  3 -6\times\frac{1}{3}-\frac{2}{9}-4\times\frac{1}{6}-\frac{1}{9}=0$ after  $v$ sends $\frac{1}{3}$ to each of its incident $3$-faces by R1, $\frac{2}{9}$ to its $5(5)$-neighbour by R7(c), $\frac{1}{6}$ to each of its $4(1^+)$-neighbours by R7(a), and   $\frac{1}{9}$ to each of its $5(4)$-neighbours by R7(b).

Suppose now that $v$ has exactly two $5(5)$-neighbours.  Then, by Lemma \ref{8lem:7-vertex}(d), $v$ has at most three neighbours consisting of  $4$- or $5(4)$-vertices, and so $\mu^*(v)\geq  3 -6\times\frac{1}{3}-2\times\frac{2}{9}-3\times\frac{1}{6}>0$ after  $v$ sends $\frac{1}{3}$ to each of its incident $3$-faces by R1, $\frac{2}{9}$ to each of its $5(5)$-neighbours by R7(c), and  at most $\frac{1}{6}$ to each of its $4(1^+)$ and $5(4)$-neighbours by R7(a),(b).

Suppose that $v$ has exactly three $5(5)$-neighbours.  It follows that $v$ has at most one $4$-neighbour by Lemma \ref{8lem:7-vertex}(e). By the same reason, $v$ has at most two neighbours consisting of  $4$- or $5(4)$-vertices. Thus, we have  $\mu^*(v)\geq  3 -6\times\frac{1}{3}-3\times\frac{2}{9}-\frac{1}{6}-\frac{1}{9}>0$ after  $v$ sends $\frac{1}{3}$ to each of its incident $3$-faces by R1, $\frac{2}{9}$ to its $5(5)$-neighbour by R7(c), $\frac{1}{6}$ to each of its $4(1^+)$-neighbours by R7(a), and   $\frac{1}{9}$ to each of its $5(4)$-neighbours by R7(b).

Finally, suppose that $v$ has four $5(5)$-neighbours.  Then $v$ has neither $4$- nor $5(4)$-neighbours by Lemma \ref{8lem:5-vertex}(a),(b), and so $\mu^*(v)\geq  3 -6\times\frac{1}{3}-4\times\frac{2}{9}>0$ after  $v$ sends $\frac{1}{3}$ to each of its incident $3$-faces by R1,  $\frac{2}{9}$ to each of its $5(5)$-neighbours by R7(c). \medskip

\textbf{(5.4).}  Let $m_3(v)=7$. By Lemma \ref{8lem:7-vertex}(f), $v$ has at most five $5(4^+)$-neighbours, and so  $\mu^*(v)\geq  3-7\times \frac{1}{3}-5\times \frac{1}{9}>0$  after  $v$ sends $\frac{1}{3}$ to each of its incident $3$-face by R1, $\frac{1}{9}$ to each of its $5(4^+)$-neighbour by R8.\medskip

\textbf{(6).} Let $k=8$. The initial charge of $v$ is $\mu(v)=d(v)-4=4$.  Notice first that if $m_3(v)\leq 4$, then we have $\mu^*(v)\geq  4 -4\times\frac{1}{3}-8\times\frac{1}{3}=0$ after  $v$ sends $\frac{1}{3}$ to each of its incident $3$-faces by R1, and at most $\frac{1}{3}$ to each of its neighbours by R3, R5, R9. Therefore, we may  assume that $5 \leq m_3(v) \leq 8$. \medskip

\textbf{(6.1).}  Let $m_3(v)=5$. By Lemma \ref{8lem:4-vertex-neighbours}(c), a $4(3^+)$-vertex cannot be adjacent to any $4$-vertex. It follows that  $v$ has at most four $4(3^+)$-neighbours. Then  $\mu^*(v)\geq  4 -5\times\frac{1}{3}-4\times\frac{1}{3}-4\times\frac{1}{4}=0$ after  $v$ sends $\frac{1}{3}$ to each of its incident $3$-faces by R1, $\frac{1}{3}$ to each of its $4(3^+)$-neighbours by R9(c), and at most $\frac{1}{4}$ to each of its other neighbours by R3, R5, R9. \medskip

%If $v$ has five $4(3^+)$-neighbours, then $v$ cannot have any $3$-, $4(2^-)$- or $5(4^+)$-neighbour by Lemmas \ref{8lem:3-vertex}(a) and \ref{8lem:4-vertex-neighbours}(c). Thus $\mu^*(v)\geq  4 -5\times\frac{1}{3}-5\times\frac{1}{3}-3\times\frac{1}{18}>0$ after  $v$ sends $\frac{1}{3}$ to each of its incident $3$-faces by R1, $\frac{1}{3}$ to each of its $4(3^+)$-neighbours by R9(c), and $\frac{1}{18}$ to each of its $6(5)$-neighbours by R5.  \medskip

\textbf{(6.2).} Let $m_3(v)=6$. By Lemma \ref{8lem:4-vertex-neighbours}(c), $v$ has at most four $4(3^+)$-neighbours. 
%If $v$ has five such neighbours, then $v$ is not adjacent to any $3$-, $4(2^-)$- or $5(4^+)$-vertex by Lemmas \ref{8lem:3-vertex}(a) and \ref{8lem:4-vertex-neighbours}(c). Thus $\mu^*(v)\geq  4-6\times \frac{1}{3}-5\times \frac{1}{3}-3\times\frac{1}{18}>0$ after  $v$ sends $\frac{1}{3}$ to each of its incident $3$-faces by R1, $\frac{1}{3}$ to each of its $4(3^+)$-neighbours by  R9(c), and $\frac{1}{18}$ to each of its $6(5)$-neighbours by R5. We may further assume that $v$ has at most four $4(3^+)$-neighbours.
Suppose first that $v$ has no $4(3^+)$-neighbour. Then $\mu^*(v)\geq  4-6\times \frac{1}{3}-8\times\frac{1}{4}=0$ after  $v$ sends $\frac{1}{3}$ to each of its incident $3$-faces by R1, and at most $\frac{1}{4}$ to each of  its  neighbours by R3, R5, R9. 

Next suppose that $v$ has exactly one $4(3^+)$-neighbour. It follows from  Lemmas \ref{8lem:3-vertex}(a) and \ref{8lem:4-vertex-neighbours}(c) that $v$ has at most six neighbours consisting of  $3$-, $4(2^-)$- or $5(4^+)$-vertices. Thus $\mu^*(v)\geq  4-6\times \frac{1}{3}- \frac{1}{3}-6\times\frac{1}{4}-\frac{1}{18}>0$ after  $v$ sends $\frac{1}{3}$ to each of its incident $3$-faces by R1, $\frac{1}{3}$ to its $4(3^+)$-neighbour by  R9(c),  at most $\frac{1}{4}$ to each of  its $3$-, $4(1)$-, $4(2)$- and $5(4^+)$-neighbours by R3, R9, and $\frac{1}{18}$ to each of its $6(5)$-neighbours by R5. 

Suppose now that $v$ has exactly two $4(3^+)$-neighbours. In such a case, $v$ has at most five neighbours consisting of  $3$-, $4(2^-)$- or $5(4^+)$-vertices by Lemmas \ref{8lem:3-vertex}(a) and \ref{8lem:4-vertex-neighbours}(c). Thus $\mu^*(v)\geq  4-6\times \frac{1}{3}-2\times \frac{1}{3}-5\times\frac{1}{4}-\frac{1}{18}>0$ after  $v$ sends $\frac{1}{3}$ to each of its incident $3$-faces by R1, $\frac{1}{3}$ to each of its $4(3^+)$-neighbours by  R9(c), at most $\frac{1}{4}$ to each of  its $3$-, $4(1)$-, $4(2)$- and $5(4^+)$-neighbours by R3, R9, and $\frac{1}{18}$ to each of its $6(5)$-neighbours by R5.

Suppose that $v$ has exactly three $4(3^+)$-neighbours. In such a case, $v$ can have at most three neighbours consisting of  $3$-, $4(2^-)$- or $5(4^+)$-vertices by Lemmas \ref{8lem:3-vertex}(a) and \ref{8lem:4-vertex-neighbours}(c). Thus $\mu^*(v)\geq  4-6\times \frac{1}{3}-3\times \frac{1}{3}-3\times\frac{1}{4}-2\times\frac{1}{18}>0$ after  $v$ sends $\frac{1}{3}$ to each of its incident $3$-faces by R1, $\frac{1}{3}$ to each of its $4(3^+)$-neighbours by  R9(c),  at most $\frac{1}{4}$ to each of  its $3$-, $4(1)$-, $4(2)$- and $5(4^+)$-neighbours by R3, R9, and $\frac{1}{18}$ to each of its $6(5)$-neighbours by R5. 

Finally suppose that $v$ has exactly four $4(3^+)$-neighbours. In such a case,  $v$ can have at most one neighbour consisting of  $3$-, $4(2^-)$- or $5(4^+)$-vertices by Lemmas \ref{8lem:3-vertex}(a) and \ref{8lem:4-vertex-neighbours}(c). Thus $\mu^*(v)\geq  4-6\times \frac{1}{3}-4\times \frac{1}{3}-\frac{1}{4}-3\times \frac{1}{18}>0$ after  $v$ sends $\frac{1}{3}$ to each of its incident $3$-faces by R1, $\frac{1}{3}$ to each of its $4(3^+)$-neighbours by  R9(c), and at most $\frac{1}{4}$ to its $3$-, $4(1)$-, $4(2)$- and $5(4^+)$-neighbours by R3, R9, and $\frac{1}{18}$ to each of its $6(5)$-neighbours by R5. \medskip

\textbf{(6.3).} Let $m_3(v)=7$. Let $f_i=v_ivv_{i+1}$ for $i\in [7]$. Clearly, $v$ has no $3$-neighbour by Lemma \ref{8lem:3-vertex}.

Let us first claim that $v$ has a support neighbour. By contradiction, assume that $v$ does not have. Then, $v_2,v_3,\ldots,v_7$ are $5^+$-vertices.
Recall that a $5(5)$-vertex has at most one $5(5)$-neighbour by Lemma \ref{8lem:5-vertex}(b). Thus we infer that $v$ has at most four $5(5)$-neighbour. If one of $v_1,v_8$ is a $5^+$-vertex, then $\mu^*(v)\geq  4-7\times \frac{1}{3}-\frac{1}{3}-4\times\frac{2}{9}-3\times \frac{1}{9}>0$ after  $v$ sends $\frac{1}{3}$ to each of its incident $3$-faces by R1,  at most $\frac{1}{3}$ to each of  its $4$-neighbours by  R9(a)-(c), $\frac{2}{9}$ to each of  its  $5(5)$-neighbours by R9(e),  $\frac{1}{9}$ to each of  its  $5(4)$-neighbours by R9(d). Thus, we assume that $v_1$ and $v_8$ are $4$-vertices.
By Lemma \ref{8lem:4-vertex-neighbours}, this infer that $v$ has at most three $5(5)$-neighbours.
%If $v$ has exactly four $5(5)$-neighbours, then $v$ would have at least two $6^+$-neighbours by Lemmas \ref{8lem:3-vertex} and \ref{8lem:5-vertex}(a). In such a case, we have $\mu^*(v)\geq  4-7\times \frac{1}{3}-2\times\frac{1}{3}-4\times\frac{2}{9}>0$ after  $v$ sends $\frac{1}{3}$ to each of its incident $3$-faces by R1,  at most $\frac{1}{3}$ to each of  its $4$-neighbours by  R9(a)-(c), $\frac{2}{9}$ to each of  its  $5(5)$-neighbours by R9(e). Assume next that $v$ has at most three $5(5)$-neighbours. 
It then follows that  $\mu^*(v)\geq  4-7\times \frac{1}{3}-2\times\frac{1}{3}-3\times\frac{1}{9}-3\times\frac{2}{9}=0$ after  $v$ sends $\frac{1}{3}$ to each of its incident $3$-faces by R1,  at most $\frac{1}{3}$ to each of  its $4$-neighbours by  R9(a)-(c), $\frac{1}{9}$ to each of  its  $5(4)$-neighbours by R9(d), $\frac{2}{9}$ to each of  its  $5(5)$-neighbours by R9(e). Consequently, we deduce that $v$ has a support neighbour, i.e., one of $v_2,v_3,\ldots,v_7$ is a $4^-$-vertex.
%This clearly infers that $d_2(v)\geq 23$ by Lemma \ref{8lem:7-vertex}.

%By Lemma \ref{8lem:7-vertex}(g),  $v$ has at most six $4$-neighbours.

By Lemma \ref{8lem:4-vertex-neighbours}(c), a $4(3^+)$-vertex has no $4$- and $5(4^+)$-neighbours. So, we deduce that $v$ has at most four $4(3^+)$-neighbours. Indeed, if $v$ has  four such neighbours, then $v$ is not adjacent to any $4(2^-)$- or $5(4^+)$-vertices by Lemma \ref{8lem:4-vertex-neighbours}(c). Thus $\mu^*(v)\geq  4-7\times \frac{1}{3}-4\times \frac{1}{3}>0$ after  $v$ sends $\frac{1}{3}$ to each of its incident $3$-faces by R1, $\frac{1}{3}$ to each of its $4(3^+)$-neighbours by  R9(c). 
We may further assume that $v$ has at most three $4(3^+)$-neighbours. 

Suppose first that $v$ has no $4(3^+)$-neighbour. Note that $v$ has at most six neighbours consisting of  $4(2^-)$- or $5(5)$-vertices by Lemma \ref{8lem:7-vertex}(g). Moreover, if $v$ has six such  neighbours, then $v$ cannot be adjacent to any $5(4)$-vertex by Lemmas \ref{8lem:4-vertex-neighbours} and \ref{8lem:5-vertex}, and so $\mu^*(v)\geq  4-7\times \frac{1}{3}-6\times\frac{1}{4}>0$ after  $v$ sends $\frac{1}{3}$ to each of its incident $3$-faces by R1,  at most $\frac{1}{4}$ to each of  its $4(1)$-, $4(2)$- and $5(5)$-neighbours by R9(a), (b), (e).
If $v$ has at most five neighbours consisting of  $4(2^-)$- or $5(5)$-vertices, then $\mu^*(v)\geq  4-7\times \frac{1}{3}-5\times\frac{1}{4}-3\times\frac{1}{9}>0$ after  $v$ sends $\frac{1}{3}$ to each of its incident $3$-faces by R1,  at most $\frac{1}{4}$ to each of  its $4(1)$-, $4(2)$- and $5(5)$-neighbours by R9(a),(b),(e), and $\frac{1}{9}$ to each of  its  $5(4)$-neighbours by R9(d).

Next we suppose that $v$ has exactly one $4(3^+)$-neighbour. By Lemma \ref{8lem:4-vertex-neighbours}(c), $v$ has at least one neighbour different from $4$- and $5(4^+)$-vertices, which does not receive any charge from $v$. On the other hand, a $4(t)$-vertex with $1\leq t \leq 2$ has at most one neighbour consisting of  $4$- or $5(5)$-vertices by Lemma \ref{8lem:4-vertex-neighbours}(a),(b).  It then follows from Lemma \ref{8lem:4-vertex-neighbours}(c) and Lemma \ref{8lem:5-vertex}(b) that  $v$ has at most four neighbours consisting of  $4(1)$-, $4(2)$- or $5(5)$-vertices. 
Thus  $\mu^*(v)\geq  4-7\times \frac{1}{3}-\frac{1}{3}-4\times\frac{1}{4}-2\times\frac{1}{9}>0$ after  $v$ sends $\frac{1}{3}$ to each of its incident $3$-faces by R1, $\frac{1}{3}$ to each of its $4(3^+)$-neighbours by R9(c),  at most $\frac{1}{4}$ to each of  its $4(1)$-, $4(2)$- and $5(5)$-neighbours by R9(a),(b),(e), and $\frac{1}{9}$ to each of  its  $5(4)$-neighbours by R9(d).

Suppose that $v$ has exactly two $4(3^+)$-neighbours.  By Lemma \ref{8lem:4-vertex-neighbours}(c), a $4(3^+)$-vertex has neither $4$-neighbour nor $5(4^+)$-neighbour. This means that $v$ has at least two neighbours different from $4$- and $5(4^+)$-vertices, which do not receive any charge from $v$. 
Thus $\mu^*(v)\geq  4-7\times \frac{1}{3}-2\times \frac{1}{3}-4\times\frac{1}{4}=0$ after  $v$ sends $\frac{1}{3}$ to each of its incident $3$-faces by R1, $\frac{1}{3}$ to each of its $4(3^+)$-neighbours by R9(c), and  at most $\frac{1}{4}$ to each of  its $4(1)$-, $4(2)$- and $5(4^+)$-neighbours by R9(a),(b),(d),(e).

Finally suppose that $v$ has exactly three $4(3^+)$-neighbours. It follows  from Lemma \ref{8lem:4-vertex-neighbours}(c) that  $v$ has at least three neighbours different from $4$- and $5(4^+)$-vertices,  which do not receive any charge from $v$.  
Thus $\mu^*(v)\geq  4-7\times \frac{1}{3}-3\times \frac{1}{3}-2\times\frac{1}{4}>0$ after  $v$ sends $\frac{1}{3}$ to each of its incident $3$-faces by R1, $\frac{1}{3}$ to each of its $4(3^+)$-neighbours by R9(c),  at most $\frac{1}{4}$ to each of  its $4(1)$-, $4(2)$- and $5(4^+)$-neighbours by R9(a),(b),(d),(e).\medskip

\textbf{(6.4).}  Let $m_3(v)=8$. By Lemma \ref{8lem:5-vertex}(b),  $v$ has at most five $5(5)$-neighbours. 
If $v$ has at most four such neighbours, then  $\mu^*(v)\geq  4 -8\times\frac{1}{3}-4\times\frac{2}{9}-4\times\frac{1}{9}=0$ after  $v$ sends $\frac{1}{3}$ to each of its incident $3$-faces by R1, $\frac{2}{9}$ to each of its $5(5)$-neighbours by R10(b), and  at most $\frac{1}{9}$ to each of its $4(3)$- and $5(4)$-neighbours by R10(a). 
If $v$ has five $5(5)$-neighbours, then $v$ has neither $4$-neighbour nor $5(4)$-neighbour by Lemma \ref{8lem:5-vertex}(a). Thus,  $\mu^*(v)\geq  4 -8\times\frac{1}{3}-5\times\frac{2}{9}>0$ after  $v$ sends $\frac{1}{3}$ to each of its incident $3$-faces by R1, and $\frac{2}{9}$ to each of its $5(5)$-neighbours by R10(b).

\section*{Acknowledgement}

The author extends their gratitude to Kengo Aoki for pointing out a significant error in a previous proof of Lemma \ref{6lem:min-deg-4}.

\section*{Declarations}

\textbf{Conflict of interest} The author has no conflicts of interest to declare that are relevant to the content of this article. \medskip

\textbf{Availability of data and material} Data sharing not applicable to this article as no datasets were generated or analysed during the current study. \medskip

%%%%%%%%%%%%%%%%%%%%%%%%%%%%%%%%%%%%%%%%%%%%%%%%%%%%%%%%%%%%%%%%%%%%%%%%%%%%%%%

%%%%%%%%%%%%%%%%%%%%%%%%%%%%%%%%%%%%%%%%%%%%%%%%%%%%%%%%%%%%%%%%%%%%%%%%%%
%%%%%%%%%%%%%%%%%%%%%%%%%%%%%%%%%%%%%%%%%%%%%%%%%%%%%%%%%%%%%%%%%%%%%%%%%%%%%%%%%%%

\end{document}